\documentclass[a4paper, 12pt]{amsart}
\usepackage{amsmath,mathrsfs,amsthm}
\usepackage{txfonts,rotating,tikz}
\usetikzlibrary{intersections,shapes.arrows,calc}

\usepackage[all,cmtip]{xy}
\usepackage{amssymb}
\usepackage{amsfonts}
\usepackage{amscd}
\usepackage{color} 
\usepackage[utf8]{inputenc}
\usepackage{hyperref}
\usepackage{fullpage}
\usepackage{epic}
\usepackage{mathtools}
\usepackage{tikz}
\usepackage{tikz-cd}
\usetikzlibrary{chains}
\usepackage[thinlines]{easytable}
\usepackage{array}
\usepackage{ctable}
\usepackage{pbox}
\usepackage[usestackEOL]{stackengine}
\usepackage[sort, numbers]{natbib}
\usepackage{float}
\usepackage{relsize}
\usepackage{extarrows}
\usepackage{enumitem}
\usepackage{caption}
\usepackage{subcaption}
\usepackage[capitalise]{cleveref}
\tikzcdset{scale cd/.style={every label/.append style={scale=#1},
 cells={nodes={scale=#1}}}}
\newcolumntype{?}{!{\vrule width 1pt}}

\newtheorem{thm}{Theorem}[section]
\newtheorem{prop}[thm]{Proposition}
\newtheorem{lem}[thm]{Lemma}
\newtheorem{cor}[thm]{Corollary}
 
\theoremstyle{definition}
\newtheorem{definition}[thm]{Definition}

\theoremstyle{remark}
\newtheorem{remark}[thm]{Remark}

\newcommand{\mb}{\mathbb}
\newcommand{\mc}{\mathcal}
\newcommand{\mf}{\mathfrak}
\newcommand{\mr}{\mathrm}
\newcommand{\ms}{\mathscr}
\newcommand{\A}{\mathbb{A}}

\newcommand{\C}{\mathbb{C}}

\newcommand{\G}{\mathcal{G}}
\newcommand{\Gr}{\mathrm{Gr}}
\newcommand{\Fl}{\mathrm{Fl}}
\newcommand{\Grb}{\overline{\mathrm{Gr}}}
\newcommand{\m}{\mathrm{m}}
\newcommand{\rk}{\mathrm{k}}

\newcommand{\st}{\,|\,}

\newcommand{\U}{\mathrm{U}}

\newcommand{\arxiv}[1]{\href{https://arxiv.org/abs/#1}{\texttt{arXiv:#1}}}

\begin{document}
\title{A refinement of the coherence conjecture of Pappas and Rapoport}

\author{Jiuzu Hong}
\address[J.\,Hong]{Department of Mathematics, University of North Carolina at Chapel Hill, Chapel Hill, NC 27599-3250, U.S.A.}
\email{jiuzu@email.unc.edu}

\author{Huanhuan Yu}
\address[H.Yu]{Department of Mathematics, Jiangxi University of Finance and Economics, Nanchang, 330032, China}
\email{yuhuanhuan@jxufe.edu.cn}

\begin{abstract}
The coherence conjecture of Pappas and Rapoport, proved by Zhu, asserts the equality of dimensions for the global sections of a line bundle over
a spherical Schubert variety in the affine Grassmannian and those of another line bundle over a certain union of Schubert varieties in a partial affine flag variety.

We refine this equality of dimensions to an isomorphism of representations. The comparison is established by introducing a parahoric Bruhat-Tits group scheme $\mathcal{G}$ over the affine line, ramified at $0$. We further strengthen this comparison by  equipping any line bundle on the global Schubert variety of $\mathcal{G}$ with a unique equivariant structure under the global jet group scheme. As an application, we obtain new relations among affine Demazure modules.

\end{abstract}

\maketitle
\tableofcontents

\section{Introduction}

The coherence conjecture of Pappas and Rapoport \cite{PR08} arose in the study of local models of Shimura varieties and was proved by Zhu \cite{Zhu14} using the geometry of Beilinson–Drinfeld Grassmannians. Its formulation is numerical: it compares the dimensions of spaces of global sections on two different Schubert-type varieties.  In this paper, we shall refine Zhu's theorem by upgrading the dimension equality to an isomorphism of representations, which will be further enhanced by equipping equivariant structures on the line bundles over ramified global affine Grassmannians.  We will also discuss its applications to affine Demazure modules.


Let $G$ be a simple algebraic group over an algebraically closed field $\rk$ of characteristic $p$. Let $\Gr_G$ denote the affine Grassmannian of $G$. Let $L$ denote the line bundle on $\Gr_G$ of central charge 1. For each dominant coweight $\mu$ of $G$,  let $\overline{\Gr}_{G,\mu}$ denote the associated spherical Schubert variety in $\Gr_G$. Let $G((z))$ denote the loop group. Let $\mathcal{P}_\vartheta$ be the parahoric subgroup in $G((z))$ associated to a rational point $\vartheta$ in the fundamental alcove of $G((z))$. Let $Y$ be the set of vertices of the minimal facet containing $\vartheta$. 
We then attach a partial affine flag variety $\Fl_Y=G((z)) / \mathcal{P}_\vartheta$. For any dominant coweight $\mu$ of $G$, 
we can define a subvariety $\mc{A}_Y(\mu)$ in $\Fl_Y$, 
\begin{equation}
\label{eq_union_1}
\mc{A}_Y(\mu):= \bigcup_{w\in W} \overline{\Fl}_{Y,w(\mu)},  \end{equation}
where $W$ is the Weyl group of $G$ and $\overline{\Fl}_{Y,w(\mu)}$ is the Iwahori Schubert variety associated to coweight $w(\mu)$ as an element in the extended affine Weyl group of $W$. 
The following is the coherence conjecture of Pappas and Rapoport \cite{PR08}, which was extended and proved by Zhu \cite{Zhu14}. 
\begin{thm}
\label{thm 1.1}
For any ample line bundle $\mathscr{L}$ on $\Fl_Y$ of central charge $c$, we have the following  equality
\[ \dim H^0(\overline{\Gr}_{G,\mu}, L^c  )= \dim H^0( \mc{A}_Y({\mu}), \mathscr{L}) ,  \]
where $H^0(\cdot)$ denotes the space of global sections of line bundles. 
\end{thm}

In fact, Theorem \ref{thm 1.1} also holds in the twisted case: the left-hand side is unchanged while $\Fl_Y$ on the right-hand side is a partial affine flag variety of twisted type. For simplicity of illustration, in this introduction we only focus on the untwisted situation and we also assume that $G$ is simply-connected.  We emphasize that all of our results presented in the introduction are proved in full generality.

In his proof of Theorem \ref{thm 1.1}, Zhu considered a parahoric Bruhat-Tits group scheme $\mathcal{G}$ over $\mathbb{A}^1$ that is ramified at $0$. The main geometric result proved by Zhu is that  
the global Schubert variety $\overline{\Gr}^\mu_{\mathcal{G},\mathbb{A}^1}$ of $\mathcal{G}$, which is flat over $\mathbb{A}^1$, has the generic fiber $\overline{\Gr}_{G,\mu}$ and the special fiber $\mc{A}_Y(\mu) $ over $0$. Furthermore, Zhu constructed a line bundle on the global affine Grassmannian $\Gr_\G$ which specializes to powers of $L^c$ and $\ms{L}$ at generic and special fibers. Using a Frobenius-splitting argument, he derived the equality in Theorem  \ref{thm 1.1} comparing the global sections of $L^c$ and $\ms{L}$.

\vspace{0.1cm}
Every ample line bundle $\ms{L}$ on $\Fl_Y$ is of the form $\ms{L}(\Lambda)$ for an affine dominant weight $\Lambda=\sum_{i\in Y}n_i\Lambda_i$ with $n_i>0$. Let $\lambda$ be the restriction of $\Lambda$ to $T$, and it is a dominant weight of $G$. Let $G_\vartheta$ be the centralizer of a multiple of $\vartheta$ in $G$, defined in \eqref{eq_Glambda}. Equivalently, $G_\vartheta$ is the maximal Levi subgroup of $G$ containing $T$, such that the character $\lambda: T\to \mathbb{G}_m$ can be extended to $G_\vartheta$.  
The following is a special case of Theorem \ref{thm_mainthm} in Section \ref{3208530}.
\begin{thm}\label{thm1.2}
Let $\ms{L}(\Lambda)$ be the ample line bundle on $\Fl_Y$. When ${\rm char}(\rk)=0$, there is an isomorphism of representations of $G_\vartheta$:
\[  H^0(\overline{\Gr}_{G,\mu}, L^c  )\simeq H^0( \mc{A}_Y (\mu), \ms{L}(\Lambda))\otimes \rk_\lambda,  \]
where the action of $G_\vartheta$ on the left side (resp. right side) arises from the unique $G(\mc{O})$-equivariant structure on $L^c$ (resp. $\ms{L}(\Lambda)$). 

In general, this is an isomorphism of $T$-representations,  when ${\rm char}(\rk)$ does not divide $|\pi_1(G)|$ and also an integer $a_Y$ defined in \eqref{index_aY}.
\end{thm}

In the above statement, we regard $G_\vartheta$ as a subgroup of $G(\mathcal{O})$, and the difference of $\rk_\lambda$ arises from comparing equivariant structures of $G(\mathcal{O})$ and $\mathcal{P}_\vartheta$ on $\mathscr{L}(\Lambda)$, cf.\, Proposition \ref{thm_Gtau}. 
We now explain the strategy of the proof of Theorem \ref{thm_mainthm} (the more general version of Theorem \ref{thm1.2}). In Section \ref{sec_equivariantbundle}, we use the method of Faltings \cite{Faltings03} to construct a $\mc{P}_\vartheta^-$-equivariant structure on any line bundle $\ms{L}(\Lambda)$ on $\Fl_Y$, where $\mc{P}_\vartheta^-$ denotes the opposite parahoric subgroup in $G((z))$ relative to $\mc{P}_\vartheta$, see Proposition \ref{thm_lgequivarian}. This descends to a line bundle $\mc{L}_{{\rm Bun}_{\mc{G}}}$ on the moduli stack ${\rm Bun}_{\mc{G}}$ of $\mc{G}$-bundles over $\mb{P}^1$, where $\mc{G}$ is the parahoric Bruhat-Tits group scheme over $\mb{P}^1$ with parahoric subgroup $\mathcal{P}_{\vartheta}$ at $0$ and the parahoric subgroup $\mc{P}_{-\vartheta}$ at $\infty$, cf.\,Proposition \ref{45393240}.
 This gives rise to a line bundle $\mc{L}$ on the global affine Grassmannian $\Gr_{\mathcal{G}}$ of $\mathcal{G}$ via the forgetful map $\Gr_{\mc{G}}\to {\rm Bun}_{\mc{G}}$.  The line bundle $\mc{L}$
 specializes to the line bundle $\ms{L}(\Lambda)$ on $\Fl_Y$ at $0$, and specializes to $L^c$ on $\Gr_G$ at a generic  point in $\mb{P}^1$.
In Theorem \ref{431424}, we equip the line bundle $\mc{L}$ with an $L^-\mc{G}$-equivariant structure, where $L^-\mc{G}$ is the global negative loop group over $\mb{P}^1$ with the generic fiber $G[z^{-1}]$ and the special fiber $\mc{P}^-_\vartheta$ at $0$. Furthermore, we present $G_\vartheta\times \mathbb{P}^1$ as a closed subgroup scheme of $L^-\mc{G}$ over $\mb{P}^1$. 
\vspace{0.1cm}

 To prove Theorem \ref{thm_mainthm} (or Theorem \ref{thm1.2}), we also make crucial use of an equivariant realization of $\mc{G}$ in Theorem \ref{paraohoric_thm} of Section \ref{Section_Parahoric}, where we extend the equivariant construction of parahoric group schemes from the simply-connected case in \cite{DH} to general simple algebraic groups. In Lemma \ref{prop_goodprime}, we  show how to choose a rational point in the facet to avoid the characteristic constraints. 
 In Proposition \ref{eq_levi}, 
we describe the canonical Levi subgroup and the subgroup $G_\vartheta$ in any parahoric subgroup. These results will play a crucial role in Section \ref{3208530}.
\vspace{0.1cm}

 As explained above,  Theorem \ref{thm1.2} is a consequence of the $L^-\mc{G}$-equivariant structure on the line bundle $\mc{L}$. We now restrict $\mc{G}$ to $\A^1$, and let $L^+\mc{G}$ be the global jet group scheme of $\mc{G}$.  In fact, Theorem \ref{thm1.2} is also a consequence of an $L^+\mc{G}$-equivariant structure on $\mc{L}$, which will be proved in Section \ref{32895693}.  In Section \ref{sect_unique_ext},  we discuss a general setting of central extensions of affine group schemes over a smooth curve $\Sigma$. We show in Proposition \ref{prop_ext_split} that,  if a central extension of a fiberwise smooth affine group scheme $H$ over $\Sigma$ splits over an open subset of $\Sigma$, then this splitting can automatically extend to $\Sigma$. We also show in Proposition \ref{prop_exact_fppf} that a line bundle on a fiberwise reduced and connected projective flat scheme over $\Sigma$ with an action of a smooth group scheme, gives rise to a central extension of this group scheme in fppf topology.  These results can be applied to obtain the following special case of Theorem \ref{569235}. 
\begin{thm}
\label{thm1.4}
Any line bundle $\mc{L}$ on $\Gr_{\G}$ and its restriction to the global Schubert variety $\Grb^\mu_{\mc{G}, \mathbb{A}^1}$ have a unique $L^+\mc{G}$-equivariant structure.  
\end{thm}

This theorem can be viewed as an enhancement of Theorem \ref{thm1.2} or, more precisely, Proposition \ref{equivariant}, as we can regard $G_\vartheta\times \mb{A}^1$ as a closed subgroup scheme of $L^+\mc{G}$. When $\mc{G}$ is special parahoric at $0\in \A^1$ and ${\rm char} (\rk)=0$, this result was essentially established in a previous work \cite{HY} of the authors. We expect that Theorem \ref{thm1.4} will have  applications in the equivariant K-theory of partial affine flag varieties, which will be the focus of a future research project.

In Section \ref{4924695}, we classify all line bundles on the global affine Grassmannian $\Gr_{\mc{G}}$, where $\mc{G}$ is the parahoric Bruhat-Tits group scheme over $\mathbb{A}^1$ ramified at $0$. In fact, any line bundle $\mc{L}$ on $\Gr_{\G}$ is uniquely determined by its restriction to the fiber $\Gr_{\mc{G}, 0}\simeq\Fl_Y$ over $0$, and thus it is parametrized by an affine dominant weight $\Lambda$, see Proposition \ref{416947}. 
\vspace{0.1cm}

In Section \ref{3985299}, we explore a consequence of Theorem \ref{thm1.2} in the context of affine Demazure modules. By a theorem of Kumar and Mathieu, the dual of $H^0(\overline{\Gr}_{G,\mu}, L^c)$ is an affine Demazure module $D(c,\mu)$, see Theorem \ref{329825}. Furthermore, following a result of Littelmann \cite{Li98}, the dual of $H^0( \mc{A}_Y({\mu}), \mathscr{L}(\Lambda))$ is the sum $\sum_{w\in W} D(\Lambda,w(\mu))$ of Demazure modules in the integrable highest weight module $V(\Lambda)$. Therefore, Theorem \ref{thm1.2} can be reformulated purely algebraically as the following statement, which is the untwisted case of Theorem \ref{589379}.
\begin{thm}
\label{thm1.3}
For any dominant coweight $\mu\in X_*(T)$, there is an isomorphism of $\mf{g}_\vartheta$-modules
$$D(c,\mu)\otimes \rk_\lambda\simeq \sum_{w\in W} D(\Lambda,w(\mu)),$$
where $\mf{g}_\vartheta$ denotes the Lie algebra of $G_\vartheta$ and $\rk_\lambda$ is the $1$-dimensional representation of $\mf{g}_\vartheta$.
\end{thm}

Demazure modules for Kac-Moody groups have been extensively studied in the literature, including their character formula, cf.\,\cite{Ku,Ma88}. Theorem \ref{thm1.3} reveals an interesting relationship among different Demazure modules, that has not been previously noticed. In Section \ref{sect_order}, we discuss the inclusion relation among Demazure modules $\{ D(\Lambda,w(\mu)) \,|\, w\in W \}$, 
as well as the closure relation among the corresponding Iwahori Schubert varieties in $\Fl_Y$. We also give a description of maximal Iwahori Schubert varieties contained in $\mc{A}_Y(\mu)$, which contribute to the irreducible components of $\mc{A}_Y(\mu)$, see Corollary \ref{5982347}. In the Appendix \ref{293489}, we present some explicit examples to illustrate the phenomenon described in Theorem \ref{thm1.3}.
\vspace{0.1cm}

All of the results above are proved for general simple $G$ without assuming that $G$ is simply connected, and also in the twisted setting. The non-simply-connected case introduces additional subtleties, especially in the treatment of non-neutral components of partial aﬀine flag varieties.


\vspace{0.2cm}

{\bf Acknowledgement}. We would like to thank P.\,Belkale,  S.\,Kato, S.\,Kumar, S.\,Nie, and X.\,Zhu for helpful discussions. We are especially grateful to R.\,Travkin for discussions that led to a strengthening of Theorem \ref{569235} and a simplified proof.  We also thank P.\,Littelmann for valuable email correspondence. J.\,Hong was partially supported by NSF grant DMS-2001365 and Simons grant MPS-TSM-00007468.

The second version of this paper underwent substantial revisions and many improvements, most of which were carried out during J. Hong's visits to the Sydney Mathematical Research Institute and the Max Planck Institute for Mathematics in spring 2025. He gratefully acknowledges their hospitality and financial support.

In the most recent revision, we used OpenAI’s ChatGPT to assist with proofreading. It suggested several corrections and improvements, all of which were independently evaluated by the authors.

\section{Parahoric group revisited}
\label{Section_Parahoric}
In Section \ref{sec_parahoric}, we review and extend some results in \cite{DH} on the equivariant construction of parahoric group schemes over the ring of formal power series. In Section \ref{section_flag}, we discuss how to choose a rational point in the facet to avoid the characteristic constraint. We also describe a natural Levi subgroup in any parahoric group and introduce the group $G_\vartheta$ which will be used in Section \ref{3208530}. 
\subsection{Notations}\label{Notations}
Let $\rk$ be an algebraically closed field with characteristic $p$. Let $G$ be a connected simple algebraic group over $\rk$. We fix a Borel subgroup $B\subseteq G$ and a maximal torus $T\subseteq B$. Denote by $X_*(T)$ the lattice of cocharacters of $T$. Let $G_{\mr{ad}}$ be the adjoint group of $G$ and let $T_{\mr{ad}}$ be the image of $T$ in $G_{\mr{ad}}$. Similarly, let $X_*(T_{\rm ad})$ denote the lattice of cocharacters of $T_{\rm ad}$. 

Let $\tau$ be a diagram automorphism of $G$ of order $r$ preserving $B$ and $T$, and a pinning with respect to $B$ and $T$. Note that $\tau$ can be trivial, and $r$ can only be either 1, 2, or 3. Throughout this paper, we always assume $p\not= r$.  Let $\mc{K}:= \rk((z))$ and $\mc{O}:=\rk[[z]]$ be the ring of formal Laurent series and formal power series respectively. We fix a primitive $r$-th root of unity $\epsilon$ and define $\tau(z):=\epsilon^{-1}z$. Set $\bar{\mc{K}}:=\rk((z^r))$, and $\bar{\mc{O}}:=\rk[[z^r]]$. 

Let $G^{\tau,\circ}$ (resp. $B^{\tau,\circ}$, $T^{\tau,\circ}$ ) denote the neutral component of the 
$\tau$ fixed-point subgroup of $G$ (resp. $B$, $T$). Then, $B^{\tau,\circ}$ is a Borel subgroup of $G^{\tau,\circ}$ and $T^{\tau,\circ}$ is a maximal torus. In the following sections, $\square^\circ$ always stands for the neutral component of $\square$. Note that when $G$ is simply-connected or adjoint, $G^{\tau,\circ}=G^\tau$ and $T^{\tau,\circ}=T^\tau$.

Let $I_\tau$ (resp.\,$I$) be the set of vertices of the Dynkin diagram of $G^\tau$ (resp.\,$G$). Then, $I_\tau$ can be identified with the set of $\tau$-orbits in $I$. Let $X_N^{(r)}$ be the affine Dynkin diagram associated to the twisted loop group $G(\mc{K})^\tau$.  By adjoining an affine node $o$,  $\hat{I}_\tau:=I_\tau\sqcup \{o\}$ can be naturally regarded as the set of vertices of $X_N^{(r)}$. In our convention, $o$ corresponds to the vertex $0$ if $X_N^{(r)}\neq A_{2\ell}^{(2)}$ or the vertex $\ell$ if $X_N^{(r)}=A_{2\ell}^{(2)}$ in \cite[Chapter 4]{Kac90}. 

Let $(a_i)_{i\in \hat{I}_\tau}$ be the Kac labels in the table in \cite[Chapter 4]{Kac90} and let $(\check{a}_i)_{i\in \hat{I}_\tau}$ be the dual Kac labels. Note that  $\check{a}_o=2$ if $X_N^{(r)}=A_{2\ell}^{(2)}$, and $\check{a}_o=1$ otherwise; moreover, $a_o=1$ for any cases.  Let $\{\alpha_i\st i\in \hat{I}_\tau\}$ be the set of simple affine roots and let $\{\check\alpha_i\st i\in \hat{I}_\tau\}$ be the set of simple affine coroots. Set 
\begin{equation}
\label{eq_theta_0}
\textstyle\theta_0:=\sum_{i\in I_\tau} a_i\alpha_i,
\end{equation}
We remark that $\theta_0$ agrees with the definition in \cite[Proposition 8.3]{Kac90}, which can be described as in \eqref{eq_theta0}. Let $\{\omega_i\st i\in I_\tau\}$ be the set of fundamental weights of the simply-connected cover $(G^{\tau,\circ})'$ of $G^{\tau,\circ}$.

Following \cite[\S 6.2]{Kac90}, we can associate the following vector space 
$$\mf{h}^\tau_{\mb{Q}}:=\bigoplus_{i\in I_\tau} \mb{Q}\check{\alpha}_i,\quad \hat{\mf{h}}^\tau_\mb{Q}:=\bigoplus_{i\in \hat{I}_\tau} \mb{Q}\check{\alpha}_i\oplus \mb{Q}d,$$
where $d$ is the scaling element.  We can also define the central element as $\textstyle K=\sum_{i\in \hat{I}_\tau} \check{a}_i \check{\alpha}_i$.

For each $i\in \hat{I}_\tau$, let $\Lambda_i\in (\hat{\mf{h}}^\tau_\mb{Q})^*$ be the associated affine fundamental weight, i.e. $\langle\Lambda_i, \check{\alpha}_j \rangle=\delta_{ij}$, and $\langle \Lambda_i, d \rangle=0$.
Then $ (\hat{\mf{h}}^\tau_\mb{Q})^*$ admits the following decomposition 
\begin{equation}
\label{delta_elem}
(\hat{\mf{h}}^\tau_\mb{Q})^*=(\mf{h}^\tau_\mb{Q})^*\oplus \mb{Q}\Lambda_o\oplus \mb{Q} \delta ,\end{equation}
 where $\delta=\sum_{i\in \hat{I}_\tau} a_i \alpha_i$.  
Let $(\cdot|\cdot)$ be the normalized invariant form on $\hat{\mf{h}}^\tau_\mb{Q}$ defined as in \cite[\S 6.2]{Kac90}.  This bilinear form induces a map $\nu: \hat{\mf{h}}^\tau_\mb{Q}\to (\hat{\mf{h}}^\tau_\mb{Q})^*$. For any $i\in \hat{I}_\tau$, we have 
\begin{equation}\label{eq_dualmap}
\nu(\check{\alpha}_i)=\frac{a_i}{\check{a}_i}\alpha_i, \quad \nu(K)=\delta.
\end{equation}
Note that we have the following natural identifications: $\mf{h}^\tau_{\mb{Q}}\simeq X_*(T)^\tau_{\mathbb{Q}}\simeq X_*(T_{\rm ad})^\tau_{\mb{Q}}$. Moreover, $\nu$ restricts to a map $\nu: \mf{h}^\tau_\mb{Q}\simeq (\mf{h}^\tau_\mb{Q})^*$.

\subsection{Parahoric group schemes from the equivariant point of view}\label{sec_parahoric}  

Let $\ms{G}_{\bar{\mc{K}}}$ be the quasi-split group scheme $\mr{Res}_{\mc{K}/\bar{\mc{K}}}(G_{\mc{K}})^\tau$ over $\bar{\mc{K}}$, i.e.\,the subgroup scheme of $\tau$-invariants of the Weil restriction $\mr{Res}_{\mc{K}/\bar{\mc{K}}}(G_{\mc{K}})$. Note that any tamely-ramified quasi-split simple group scheme over $\bar{\mc{K}}$ can be realized from such construction.

Let $S$ denote the group scheme $(T^{\tau,\circ})_{\bar{\mc{K}}}$, which is a maximal $\bar{\mc{K}}$-split torus of $\ms{G}_{\bar{\mc{K}}}$. Then, the relative root system $\mc{R}$ for $(\ms{G}_{\bar{\mc{K}}},S)$ can be identified with the quotient of the root system $R$ of $(G,T)$ by the $\tau$-action. We remark that the root system $R(G^{\tau,\circ}, T^{\tau,\circ})$ of $(G^{\tau,\circ}, T^{\tau,\circ})$ is naturally a subsystem of $\mc{R}$. In fact, $R(G^{\tau,\circ}, T^{\tau,\circ})=\mc{R}$ if and only if $(G,r)\not= (A_{2\ell}, 2)$.

Let $\Delta_{{\alpha}}\subseteq R$ denote the set of preimages of $\alpha\in \mc{R}$ via the projection map $R\to \mc{R}$. For each $\alpha\in \mc{R}$, we denote by $\mc{K}_\alpha$ the subfield of $\mc{K}$ which is fixed by the stabilizer of any $\tilde{\alpha}\in \Delta_{\alpha}$ under the action of the cyclic group $\langle \tau \rangle $. Let  $d_\alpha$ be the cardinality of the set $\Delta_\alpha$. Then, $\mc{K}_\alpha=\rk((z^{r/d_\alpha}))$.

We now recall the root subgroup $\ms{U}_\alpha$ of $\ms{G}_{\bar{\mc{K}}}$ for $\alpha\in \mc{R}$. When $2\alpha, \frac{1}{2}\alpha\notin \mc{R}$, the root subgroup $\ms{U}_\alpha$ is given by 
$$\Big(\prod_{\tilde{\alpha}\in \Delta_\alpha} U_{\tilde{\alpha}}(\mc{K}_\alpha)\Big)^\tau,$$
where $U_{\tilde{\alpha}}$ denote the root subgroup of $G$ associated to $\tilde{\alpha}$. 
Since the $\tau$-action on $\Delta_{\alpha}$ is transitive, the root subgroup homomorphism $x_{\tilde{\alpha}}:\mb{G}_{a}\to U_{\tilde{\alpha}}$ induces a Chevalley-Steinberg pinning $\{ x_\alpha: \mc{K}_{\alpha}\to \ms{U}_{\alpha} \,|\, \alpha\in \mc{R} \}$ of $\ms{G}_{\bar{\mc{K}}}$, cf.\,\cite[4.1.5]{BT84}. Let $\nu$ be the valuation of $\mc{K}$ with $\nu(z)=1$. Via the pinning $x_\alpha$, the restriction of $\nu$ on $\mc{K}_\alpha$ induces a valuation ${\nu}$ on $\ms{U}_\alpha$ (by abuse of notation). For every $\ell\in \mb{R}$, we define $\ms{U}_{\alpha,\ell}:=\nu^{-1}[\ell, \infty] \subseteq \ms{U}_\alpha$.

When $2\alpha\in \mc{R}$, we have $r=2$, $\Delta_\alpha=\{\tilde{\alpha},\tau(\tilde{\alpha})\}$ is the set of two adjacent roots in $R$, and $\Delta_{2\alpha}=\{\tilde{\alpha}+\tau(\tilde{\alpha}) \}$. In this case, we have 
$$\ms{U}_{\alpha}=(U_{\tilde{\alpha}}(\mc{K})U_{\tilde{\alpha}+\tau(\tilde{\alpha})}(\mc{K})U_{\tau(\tilde{\alpha})}(\mc{K}))^\tau, \quad \ms{U}_{2\alpha}=U_{\tilde{\alpha}+\tau(\tilde{\alpha})}(\bar{\mc{K}}).$$
There are natural valuations defined on $\ms{U}_{\alpha}$ and $\ms{U}_{2\alpha}$, and we can accordingly associate subgroups $\ms{U}_{\alpha,\ell}$ and $\ms{U}_{2\alpha, \ell}$ for every $\ell\in \mb{R}$, cf.\,\cite[4.1.9]{BT84}. By abuse of notation,  we will use the notation $\nu$ uniformly for the valuation on the root subgroup $\ms{U}_{\alpha}$ for any $\alpha\in \mc{R}$.

\begin{definition}\label{def_parahoric}
Following \cite[Definition 2.1.2]{DH} and references therein, for any $\vartheta\in X_*(S)_\mb{R}\simeq X_*(T)^\tau_ \mb{R}$, we define the parahoric subgroup $\mc{P}_\vartheta$ to be the subgroup of $G(\mc{K})^\tau$  generated by $\{\ms{U}_{\alpha,-\langle\vartheta,\alpha\rangle}\}_{\alpha\in\mc{R}}$ and $T(\mc{O})^{\tau,\circ}$ (the neutral component of $T(\mc{O})^{\tau}$). Then, there exists a unique affine smooth group scheme $\ms{G}_\vartheta$ over $\bar{\mc{O}}$, which extends $\ms{G}_{\bar{\mc{K}}}$ and satisfies that $\ms{G}_\vartheta(\bar{\mc{O}})=\mc{P}_\vartheta$. 
\end{definition}

Similarly, we define a subgroup $\mc{P}_\vartheta^-$ of $G[z,z^{-1}]^\tau$ as follows. Let $\nu_\infty$ be the valuation of $\rk[z,z^{-1}]$ with $\nu_\infty(z^{-1})=1$. When $2\alpha, \frac{1}{2}\alpha\notin \mc{R}$, via the pinning $x_\alpha$, the restriction of  $\nu_\infty$ on $\rk[z,z^{-1}]\cap\mc{K}_\alpha$ induces a valuation $\nu_\infty$ on $\ms{U}_\alpha\cap G[z,z^{-1}]^\tau$. For every $\ell\in \mb{R}$, we define $\ms{U}^-_{\alpha,\ell}:=\nu_\infty^{-1}[\ell,\infty] \subseteq \ms{U}_\alpha\cap G[z,z^{-1}]^\tau$. When $2\alpha\in \mc{R}$, we can define subgroups $\ms{U}_{\alpha,\ell}^-$ and $\ms{U}_{2\alpha, \ell}^-$ for every $\ell\in \mb{R}$ in a similar way as explained before. Then, $\mc{P}_\vartheta^-$ is the subgroup of $G[z,z^{-1}]^\tau$ generated by $\{\ms{U}^-_{\alpha,\langle\vartheta,\alpha\rangle}\}_{\alpha\in\mc{R}}$ and $T^{\tau,\circ}$.

\begin{remark}
The valuation on $\rk((z))$ used in our definition is different from the one used in \cite{DH}, and the parahoric subgroup $\mc{P}_\vartheta$ defined in this paper coincides with the parahoric subgroup $\mc{P}_{r\vartheta}$ defined in loc.\,cit.  
\end{remark}

Fix a point $\vartheta\in X_*(T)^\tau_ \mb{Q}$.  Let $m$ be the minimal positive multiple of $r$ such that $\frac{m}{r}\vartheta\in X_*(T_{\mr{ad}})^\tau$. Identify the lattice $X_*(T_{\mr{ad}})^\tau$ with $\bigoplus_{i\in I_\tau} \mb{Z}\cdot \check{\omega}_i \subseteq \mf{h}^\tau_{\mb{Q}}$, where $\check{\omega}_i$ represents the fundamental coweight of $G^\tau$ associated to $i\in I_\tau$. Then, we can write 
\[\textstyle \frac{m}{r}\vartheta=\sum_{i\in I_\tau}s_i\check{\omega}_i,\] 
where $s_i$ are integers. Suppose that $p\nmid m$.
Fix a primitive $m$-th root of unity $\epsilon_m$. We associate an element $(\epsilon_m)^{m\vartheta/r}\in T_{\rm ad}$, and
we define an automorphism of $G$ 
\begin{equation}
\label{eq_sigma}
\sigma:=\tau\circ \mr{Ad}_{(\epsilon_m)^{m\vartheta/r}}.
\end{equation}
\begin{lem}
\label{Order_lemma}
The order of $\sigma$ is $m$.
\end{lem}
\begin{proof}
Clearly, we have $\sigma^m=\mr{id}$. Suppose $\mr{ord}(\sigma)=m'$. Then, $m'\mid m$. Recall that $I_\tau$ can be identified with the set of $\tau$-orbits in $I$. For each $i\in I_\tau$, let $\eta_i$ be the corresponding $\tau$-orbit in $I$. Pick $j\in \eta_i$ and  let $x_{\tilde{\alpha}_j}:\mb{G}_a\to G$ be the root subgroup associated to the simple root $\tilde{\alpha}_j$ of $G$. Then, $\tilde{\alpha}_j|_{ T^{\tau,\circ}}=\alpha_j$ and we have 
$$x_{\tilde{\alpha}_j}(u)=\sigma^{m'}(x_{\tilde{\alpha}_j}(u))=x_{\tau^{m'}\tilde{\alpha}_j}\big((\epsilon_m)^{m'\langle h,\alpha_i\rangle}u\big)=x_{\tau^{m'}\tilde{\alpha}_j}\big((\epsilon_m)^{m's_i}u\big),\quad u\in \rk .$$ 
Hence, $m'$ is a multiple of $r$, and $m's_i$ is a multiple of $m$ for all $i\in I_\tau$. Therefore, $\frac{m'}{r}\cdot\gcd\{s_i\st i\in I_\tau\}$ is a multiple of $\frac{m}{r}$. From the minimality of $m$, we have $\gcd\{s_i\st i\in I_\tau\}$ and $\frac{m}{r}$ are coprime. Thus, $m\mid m'$ which implies $m=m'=\mr{ord}(\sigma)$.
\end{proof}

Set $h:= \frac{m}{r}\vartheta\in  X_*(T_{\mr{ad}})^\tau$. Set $t=z^{r/m}$, $\mc{O}_t:=\rk[[t]]$, and $\mc{K}_t:=\rk((t))$.  Define an action of $\sigma$ on $t$ by $\sigma(t):= (\epsilon_m)^{-1}t$. Then, $\bar{\mc{O}}=\rk[[z^r]]=\rk[[t^m]]$ is the set of $\sigma$-fixed elements in $\mc{O}_t$.  Consider the group scheme $\mc{H}:=\mr{Res}_{\mc{O}_t/\bar{\mc{O}}}(G_{\mc{O}_t})^\sigma$  over $\bar{\mc{O}}$, which is the subgroup scheme of $\sigma$-invariants of the Weil restriction $\mr{Res}_{\mc{O}_t/\bar{\mc{O}}}(G_{\mc{O}_t})$.
We define the neutral component $\mc{H}^\circ$ of $\mc{H}$ to be the union of the neutral components $(\mc{H}_a)^\circ$ of fibers at every $a\in \mr{Spec}(\bar{\mc{O}})$, cf.\,\cite[1.1.12]{BT84}. Then, $\mc{H}^\circ$ is an open subscheme of $\mc{H}$, denoted by $\ms{G}_{\sigma}$, i.e.
\begin{equation}\label{eq_Gsigma}
\ms{G}_{\sigma}:=\mr{Res}_{\mc{O}_t/\bar{\mc{O}}}(G_{\mc{O}_t})^{\sigma,\circ}.
\end{equation}

We can regard $G(\mc{O}_t)^\sigma$ as a pro-algebraic group over $\rk$. Let $G(\mc{O}_t)^{\sigma,\circ}$ be the connected component of $G(\mc{O}_t)^\sigma$
as a pro-algebraic group. 

The following theorem generalizes \cite[Theorem 4.1.2]{DH}.

\begin{thm}
\label{paraohoric_thm}
When $p\not \mid m$, 
there is a unique isomorphism of group schemes over $\bar{\mc{O}}$:
\begin{equation}\label{eq_phi} 
\phi_0:\ms{G}_\vartheta\simeq \ms{G}_{\sigma}.
\end{equation}
such that over its $\bar{\mc{O}}$-points, $\phi_0:\mc{P}_\vartheta\simeq G(\mc{O}_t)^{\sigma,\circ}$ is the restriction of the isomorphism $\mr{Ad}_{t^h}:G(\mc{K})^{\tau}\simeq G(\mc{K}_t)^{\sigma}$.
\end{thm}
\begin{proof}
When $G$ is simply-connected, the theorem is proved in \cite[Theorem 4.1.2]{DH}. For general simple algebraic group $G$,  we can adapt the same argument, except that we need to check  $\ms{G}_\sigma( \bar{\mc{O}})= G(\mc{O}_t)^{\sigma, \circ}$.  Since $\ms{G}_\sigma$ is the open subscheme of $\mr{Res}_{\mc{O}_t/\bar{\mc{O}}}(G_{\mc{O}_t})^{\sigma}$
by throwing away the non-neutral components of the special fiber $$\mr{Res}_{\mc{O}_t/\bar{\mc{O}}}(G_{\mc{O}_t})^{\sigma}|_{0}=G(\rk[t]/(t^m))^\sigma,$$
$\ms{G}_\sigma(\bar{\mc{O}})$ consists of the points $x\in G(\rk[[t]])^\sigma$ such that its image $\bar{x}_m \in G(\rk[t]/(t^m))^\sigma$ is contained in the neutral component $G(\rk[t]/(t^m))^{\sigma, \circ}$, equivalently the image $\bar{x}$ of $x$ in $G^\sigma$ is contained in $G^{\sigma,\circ}$. Thus, $\ms{G}_\sigma( \bar{\mc{O}})= G(\mc{O}_t)^{\sigma, \circ}$.
Applying the argument in \cite[Theorem 4.1.2]{DH}, we conclude that $\mathrm{Ad}_{t^h}(\mc{P}_\vartheta)=G(\mc{O}_t)^{\sigma, \circ}=\ms{G}_\sigma( \bar{\mc{O}})$. In view of \cite[I.7.6]{BT84}, the isomorphism $\mr{Ad}_{t^h}:G(\mc{K})^{\tau}\simeq G(\mc{K}_t)^{\sigma}$ uniquely extends to an isomorphism $\ms{G}_\vartheta\simeq \ms{G}_{\sigma}.$
\end{proof}

\begin{remark}
The parahoric subgroup only depends on the minimal facet of the alcove containing $\vartheta$. In fact, under the condition that $p\neq r$ and that $p$ satisfies a very mild condition, we can always find a rational interior point $\vartheta$ of this facet such that $p\nmid m$, including the case when the facet is a vertex, see Lemma \ref{prop_goodprime}.
\end{remark}
Replacing $z$ by $z^{-1}$, let $v_\infty$ be the valuation on $k((z^{-1}))$ with $v_\infty(z^{-1})=1$, we can similarly define a parahoric group scheme $\ms{G}_{-\vartheta,\infty}$ over $\bar{\mc{O}}_\infty$ associated with $-\vartheta$, where $\bar{\mc{O}}_\infty:=\rk[[z^{-r}]]$. Set $\mc{O}_{t,\infty}:=\rk[[t^{-1}]]$ and define a parahoric group scheme $\ms{G}_{\sigma,\infty}:= \mr{Res}_{\mc{O}_{t,\infty}/\bar{\mc{O}}_\infty}(G_{\mc{O}_{t,\infty}})^{\sigma,\circ}$. Again, we have an isomorphism of parahoric group schemes
\begin{equation}\label{eq_phi2} 
\phi_\infty:\ms{G}_{-\vartheta,\infty}\simeq \ms{G}_{\sigma,\infty}
\end{equation} 
such that over its $\bar{\mc{O}}_\infty$-points, $\phi_\infty:\ms{G}_{-\vartheta,\infty}(\bar{\mc{O}}_\infty)\simeq G[[t^{-1}]]^{\sigma,\circ}$ is the restriction of the isomorphism  
\begin{equation}\label{eq_adztheta}
\mr{Ad}_{(t^{-1})^{-h}}=\mr{Ad}_{t^{h}}:G((z^{-1}))^\tau\simeq G((t^{-1}))^{\sigma}.
\end{equation}

\subsection{Canonical Levi subgroups of Parahoric groups}\label{section_flag}

Let $\textsf{A}$ be the fundamental alcove in $X_*(T)^\tau_\mb{R}$ defined by inequalities $\langle x, \alpha_i \rangle \geq 0$ for each $i\in I_\tau$, and $\langle x, \theta_0 \rangle \leq 1$. For convenience, set $\check{\omega}_o=0$. Then, $\textsf{A}$ is a simplex with vertices $\{\frac{\check{\omega}_i }{a_i} \mid i\in \hat{I}_\tau\}$. In this section, we fix a nonempty subset $Y\subseteq \hat{I}_\tau$ and a facet $F_Y$ of $\textsf{A}$ spanned by the set of vertices $\{\frac{\check{\omega}_i }{a_i} \mid i\in Y \}$.  We associate a positive integer 
\begin{equation}
\label{index_aY}
a_Y:=r\cdot\gcd\{a_i\st i\in Y\}.\end{equation}

\begin{lem}\label{prop_goodprime}
For any $p\nmid a_Y$, there is a rational interior point $\vartheta\in F_Y$ such that $p\nmid m$, where $m$ is the minimal positive multiple of $r$ such that $\frac{m}{r}\vartheta\in X_*(T_{\mr{ad}})^\tau$.
\end{lem}
\begin{proof}
Set $n_Y:=\sum_{i\in Y}a_i$ and $x_Y:=\frac{1}{n_Y}\sum_{i\in Y} \check{\omega}_i$. For each $k\in Y$, set $n_k:=a_k+n_Y$ and $x_k:=\frac{1}{n_k}\big(\check{\omega}_k+\sum_{i\in Y} \check{\omega}_i\big)$. Then, $x_Y$ and $x_k$ are rational interior points of $F_Y$ for any $k\in Y$. We will show that some point among $\{x_Y, x_k\st k\in Y\}$ satisfies the property in the lemma. 

Let $m_Y$ and $m_k$ be minimal positive multiple of $r$ such that $\frac{m_Y}{r}x_Y\in X_*(T_{\mr{ad}})^\tau$ and $\frac{m_k}{r}x_k\in X_*(T_{\mr{ad}})^\tau$ respectively. Clearly, $\frac{m_Y}{r}\mid n_Y$ and $\frac{m_k}{r}\mid n_k$. We claim $p\nmid \frac{m_Y}{r}$ or $p\nmid \frac{m_k}{r}$ for some $k\in Y$. Suppose that $p\mid \frac{m_Y}{r}$ and $p\mid \frac{m_k}{r}$ for all $k\in Y$. Then, $p\mid n_Y$ and $p\mid n_k$ for all $k\in Y$. Hence, $p\mid a_k$ for all $k\in Y$. This contradicts with the assumption that $p\nmid a_Y$. Therefore, $p\nmid \frac{m_Y}{r}$ or $p\nmid \frac{m_k}{r}$ for some $k\in Y$. Note that the assumption $p\nmid a_Y$ implies that $p\nmid r$. Since $p$ is a prime, it follows that $p\nmid m_Y$ or $p\nmid m_k$ for some $k\in Y$.
\end{proof}
\begin{remark} From the tables in \cite[p.54-55]{Kac90}, we notice that the only primes that may  divide $a_Y$ are $2,3,5$. It is also interesting to notice that $p$ is a good prime for $G$ in the sense of \cite[p.28]{Ca} if and only if $p \nmid a_Y$ for any $Y\subset \hat{I}$ when $\tau$ is trivial.
\end{remark}
From now on, we always assume $p\nmid a_Y$. Moreover, we fix a rational interior point $\vartheta\in F_Y$ with prescribed property as in Lemma \ref{prop_goodprime}.
\begin{definition}\label{def_Py}  
Set $\mc{P}_Y:=\mc{P}_{\vartheta}$ and $\ms{G}_Y:=\ms{G}_{\vartheta}$. In fact, $\mc{P}_Y$ is the parahoric subgroup corresponding to the facet $F_Y$. We call it the parahoric subgroup associated to $Y$. When $Y=\hat{I}_\tau$, $\mc{P}_Y$ is the Iwahori subgroup, denoted by $\mc{I}$; when $Y=\{o\}$, $\mc{P}_Y=G(\mc{O})^{\tau,\circ}$. Similarly, we set $\mc{P}_Y^-:=\mc{P}^-_{\vartheta}$. When $Y=\hat{I}_\tau$, we denote $\mc{P}_Y^-$ by $\mc{I}^-$. 
\end{definition}

Set $h:=\frac{m}{r}\vartheta\in X_*(T_{\mr{ad}})^\tau$. We can write $h=\sum_{i\in I_\tau}s_i\check{\omega}_i$ for non-negative integers $s_i, i\in I_\tau$. Since $\langle \vartheta,\theta_0\rangle\leq 1$, we  have $\sum_{i\in I_\tau} a_is_i\leq \frac{m}{r}$. Set 
\[\textstyle s_o:= \frac{m}{r}-\sum_{i\in I_\tau} a_is_i\in \mb{Z}^{\geq 0}.\]
Since $\vartheta$ is an interior point of $F_Y$, we must have $Y=\{i\in \hat{I}_\tau\st s_i\neq 0\}$. Define an automorphism $\sigma:=\tau\circ \mr{Ad}_{(\epsilon_m)^{m\vartheta/r}}$ of $G$, whose order is $m$, see Lemma \ref{Order_lemma}.  By Theorem \ref{paraohoric_thm}, we have an isomorphism of parahoric group schemes
\begin{equation}\label{eq_phi3} 
\phi:\ms{G}_Y\simeq \ms{G}_{\sigma}.
\end{equation}
such that over its $\bar{\mc{O}}$-points, $\phi:\mc{P}_Y\simeq G(\mc{O}_t)^{\sigma,\circ}$ is the restriction of the isomorphism $\mr{Ad}_{t^{h}}:G(\mc{K})^\tau\simeq G(\mc{K}_t)^{\sigma}$. We define $M_Y=\mr{Ad}_{t^{-h}}(G^{\sigma,\circ})$ as a subgroup of $\mc{P}_Y$. Clearly, $M_Y$ is a Levi subgroup of $\mc{P}_Y$. 

We describe a pinning of $G^{\tau,\circ}$ as follows, cf.\,\cite[Lemma 2.2]{HS}.  Let $\{\tilde{\alpha}_j\st j\in I\}\subseteq R$ denote the set of simple roots of $G$ with respect to the pair $(B,T)$. We fix a pinning $\{x_{\pm\tilde{\alpha}_j}:\mb{G}_{\rm a}\to G\st j\in I\}$ of $G$ which is preserved by $\tau$, i.e.\,$\tau(x_{\pm\tilde{\alpha}_j}(a))=x_{\pm\tilde{\alpha}_{\tau(j)}}(a)$ for any $a\in \rk$. For each $i\in I_\tau$, let $\eta_i$ denote the preimage of $i$ via the orbit map $I\to I_\tau$. 

 If the vertices in $\eta_i$ are not adjacent, set
\begin{equation}\label{eq_pinning1}
x_{\alpha_i}(a)=\prod_{j\in \eta_i} x_{\tilde{\alpha}_j}(a), \quad x_{-\alpha_i}(a)=\prod_{j\in \eta_i} x_{-\tilde{\alpha}_j}(a),\quad \forall a\in \rk.
\end{equation} 
If the vertices in $\eta_i$ are adjacent, we can write $\eta_i=\{j, \tau(j)\}$. Set
\begin{equation}\label{eq_pinning2}
x_{\alpha_i}(a)=x_{\tilde{\alpha}_j}(a)x_{\tilde{\alpha}_{\tau(j)}}(2a)x_{\tilde{\alpha}_j}(a),\quad x_{-\alpha_i}(a)=x_{-\tilde{\alpha}_j}(\frac{a}{2})x_{-\tilde{\alpha}_{\tau(j)}}(a)x_{-\tilde{\alpha}_j}(\frac{a}{2}), \quad \forall a\in \rk.
\end{equation}

Let $\theta^0$ be the root of $G$ described in \cite[Section 8.3]{Kac90}. Then, 
\begin{equation}
\label{eq_theta}
\theta^0|_{T^{\tau,\circ}}=\theta_0. 
\end{equation}
When $(G,r)=(A_{2\ell}, 2)$, from \cite[Section 8.3]{Kac90} $\theta^{0}$ is the highest root of $G$, we choose any root subgroups $x_{\pm \theta^0}$ associated to $\pm \theta^0$. When $r>1$ and $(G,r)\not =(A_{2\ell}, 2)$, $\tau(\theta^0)\not= \theta^0$. In this case, we can choose root subgroups $\{ x_{\pm \tau^k(\theta^0)} \mid k=0,\cdots, r-1  \}$ so that they are preserved by $\tau$. 

For any $u\in \mc{K}$, we set  
$$x_{\pm\theta_0}(u):=
\begin{cases}
x_{\pm\theta_0}(u) & \text{when } r=1\\
x_{\pm\theta^0}(u)& \text{when } (G,r)= (A_{2\ell},2)\\
\prod_{k=0}^{r-1}x_{\pm\tau^k(\theta^0)}(\tau^k(u))& \text{when } r\neq 1 \text{ and } (G,r)\neq (A_{2\ell},2)
\end{cases}.$$
For any $a\in \rk$, we define $$x_{\pm\alpha_o}(a):=x_{\mp\theta_0}(az^{\pm1})\in G(\mc{K})^\tau.$$
Set $\breve{Y}=\hat{I}_\tau\setminus Y$.
 
\begin{prop}\label{eq_levi}
With the same setting as above, we have
\begin{enumerate}
\item The group $G^{\sigma, \circ}$ is generated by 
root subgroups $\{ x_{\pm \alpha} \mid \alpha \in \Pi_{\breve{Y}}  \}$ and $T^{\tau, \circ}$, where $\Pi_{\breve{Y}}$ is the following set of roots
$$\Pi_{\breve{Y}}=\begin{cases}  
\{\, -\theta_0,\alpha_i\st i\in \breve{Y} \setminus \{o\} \,\},  &  \text{if } o \in  \breve{Y}; \\
\{\,\alpha_i\st i \in \breve{Y}\}, & \text{if } o\not\in \breve{Y}. \end{cases}$$

\item The group $M_Y$ is generated by $\{x_{\pm\alpha_i}(a)\st i\in \breve{Y}, a\in \rk\}$ and $T^{\tau,\circ}$. 
\end{enumerate}
\end{prop}
\begin{proof}
By Lemma \ref{Order_lemma} and the assumption that $p\not \mid m$,  $p$ does not divide the order of $\sigma$. In view of \cite[Theorem 8.1]{Steinberg68}, $G^{\sigma,\circ}$ is a connected reductive algebraic group. Moreover, by \cite[9.1]{Borel91}, 
${\rm Lie} (G^{\sigma,\circ})= \mathfrak{g}^\sigma$.
When ${\rm char } ( \rk)=0$, the root system of $\mathfrak{g}^\sigma$ is generated by  $\Pi_{\breve{Y}}$ as a set 
of simple roots,  cf.\,\cite[Section 2]{HK}. When $p\neq 0$, recall that $\sigma=\tau\circ \mr{Ad}_{(\epsilon_m)^{m\vartheta/r}}$, the same root-space calculation in \cite[Theorem 8.7]{Kac90} remains valid since $p\nmid m$. Thus, the root system of $\mathfrak{g}^\sigma$ is still generated by  $\Pi_{\breve{Y}}$ as a set of simple roots.
By \cite[Corollary 8.2.10]{Springer98}, 
$G^{\sigma,\circ }$ is generated by the root subgroups $\{ x_{\pm \alpha} \mid \alpha \in \Pi_{\breve{Y}}  \}$ and $T^{\tau, \circ}$.  This proves part (1).
 
Recall that $M_Y=\mr{Ad}_{t^{-h}}(G^{\sigma,\circ})$. For any $i\in \breve{Y} \setminus \{o\}$, by definition $\langle h, \alpha_i \rangle=0$. Then, 
$$  \mr{Ad}_{t^{-h}}(x_{\pm\alpha_i}(a))=x_{\pm\alpha_i}(a), \quad \text{for any } a\in \rk. $$
If $o\in \breve{Y}$, then $s_o=0$, equivalently $\langle h, \theta_0  \rangle=\sum_{i\in I_\tau} a_i s_i=m/r$. Then, 
$$\mr{Ad}_{t^{-h}}(x_{\pm\theta_0}(a))=x_{\pm \theta_0}(at^{\langle -h,\pm\theta_0\rangle})=x_{\pm \theta_0}(at^{\mp m/r}   )   =  x_{\mp\alpha_o}(a).$$ 
Now, part $(2)$ follows from part $(1)$. 
\end{proof}

Note that $h\in X_*(T_{\rm ad})^\tau$ can be viewed as a cocharacter $h: \mb{G}_m\to T_{\rm ad}^{\tau}$. 
We define $G_\vartheta$ to be the centralizer of $h$ in $G^{\tau,\circ}$, i.e. 
\begin{equation}\label{eq_Glambda}
G_\vartheta=\{g\in G^{\tau, \circ} \mid  {\rm Ad}_{h(s)} (g)=g , \text{ for any } s\in \mb{G}_m \}.
\end{equation}
Here, we use the notation $G_\vartheta$, as this group only depends on $\vartheta$. In fact, it only depends on $Y$ by the following corollary.
\begin{cor}\label{cor_glambdamy}
Set $Y'=Y\cup \{o\}$. We have $G_\vartheta=M_{Y'}$. In particular, if $o\in Y$, then $G_\vartheta=M_{Y}$; if $o\not \in Y$, then $G_\vartheta\subsetneqq M_{Y}$.
\end{cor}
\begin{proof}
First of all, we have $G_\vartheta=\mr{Ad}_{t^h}(G_\vartheta)\subseteq G^{\sigma,\circ}$. Hence, $G_\vartheta=\mr{Ad}_{t^{-h}}(G_\vartheta)\subseteq M_Y$. 

When $o\in Y$, $M_Y$ is generated by $\{x_{\pm\alpha_i}(a)\st i\not \in Y, a\in \rk\}$ and $T^{\tau,\circ}$, which is contained in $G_\vartheta$. Therefore, $G_\vartheta=M_Y=M_{Y'}$.

When $o\notin Y$, i.e. $s_o=0$, we have $\langle \vartheta, \theta_0\rangle =1$. Let $q>\frac{m}{r}$ be a prime number so that $\frac{1}{q}h\notin X_*(T_{\mr{ad}})^\tau$ and $q\neq p$. Set $\vartheta':=\frac{1}{q}h$. We claim that $q$ is the minimal positive integer so that $q\vartheta'\in X_*(T_{\mr{ad}})^\tau$. Suppose $a\vartheta'\in X_*(T_{\mr{ad}})^\tau$ for some integer $0<a<q$. Pick integers $u,v$ so that $uq+va=1$. Then, $\vartheta'=(uq+va)\vartheta'=uh+va\vartheta'\in X_*(T_{\mr{ad}})^\tau$, which is a contradiction. Hence, $q$ is the minimal positive integer so that $q\vartheta'\in X_*(T_{\mr{ad}})^\tau$. Therefore, $m':=qr$ is the minimal positive multiple of $r$ so that $\frac{m'}{r}\vartheta'\in X_*(T_{\mr{ad}})^\tau$. Since $h':=\frac{m'}{r}\vartheta'=h$, we conclude that $G_{\vartheta'}=G_\vartheta$. Note that $\langle \vartheta', \theta_0\rangle =\frac{1}{q}\cdot\frac{m}{r}<1$. The element $\vartheta'$ is a rational interior point of $F_{Y'}$.  By the first case, we have $G_{\vartheta'}=M_{Y'}$ (the condition $p\nmid m'$ clearly holds). Therefore, $G_\vartheta=M_{Y'}$. In this case, we have $M_{Y'}\subsetneqq M_{Y}$, see Proposition \ref{eq_levi}.
\end{proof}

\section{A refinement of the coherence conjecture (after Zhu)}\label{3208530}

In this section, we first introduce necessary notation and formulate the main result of this section in Theorem \ref{thm_mainthm}. Section \ref{sec_equivariantbundle} is devoted to the discussion of equivariant line bundles on partial affine flag varieties.  In Section \ref{43287437}, we extend line bundles on partial affine flag varieties to the global affine Grassmannian $\Gr_{\mc{G}}$ of a parahoric Bruhat-Tits group scheme $\mc{G}$ over $\mathbb{A}^1$ associated to $G$
and endow them with an $L^-\mc{G}$-equivariant structure when $G$ is simply-connected, which ultimately leads to the proof of Theorem \ref{thm_mainthm}.
\subsection{Statement of theorem}\label{Section_main}
In this section, we follow the same notation and assumptions as in Section \ref{Section_Parahoric}. We emphasize again that the automorphism $\tau$ on $G$ can be trivial.

Let $W$ be the Weyl group of $G$ with respect to the maximal torus $T$. Then, the subgroup $W^\tau$ of $\tau$-fixed points  is naturally isomorphic to the Weyl group of $G^{\tau,\circ}$ with respect to $T^{\tau,\circ}$. Let $X_*(T)_\tau$ be the lattice of $\tau$-coinvariants of $X_*(T)$. 

Recall the isomorphism $\nu:\mf{h}^\tau_{\mb{Q}}\to (\mf{h}^\tau_{\mb{Q}})^*$ in \eqref{eq_dualmap}. We identify $\mf{h}^\tau_{\mb{Q}}\simeq X_*(T)^\tau_{\mb{Q}}$. Then, we have the following well-defined map 
\begin{equation}\label{4529634}
\textstyle \iota:X_*(T)_\tau\to (\mf{h}^\tau_{\mb{Q}})^*, \quad \mu\mapsto \nu(\sum_{i=0}^{r-1} \tau^i(\tilde\mu)), 
\end{equation}
where $\tilde\mu$ is a lifting of $\mu$ in $X_*(T)$.  

\begin{definition}
Let $\tilde{W}_{\mr{aff}}:=X_*(T)_\tau\rtimes W^\tau$ be the extended affine Weyl group. The elements in $\tilde{W}_{\mr{aff}}$ are written as $\varrho^{\bar\mu}w$ for any $\bar{\mu}\in X_*(T)_\tau$ and $w\in W^\tau$, where $\varrho^{\bar{\mu} }$ represents the translation on the apartment of $(\ms{G}_{\bar{\mc{K}}},S)$ by $-\bar{\mu}$. 
\end{definition}
The action of $\tilde{W}_{\mr{aff}}$ on $(\hat{\mf{h}}^\tau_{\mb{Q}})^*$ is given by 
\begin{equation}\label{Weyl_action}
\varrho^{\mu}w(x)= w(x)-\langle w(x),K\rangle \iota(\mu)+\Big(\big( w(x)\st\iota(\mu)\big)-\frac{1}{2}\big(\iota(\mu)\st\iota(\mu)\big)\langle w(x),K\rangle\Big)\delta
\end{equation}
for any $\varrho^{\mu}w\in \tilde{W}_{\mr{aff}}$, $x\in (\hat{\mf{h}}^\tau_{\mb{Q}})^*$, cf.\,\cite[(6.5.2)]{Kac90}. 

Let $G'$ be the simply-connected cover of $G$ with the maximal torus $T'$ mapping to $T$. We identify $X_*(T')$ with the coroot lattice $\check{Q}$ of $G$. Let $\check{Q}_\tau$ be the lattice of $\tau$-coinvariants of $\check{Q}$. Then, $W_{\rm aff}:= \check{Q}_\tau \rtimes W^\tau$ is an affine Weyl group,  cf.\,\cite[A.1]{DH}. Let $\Omega$ denote the group of length zero elements in $\tilde{W}_{\mr{aff}}$. Then, $\Omega\simeq X_*(T)_\tau/\check{Q}_\tau$ and  $\tilde{W}_{\mr{aff}}=W_{\rm aff} \rtimes \Omega$. Moreover, the group $\Omega$ acts on $\hat{I}_\tau$. 

Let $w_0$ be the longest element in $W^\tau$. Any element in $\Omega$ can be written as 
\begin{equation}\label{529345}
\varsigma_\kappa: = \varrho^{-\kappa} w^\kappa, \quad w^\kappa:=w_\kappa w_0,\end{equation}
where $\kappa$ is 0 or a minuscule dominant element in $X_*(T)^+_\tau$ (cf.\,\cite[Section 2.3]{BH}), and $w_\kappa$ is the longest element in the stabilizer $W^\tau_{\kappa}$ of $\kappa$ in $W^\tau$, cf.\,\cite[Chapter 6, Section 2.3]{Bourbaki}.

For any nonempty subset $Y \subset \hat{I}_\tau$, let 
$\mc{P}_Y\subseteq G(\mc{K})^\tau$ be the parahoric subgroup associated to $Y$ as in Definition \ref{def_Py}, where $\mc{K}=\rk((z))$.
Let $\ms{G}_Y$ denote the corresponding parahoric group scheme over $\bar{\mc{O}}$. When $Y=\hat{I}_\tau$, $\mc{P}_Y$ is the Iwahori subgroup of $G(\mc{K})^\tau$, denoted by $\mc{I}$.

\begin{definition}\label{952819}
We define the partial affine flag variety $$\Fl_Y:=L\ms{G}_Y/L^+\ms{G}_Y,$$
where $L\ms{G}_Y$ (resp.\,$L^+\ms{G}_Y$) is the loop group scheme (resp.\,jet group scheme) of $\ms{G}_Y$. Then, $\Fl_Y(\rk)= G(\mc{K})^\tau/ \mc{P}_Y $. In particular, when $\tau$ is trivial and $Y=\{o\}$, $\Fl_Y$ is the affine Grassmannian of $G$, denoted by $\Gr_G$. Let $e_0$ denote the base point in $\Fl_Y$ and $\Gr_G$.
\end{definition}

Set $\breve{Y}:=\hat{I}_\tau \setminus Y $, and let $W_{\breve{Y}}$ denote the parabolic subgroup of $W_{\rm aff}$ generated by the set of simple reflections $\{ r_i\mid  i\in \breve{Y} \}$.  
\begin{definition}
For any $w\in \tilde{W}_\mr{aff}/W_{\breve{Y}}$, we define the Schubert cell $\Fl_{Y, w}:= \mc{I} \dot{w}e_0$. The affine Schubert variety $\overline{\Fl}_{Y,w}$ is the closure of $\Fl_{Y, w}$, where $e_0$ is the based point in $\Fl_Y$ and $\dot{w}$ is a lift of $w$ in $G(\mc{K})^\tau$. When $Y=\hat{I}_\tau$, we simply use the notations $\Fl_w$ and $\overline{\Fl}_w$.

For any dominant coweight $\mu\in X_*(T)^+$, we define the (spherical) affine Schubert variety $\Grb_{G,\mu}$ to be the closure of the orbit $G(\mc{O})\cdot e_\mu$ in $\Gr_G$, where $e_\mu:=z^{\mu} e_0\in \Gr_G$.
\end{definition}

Given any $\bar\mu\in X_*(T)_\tau$, we simply denote by $\Fl_{Y,\bar\mu}$ (resp.\,$\overline{\Fl}_{Y,\bar\mu}$) the Schubert cell $\Fl_{Y,\varrho^{\bar\mu}}$ (resp. Schubert variety \,$\overline{\Fl}_{Y,\varrho^{\bar\mu}}$) in $\Fl_Y$ associated to $\varrho^{\bar{\mu}}$. In fact, a lift of $\varrho^{\bar\mu}$ can be taken to be the following element in $T(\mc{K})^\tau$,
\begin{equation}\label{462918}
n^{\mu}:=\prod_{i=0}^{r-1}\tau^i(z^{\mu})\in T(\mc{K})^\tau,  
\end{equation}
where $\mu$ is a lifting of $\bar\mu$ in  $X_*(T)$. Then, $\Fl_{Y,\bar\mu}$ is exactly the orbit $\mc{I}\cdot e_{\bar{\mu}}$ in $\Fl_Y$, where $ e_{\bar{\mu}}:=n^\mu e_0$. We define the following union of Schubert varieties in $\Fl_Y$,
\begin{equation}
\label{adm_def}
\mc{A}_Y(\bar\mu):=\bigcup_{w\in W^\tau}\overline{\Fl}_{Y,w(\bar\mu)}. 
\end{equation}

We now recall the admissible set of affine Weyl group, cf.\,\cite[2.1.6]{Zhu14}. For any $\bar{\mu}\in X_*(T)_\tau$, define 
$$\mr{Adm}(\bar{\mu}):=\{w\in \tilde{W}_{\mr{aff}}\st w\preccurlyeq \varrho^{\eta} \text{ for some } \eta\in W^\tau \bar\mu\}.$$ 
Let $\mr{Adm}^Y(\bar{\mu})$ denote the subset $W_{\breve{Y}}\mr{Adm}(\bar{\mu})W_{\breve{Y}}\subset \tilde{W}_{\mr{aff}}$. For any $w\in \tilde{W}_{\mr{aff}}$, let $_Y\Fl_{Y,w}$ denote the $\mc{P}_Y$-orbit through $\dot{w}$ in $\Fl_Y$ and let $_Y\overline{\Fl}_{Y,w}$ be the closure of $_Y\Fl_{Y,w}$ in $\Fl_Y$.  In the following proposition, we show that  $\mc{A}_Y(\bar\mu)$ defined in (\ref{adm_def}) agrees with the definition in \cite[2.2.2]{Zhu14} in terms of admissible set.  
\begin{lem}\label{7842349}
For any $\bar{\mu}\in X_*(T)_\tau$,  we have 
$$\mc{A}_Y(\bar\mu)=\bigcup_{w\in \mr{Adm}^Y(\bar{\mu})}{}_Y\overline{\Fl}_{Y,w}.$$
\end{lem}
\begin{proof}
Denote the right side by $X_Y(\bar{\mu})$. For any $w\in W^\tau$, by definition, we have $\Fl_{Y,w(\bar\mu)}\subseteq {}_Y\Fl_{Y,\varrho^{w(\bar\mu)}}$. Hence, $\overline{\Fl}_{Y,w(\bar\mu)}\subseteq {}_Y\overline{\Fl}_{Y,\varrho^{w(\bar\mu)}}\subseteq X_Y(\bar{\mu})$. This shows that $\mc{A}_Y(\bar\mu)\subseteq X_Y(\bar{\mu})$. 

Conversely, for any $w\in \mr{Adm}^Y(\bar{\mu})$, we write $w=w_1 w' w_2$, where $w_1,w_2\in W_{\breve{Y}}$ and $w'\preccurlyeq \varrho^\eta$ for some $\eta\in W^\tau\bar\mu$. Then, ${}_Y\overline{\Fl}_{Y,w}={}_Y\overline{\Fl}_{Y,w'}\subseteq {}_Y\overline{\Fl}_{Y,\varrho^\eta}$. We have  
\begin{equation}\label{eq_union}
    \mc{P}_Y\varrho^\eta\mc{P}_Y=\bigcup_{y\in W_{\breve{Y}}}\mc{I}y\mc{I}\varrho^\eta \mc{P}_Y\overset{(*)}{\subseteq}\bigcup_{y\in W_{\breve{Y}}}\mc{I}y\varrho^\eta \mc{P}_Y=\bigcup_{y\in W_{\breve{Y}}}\mc{I}y\varrho^\eta y^{-1} \mc{P}_Y.
\end{equation}
The inclusion $(*)$ follows from the fact that $\mc{I}y\mc{I}=\mc{I}s_{i_1}\mc{I}\cdots \mc{I}s_{i_k}\mc{I}$  for reduced word $y=s_{i_1}\cdots s_{i_k}$ and $\mc{I}s_i\mc{I}\mc{I}\varrho^\eta\mc{I}\subseteq \mc{I}\varrho^\eta\mc{I}\cup \mc{I}s_i\varrho^\eta\mc{I}$ for any simple reflection $s_i\in W_{\rm aff}$, cf.\,\cite[Theorem 5.1.3(d)]{Ku}. For any $y\in W_{\breve{Y}}$, we can write $y=y'\varrho^\gamma$ for some $y'\in W^\tau, \gamma\in X_*(T)_\tau$. Then, $y\varrho^\eta y^{-1}=y'\varrho^\eta {y'}^{-1}=\varrho^{y'(\eta)}$ and we have 
\[\mc{I}y\varrho^\eta y^{-1} \mc{P}_Y=\mc{I}\varrho^{y'(\eta)}\mc{P}_Y\subseteq \bigcup_{v \in W^\tau}\mc{I}\varrho^{v(\eta)}\mc{P}_Y.\]
Hence, $\mc{P}_Y\varrho^\eta\mc{P}_Y\subseteq \bigcup_{v\in W^\tau}\mc{I}\varrho^{v(\eta)}\mc{P}_Y$. Therefore, ${}_Y\overline{\Fl}_{Y,w}\subseteq {}_Y\overline{\Fl}_{Y,\varrho^\eta}\subseteq \bigcup_{v\in W^\tau}{}_Y\overline{\Fl}_{Y,\varrho^{v(\eta)}}=\mc{A}_Y(\bar\mu)$. This shows that $X_Y(\bar{\mu})\subseteq \mc{A}_Y(\bar\mu)$.
\end{proof}

For any $\varsigma_\kappa\in \Omega$, let $\Fl_Y^{\kappa}$ denote the connected component of $\Fl_Y$ containing $\varsigma_{\kappa}$ (equivalently $\varrho^{-\kappa}$). Then, $\Fl_Y^{\kappa}$ is a homogeneous space of $G'(\mc{K})^\tau$, and
\begin{equation}\label{eq_sckappa}
\Fl_Y^{\kappa}\simeq  G'(\mc{K})^\tau/ \mc{P}'_{Y_\kappa}=:\Fl_{Y_\kappa}',
\end{equation}
where $\mc{P}'_{Y_\kappa}$ is the parahoric subgroup of $G'(\mc{K})^\tau$ associated to $Y_\kappa:=\{ \varsigma_\kappa(i)  \mid  i\in Y \}$. By \cite[Section 10]{PR08}, we have the following isomorphism,
\begin{equation}\label{5324798}
\mr{Pic}(\Fl^{\kappa}_Y)\simeq \bigoplus_{i\in  Y_\kappa} \mb{Z}\Lambda_i.
\end{equation} 
Given any $\bar\mu\in X_*(T)_\tau$, let $\varsigma_\kappa\in \Omega$ be the unique element such that $\varrho^{\bar\mu}\varsigma_{\kappa}^{-1}\in W_{\mr{aff}}$, equivalently 
\begin{equation}
\label{eq_kappa}
\bar\mu+\kappa\in \check{Q}_\tau.\end{equation} 
Then, for any $w\in W^\tau$, we have $w(\bar\mu)+\kappa\in \check{Q}_\tau$, equivalently $\varrho^{w(\bar\mu)}\varsigma_{\kappa}^{-1}\in W_{\mr{aff}}$. Hence, $\Fl_{Y,w(\bar\mu)}\subseteq \Fl_Y^\kappa$ for any $w\in W^\tau$. Let $\dot{\varsigma}_\kappa$ be a lifting of $\varsigma_\kappa$ in $G(\mc{K})^\tau$. Consider the translation 
\begin{equation}\label{eq_translation}
\dot{\varsigma}_\kappa:\Fl_Y^\circ\to \Fl_Y^\kappa
\end{equation}
given by $x\mapsto \dot{\varsigma}_\kappa x$. For any $w\in W^\tau$, the orbit $\mc{I}(\dot{\varsigma}_\kappa)^{-1} \varrho^{w(\bar\mu)}e_0$ maps to the orbit $\mc{I} \varrho^{w(\bar\mu)}e_0$. Thus, we have an isomorphism $\overline{\Fl}_{Y,\varsigma_\kappa^{-1} \varrho^{w(\bar\mu)}}\simeq \overline{\Fl}_{Y,w(\bar\mu)}$ under the translation map. 

Consider the conjugation map $\mr{Ad}_{\dot{\varsigma}_\kappa}:G'(\mc{K})^\tau\to G'(\mc{K})^\tau$ given by $x\mapsto \dot{\varsigma}_\kappa x (\dot{\varsigma}_\kappa)^{-1}$. It induces an isomorphism $\mr{Ad}_{\dot{\varsigma}_\kappa}: \Fl_Y'\to \Fl_{Y_\kappa}'$. Then, \eqref{eq_sckappa} is the composition of the following isomorphisms 
\begin{equation}
\label{eq_FL_comp}
\Fl_Y^{\kappa}\xrightarrow{(\dot{\varsigma}_\kappa)^{-1}}  \Fl_Y^\circ \simeq \Fl_Y'\xrightarrow{\mr{Ad}_{\dot{\varsigma}_\kappa}} \Fl_{Y_\kappa}'.\end{equation} 
Under these isomorphisms, we have 
$$\overline{\Fl}_{Y,w(\bar\mu)}\simeq \overline{\Fl}_{Y,\varsigma_\kappa^{-1} \varrho^{w(\bar\mu)}}\simeq \overline{\Fl}'_{Y,\varsigma_\kappa^{-1} \varrho^{w(\bar\mu)}}\simeq \overline{\Fl}'_{Y_\kappa, \varrho^{w(\bar\mu)}\varsigma_\kappa^{-1}}.$$
Therefore, \eqref{eq_sckappa} restricts to an isomorphism
\begin{equation}\label{523012}
\overline{\Fl}_{Y,w(\bar\mu)}\simeq \overline{\Fl}'_{Y_\kappa, \varrho^{w(\bar\mu)}\varsigma_\kappa^{-1}}.
\end{equation}
Hence, the variety $\mc{A}_Y(\bar\mu)$ maps to $\bigcup_{w\in W^\tau}\overline{\Fl}'_{Y_\kappa, \varrho^{w(\bar\mu)}\varsigma_\kappa^{-1}}$ under the map \eqref{eq_sckappa}. 

Let $\Lambda=\sum_{i\in  Y} n_i \Lambda_i$ with $n_i>0$. Let $c:=\Lambda(K)$ be the level of $\Lambda$ and let $\lambda=\sum_{i\in Y} n_i\omega_i$ be the restriction of $\Lambda$ to $(T^{\tau,\circ})'$. Let $\ms{L}(\Lambda)$ be the line bundle on $\Fl_{Y}^\circ\simeq \Fl'_Y$ associated to the weight $\Lambda$. Let $\ms{L}(\Lambda)_\kappa$ be the line bundle on $\Fl_{Y}^\kappa$ which is the push-forward of $\ms{L}(\Lambda)$ via the translation map $\dot{\varsigma}_\kappa:\Fl_{Y}^\circ\to \Fl_{Y}^\kappa$. Regarded as a line bundle on $\Fl'_{Y_\kappa}$ via the identification \eqref{eq_FL_comp}, $\ms{L}(\Lambda)_\kappa$ is the line bundle associated to the affine dominant weight $\varsigma_\kappa(\Lambda)$. 

Suppose that $p\nmid a_{Y}$, where $a_{Y}$ is defined in (\ref{index_aY}) for the subset $Y\subset \hat{I}_\tau$. 
By Lemma \ref{prop_goodprime}, we can choose an interior point $\vartheta$ in the facet associated to $Y$ so that $p\nmid m$. Let $G_\vartheta$ be the subgroup  of $G^{\tau,\circ}$ associated to $\vartheta$ defined in \eqref{eq_Glambda} of Section \ref{section_flag}. Let $G^\tau{}'$ denote the simply-connected cover of $G^{\tau,\circ}$. Let $G'_\vartheta$ (resp.\,$T^\tau{}'$) be the pre-image of $G_\vartheta$ (resp.\,$T^{\tau,\circ}$) via the simply-connected cover map $G^\tau{}' \to G^{\tau,\circ}$. 

Set $G(\mc{O})^{\tau}{}':= G(\mc{O})^{\tau, \circ}\times_{G^{\tau,\circ} }G^\tau{}'$. In Proposition \ref{thm_Gtau}, we will show that the line bundle $L$ of central charge one on $\Gr_G$ has a unique $G'(\mc{O})$-equivariant structure and the line bundle $\ms{L}(\Lambda)_\kappa$ on $\Fl^\kappa_Y$ has a unique $G(\mc{O})^{\tau}{}'$-equivariant structure. These equivariant structures induce   $G'_\vartheta$-equivariant structures on $L$ and $\ms{L}(\Lambda)_\kappa$, via the composition of natural morphisms 
$$G'_\vartheta\hookrightarrow G^\tau{}' \to G'\hookrightarrow G'(\mc{O}),$$
and the composition map $G'_\vartheta\hookrightarrow G^\tau{}'\to G(\mc{O})^{\tau}{}'$ respectively.

We now state our main result of this section, which will be proved in the rest of the section.

\begin{thm}\label{thm_mainthm}
Let $\mu\in X_*(T)^+$. Let $\Fl_Y^\kappa$ be the component containing $\overline{\Fl}_{Y,\bar{\mu}}$  (equivalently $\varrho^{\bar\mu}\varsigma_{\kappa}^{-1}\in W_{\mr{aff}}$). Let $\ms{L}(\Lambda)_\kappa$ be the line bundle  on $\Fl_Y^\kappa$ of central charge $c$ defined as above.
\begin{enumerate}
\item Assume that $p\nmid a_{Y}$ and $p \nmid |\pi_1(G)|= |X_*(T)/\check{Q}|$. Then, we have an isomorphism of $T^{\tau}{}'$-representations: 
\[H^0(\mc{A}_Y(\bar\mu),\ms{L}(\Lambda)_\kappa)\simeq H^0(\Grb_{G,\mu},L^c)\otimes \rk_{-\lambda},\]
where $\rk_{-\lambda}$ is the 1-dimensional representation of $T^{\tau}{}'$.
\item Assume $p=0$. Then, we have an isomorphism of $G'_\vartheta$-representations:
\[H^0(\mc{A}_Y(\bar\mu),\ms{L}(\Lambda)_\kappa)\simeq H^0(\Grb_{G,\mu},L^c)\otimes \rk_{-\lambda}.\]
\end{enumerate}
\end{thm}

\subsection{Equivariance of line bundles on partial affine flag varieties}\label{sec_equivariantbundle}
In this subsection, we assume that $G$ is simply-connected.

As in Section \ref{section_flag}, we fix a non-empty subset $Y\subseteq \hat{I}_\tau$ and assume $p\nmid a_Y$. By Lemma \ref{prop_goodprime}, we can choose an interior point $\vartheta$ in the facet associated to $Y$ so that $p\nmid m$, where $m$ is the minimal positive multiple of $r$ such that $h:=\frac{m}{r}\vartheta\in X_*(T_{\mr{ad}})^\tau$. Define an automorphism $\sigma:=\tau\circ \mr{Ad}_{(\epsilon_m)^{m\vartheta/r}}$ of $G$, whose order is $m$, see Lemma \ref{Order_lemma}. 

Recall the groups $\mc{P}_Y^-$ and $\mc{I}^-$ in Definition \ref{def_Py} and the parahoric group scheme $\ms{G}_{-\vartheta,\infty}$ over $\bar{\mc{O}}_\infty$ defined in Section \ref{sec_parahoric}. 
\begin{prop}\label{prop_3284932}
The isomorphism  $\mr{Ad}_{t^{h}}:G[z,z^{-1}]^\tau\simeq G[t,t^{-1}]^\sigma$  restricts to the following isomorphism 
\begin{equation}\label{eq_Pminus}
\mr{Ad}_{t^h}:\mc{P}_Y^-\simeq G[t^{-1}]^{\sigma}.
\end{equation}
As a consequence, $\mc{P}_Y^-$ has an integral ind-group scheme structure.
\end{prop}
\begin{proof}
It is clear that the group $G[t^{-1}]^\sigma$ admits a natural ind-group scheme structure from $L^-\ms{G}_\sigma$, i.e. $L^-\ms{G}_\sigma (\rk)=G[t^{-1}]^\sigma$. Define the ind-group scheme $L^{--} \ms{G}_\sigma$ as follows
\[ L^{--} \ms{G}_\sigma:=\ker ({\rm ev}_\infty:   L^-\ms{G}_\sigma  \to G^\sigma ) .\]
By Theorem \ref{paraohoric_thm} and \cite[Theorem 0.1, Theorem 0.2]{PR08},  the twisted affine Grassmannian $\Gr_{\ms{G}_\sigma} :=G((t))^\sigma/ G[[t]]^\sigma$ is reduced and irreducible. Moreover, by \cite[Lemma 3.1]{HR21} the morphism $ L^{--}\ms{G}_\sigma \to  \Gr_{\ms{G}_\sigma}$ given by $g\mapsto g\cdot e$ is an open embedding. Thus, $L^{--} \ms{G}_\sigma$ is integral. Since $G$ is simply-connected, $G^\sigma$ is connected. Therefore, $G[t^{-1}]^\sigma$ is integral.  

It remains to show that \eqref{eq_Pminus} is an isomorphism. We first prove that 
$$\mc{P}_Y^-=\ms{G}_{-\vartheta,\infty}(\bar{\mc{O}}_\infty)\cap G[z,z^{-1}]^\tau.$$ 
The inclusion $\mc{P}_Y^-\subseteq \ms{G}_{-\vartheta,\infty}(\bar{\mc{O}}_\infty)\cap G[z,z^{-1}]^\tau$ follows from Definition directly. Recall that there is a negative Bruhat decomposition (cf.\cite{MPS})
$$G[z,z^{-1}]^\tau=\bigsqcup_{w\in W_{\mr{aff}}} \mc{I}^-\dot{w}\mc{I}^-.$$ 
For any $g\in \ms{G}_{-\vartheta,\infty}(\bar{\mc{O}}_\infty)\cap G[z,z^{-1}]^\tau$, we write $g=g_1\dot{w}g_2$, where $g_1,g_2\in \mc{I}^-$. Note that $\mc{I}^-\subseteq \mc{P}_Y^-\subseteq \ms{G}_{-\vartheta,\infty}(\bar{\mc{O}}_\infty)$. Hence, $\dot{w}\in \ms{G}_{-\vartheta,\infty}(\bar{\mc{O}}_\infty)$. It follows that 
$$w(-\vartheta)=-\vartheta.$$
Note that the stabilizer of $-\vartheta$ in $W_{\mr{aff}}$ is generated by the reflections in the walls of $-\textsf{A}$ containing $-\vartheta$. Let $\psi_1,\ldots,\psi_i$ be the affine roots defining those walls and let $s_{\psi_1},\ldots,s_{\psi_i}$ denote the reflections in the corresponding walls. Then, 
$$w\in \langle s_{\psi_1},\ldots,s_{\psi_i} \rangle.$$ 
For these affine roots $\psi$'s, we have  $\psi(-\vartheta)=0$. Hence, both affine root subgroups $U_{\psi},U_{-\psi}$ belong to $\mc{P}_Y^-$. Then, the standard representative of the corresponding reflection $s_\psi$ lies in $\langle U_{\psi},U_{-\psi} \rangle \subseteq \mc{P}_Y^-$. Therefore, the representative $\dot{w}$ of $w$ can be chosen to be in $\mc{P}_Y^-$. This shows that $g\in \mc{P}_Y^-$. Hence, $$\mc{P}_Y^-=\ms{G}_{-\vartheta,\infty}(\bar{\mc{O}}_\infty)\cap G[z,z^{-1}]^\tau.$$
Now, by isomorphism \eqref{eq_phi2}, we have 
$$\mr{Ad}_{t^h}(\mc{P}_Y^-)=G[[t^{-1}]]^{\sigma}\cap G[t,t^{-1}]^\sigma=G[t^{-1}]^\sigma.$$
\end{proof}

\begin{remark}
In the above proposition, the ind-group scheme structure on $\mc{P}_Y^-$ relies on the condition $p\not \mid  a_Y$. It is possible that this condition can be removed. However, we don't know if there is an appropriate reference. 
\end{remark}
Recall that $M_Y=\mr{Ad}_{t^{-h}}(G^\sigma)$. As a corollary of Theorem \ref{paraohoric_thm} and Proposition \ref{prop_3284932}, we have
\begin{cor}
We have $M_Y=
\mc{P}_Y\cap\mc{P}_Y^{-}$, and $M_Y$ is a reductive group.\qed
\end{cor}

Consider the full affine flag variety $\Fl=G(\mc{K})^\tau/ \mc{I}$ and the natural morphism $p_Y:\Fl\to \Fl_Y.$
Let $\mb{P}^1_j$ denote the image of the morphism $\mb{P}^1\simeq \mc{P}_{\hat{I}_\tau\setminus\{j\}}/\mc{I}\to \Fl$. For each $i\in Y$, let $\ms{L}_i$ denote the line bundle $\ms{L}(\Lambda_i)$ on $\Fl_Y$.  Then, we have 
$$\deg\Big(p_Y^*(\ms{L}_i)\big|_{\mb{P}^1_j}\Big)=\delta_{i,j}$$ 
for any $i\in Y$ and $j\in \hat{I}_\tau$, cf.\,\cite[Section 10]{PR08}. 

Let $e_0$ denote the base point of $\Fl_Y$. 
We have the following decompositions 
\begin{equation}\label{eq_birkhoff}
\Fl_Y=\bigsqcup_{w\in  W_{\mr{aff}}/W_{\breve{Y}}} \mc{I} \dot{w} e_0,\quad \Fl_Y=\bigsqcup_{w\in  W_{\mr{aff}}/W_{\breve{Y}}} \mc{I}^- \dot{w} e_0,
\end{equation}
where $\dot{w}$ is a lift of $w$ in $G(\mc{K})^\tau$. In fact, $\Fl_Y$ can be identified with the Kac-Moody flag variety $\Fl_Y^{\mr{KM}}$ in the sense of \cite{Ku87} and \cite{Ma88}, cf.\,\cite[\S 9.f]{PR08}. Thus, the decompositions in \eqref{eq_birkhoff} arise from the decompositions for $\Fl_Y^{\mr{KM}}$ in \cite{Ku87} and \cite{Ma88}.

For any $w\in W_{\rm aff}$, let $\Fl_Y^w$ denote the $\mc{I}^-$-orbit $\mc{I}^-we_0$ in $\Fl_Y$ and let $\overline{\Fl}_Y^w$ denote the closure of $\Fl_Y^w$. When $Y=\hat{I}_\tau$, we simply use the notations $\Fl^w$ and $\overline{\Fl}^w$.
\begin{lem}\label{lem_oppositeorder}
$\mc{I}^- y e_0 \subseteq \overline{\Fl}_Y^w$ if and only if $w_{\rm min}\preccurlyeq  y_{\rm min}$, where $w_{\rm min}$ represents the minimal element in the right coset $w  W_{\breve{Y}}$ (similarly for $y_{\rm min}$).  

\end{lem}
\begin{proof}
When char($\rk$)=0, it follows from \cite[Prop.7.1.21]{Ku}. For general characteristic, it follows from \cite[Theorem 1.26, Lemma 2.9]{Kato20}.
\end{proof}
\begin{lem}\label{45395}
$\overline{\Fl}_Y^w\cap \Fl_{Y,y}\neq \emptyset$ if and only if $w_{\rm min}\preccurlyeq  y_{\rm min}$. 
\end{lem}
\begin{proof}
Assume $w_{\rm min}\preccurlyeq  y_{\rm min}$. By Lemma \ref{lem_oppositeorder}, we have $\mc{I}^- y e_0 \subseteq \overline{\Fl}_Y^w$. Hence, $\overline{\Fl}_Y^w\cap \Fl_{Y,y}\supseteq \mc{I}^- y e_0\cap \Fl_{Y,y}\ni ye_0 $. Therefore, $\overline{\Fl}_Y^w\cap \Fl_{Y,y}\neq \emptyset$. Conversely, assume $\overline{\Fl}_Y^w\cap \Fl_{Y,y}\neq \emptyset$, then $\overline{\Fl}_Y^w\cap \overline{\Fl}_{Y,y}\neq \emptyset$. Hence, the $T^\tau$-fixed point of the projective variety $\overline{\Fl}_Y^w\cap \overline{\Fl}_{Y,y}$ is nonempty. Note that the $T^\tau$-fixed points of $\Fl_Y$ are taking the form $xe_0$, where $x\in W_{\mr{aff}}$. It follows that $w_{\rm min}\preccurlyeq x_{\rm min} \preccurlyeq  y_{\rm min}$ for some $x\in W_{\mr{aff}}$.
\end{proof}

Define $^{\breve{Y}}W^{\breve{Y}}$ to be the set of minimal length coset representatives of $W_{\breve{Y}}\backslash W_{\mr{aff}}/W_{\breve{Y}}$. Let $R\subseteq {}^{\breve{Y}}W^{\breve{Y}}$ be a finite subset so that $w\in R$ and $w'\preccurlyeq w$ imply $w'\in R$. Set 
$$S:=W_{\breve{Y}}\cdot R\cdot W_{\breve{Y}}\subseteq W_{\mr{aff}}.$$ 
Then, $w\in S$ and $w'\preccurlyeq w$ implies $w'\in S$.

We define the following open subset of $\Fl_Y$,
$$V^S:=\bigcup_{w\in S} w\mc{N}^-e_0,$$
where  $\mc{N}^-$ is the unipotent radical of $\mc{I}^-$. 
\begin{lem}
We have $ V^S=\bigcup_{w\in S} \mc{I}^-we_0$. As a consequence, $V^S$ is $\mc{P}^-_Y$-stable.
\end{lem}
\begin{proof}
We write $\mc{N}^-V^S=\bigcup_{y\in S'} \mc{N}^-ye_0$ for some $S'\subseteq W_{\mr{aff}}$. Clearly, $S\subseteq S'$.
Note that  
$$\mc{N}^-V^S= \bigcup_{w\in S} \mc{N}^-w\mc{N}^-e_0= \bigcup_{w\in S} \mc{N}^-(\mc{N}^+\cap w\mc{N}^-w^{-1})we_0=\bigcup_{w\in S} \mc{N}^-\mc{I}we_0,$$
where $\mc{N}^+$ is the unipotent radical of $\mc{I}$. Then, for each $y\in S'$, there exists $w\in S$ such that $\mc{N}^-ye_0\cap \mc{N}^-\mc{I}we_0\neq \emptyset$. This implies that $\mc{N}^-ye_0\cap \mc{I}we_0\neq \emptyset$. By Lemma \ref{45395}, we get $y_{\rm min}\preccurlyeq  w_{\rm min}$. Since $S$ is saturated and stable under right action of $W_{\breve{Y}}$, we must have $y\in S$. Therefore, $\mc{N}^-V^S=\bigcup_{y\in S'} \mc{N}^-ye_0\subseteq \bigcup_{y\in S} \mc{N}^-ye_0$. It follows that $V^S\subseteq \mc{N}^-V^S\subseteq \bigcup_{y\in S} \mc{N}^-ye_0$. Conversely,  we have 
$$\mc{N}^-ye_0=(\mc{N}^-\cap y\mc{N}^-y^{-1})ye_0\subseteq y\mc{N}^-e_0.$$
Hence, $\bigcup_{y\in S} \mc{N}^-ye_0\subseteq V^S$. Therefore, $ V^S=\bigcup_{y\in S} \mc{N}^-ye_0=\bigcup_{y\in S} \mc{I}^-ye_0$.
\end{proof}
Set $\overline{\Fl}^w_Y(S):=\overline{\Fl}_Y^w\cap V^S$. In particular, when $Y=\hat{I}_\tau$, we simply denote it by $\overline{\Fl}^w(S)$. For each $i\in Y$, let $\ms{L}_i$ denote the line bundle $\ms{L}(\Lambda_i)$ on $\Fl_Y$. 

\begin{prop}\label{thm_lgequivarian}
For any $i\in Y$, the line bundle $\ms{L}_i$ has a unique $\mc{P}^-_Y$-equivariant structure with $M_Y$ acting trivially on the fiber $\ms{L}_i|_{e_0}$ at the base point. Moreover, the line bundle $\ms{L}_i$ has a unique $\mc{P}_Y$-equivariant structure with $M_Y$ acting trivially on $\ms{L}_i|_{e_0}$.
\end{prop}
\begin{proof}
Following the idea of \cite[Theorem 7]{Faltings03}, we shall construct a $\mc{P}^-_Y$-equivariant
line bundle $L_i$ on $\Fl_Y$, and then show that $L_i$ is isomorphic to $\ms{L}_i$.

Given any $i\in Y$, choose $R$ as above so that $s_i\in R$. Then, we have $W_{\breve{Y}}s_i\subseteq S$.  Set $D:=\overline{\Fl}^{s_i}_Y(S)$. Then,  $\overline{\Fl}^{w}_Y(S)\subseteq D$ for any $w\in W_{\breve{Y}}s_i$ since $w\succcurlyeq s_i$.  Thus, $D$ is $\mc{P}^-_Y$-stable. 

For each $w\in S$, we define a subgroup $\mc{N}_w^-:=\mc{N}^-\cap w\mc{N}^-w^{-1}$. Let $K^-_n$ be the $n$-th congruence subgroup of $G[t^{-1}]^\sigma$. Let $\mc{N}^-(n)$ be the associated subgroup in $\mc{P}^-_Y$ via the isomorphism \eqref{eq_Pminus}. Choose $n$ big enough so that $\mc{N}^-(n)\subseteq \mc{N}_w^-$ for all $w\in S$. By the same argument of \cite[Lemma 6]{Faltings03}, $\mc{N}^-(n)\backslash V^S$ is covered by finitely many open subsets, which are finite-dimensional affine spaces of the same dimension. Thus, $\mc{N}^-(n)\backslash V^S$ is a smooth variety.

Recall that $D$ is $\mc{P}^-_Y$-stable and of codimension  $1$ in $V^S$. Thus, $\mc{N}^-(n)\backslash D$ is a $\mc{N}^-(n)\backslash\mc{P}^-_Y$-stable divisor of $\mc{N}^-(n)\backslash V^S$. This defines an $\mc{N}^-(n)\backslash\mc{P}^-_Y$-equivariant line bundle $L_{i,S}$ on $\mc{N}^-(n)\backslash V^S$ and $\mc{N}^-(n)\backslash M_Y$ acts trivially on the canonical section $1\in \Gamma(\mc{N}^-(n)\backslash V^S, L_{i,S})$. Thus, the action of $\mc{N}^-(n)\backslash M_Y$ on $L_{i,S}|_{e_0}$ is trivial. By taking $S$ arbitrarily large, we get a line bundle  on $\Fl_Y$, denoted by $L_i$. From construction, this line bundle $L_i$ has a $\mc{P}^-_Y$-equivariant structure such that $M_Y$ acts on $L_i|_{e_0}$ trivially.

Consider the natural morphism $p_Y:\Fl\to \Fl_Y$. We have $p_Y^{-1}(\overline{\Fl}^{s_i}_Y(S))=\overline{\Fl}^{s_i}(S)$.  Thus, the degree of  the restriction of the pull-back $p_Y^*(L_i)|_{\mb{P}^1_j}$ equals the degree of the divisor $\overline{\Fl}^{s_i}(S)\cap \mb{P}^1_j$ in $\mb{P}^1_j$, and equals  the intersection number $\#\{\overline{\Fl}^{s_i}(S)\cap \mb{P}^1_j\}$ since $\overline{\Fl}^{s_i}(S)\cap \mb{P}^1_j$ intersects transversely. By Lemma \ref{45395}, $\overline{\Fl}^{s_i}\cap \mb{P}^1_j$ is empty if $j\neq i$. When $j=i$, from the proof of Lemma \ref{45395}, we have $\overline{\Fl}^{s_i}\cap \mb{P}^1_i=\mc{I}^-s_ie_0\cap \mc{I}s_ie_0\subsetneqq \mb{P}^1_i$. Hence, every point in $\overline{\Fl}^{s_i}\cap \mb{P}^1_i$ is fixed by $T^\tau$. Thus, $\overline{\Fl}^{s_i}\cap \mb{P}^1_i=\{s_ie_0\}$.  Therefore, $\deg\big(p_Y^*(L_i)|_{\mb{P}^1_j}\big)$ equals $1$ if $j=i$ and equals $0$ if $j\neq i$. This shows that $L_i\simeq \ms{L}_i$ as line bundles on $\Fl_Y$. Therefore, $\ms{L}_i$ has a $\mc{P}^-_Y$-equivariant structure such that $M_Y$ acts on $\ms{L}_i|_{e_0}$ trivially. This equivariant structure is unique since the character of $M_Y$ uniquely extends to a character of $\mc{P}^-_Y$. This proves the first statement.

We now prove the second statement. 
Let $K_n$ be the $n$-th congruence subgroup in $G(\mc{O} )^\tau$, i.e. whose elements are identity modulo $z^n$.  Let $\mathcal{P}_{Y,n}$ be the quotient group $\mathcal{P}_{Y}/ K_n$. 
For any $w\in W_{\rm aff}$, we consider the $\mathcal{P}_Y$-Schubert variety $_Y\overline{\Fl}_{Y,w}$.  It is well-known that when $n$ is sufficiently large,  the action of $\mathcal{P}_{Y}$ factors through $\mathcal{P}_{Y,n}$. Consider the restriction of $\ms{L}_i$ to $_Y\overline{\Fl}_{Y,w}$, which will still be denoted by $\ms{L}_i$. Let $\widehat{\mc{P}_{Y,n}}$ be the $\mb{G}_\m$-central extension of $\mc{P}_{Y,n}$ associated to the line bundle $\ms{L}_i$ on $_Y\overline{\Fl}_{Y,w}$.  Note that $\mc{P}_{Y,n}$ is the semidirect product of $M_{Y}$ with a unipotent subgroup $U_{Y,n}$. By the first part of the proposition, we have a splitting $M_{Y}\to \widehat{\mc{P}_{Y,n}}$ over $M_{Y}$. On the other hand, by \cite[Corollary 5.7]{Co08}, the group $\mr{Pic}(U_{Y,n})$ classifies the isomorphism classes of extensions of $U_{Y,n}$ by $\mb{G}_m$. Moreover, by \cite[Theorem 1]{Popov}, the Picard group $\mr{Pic}(U_{Y,n})$ of unipotent group is trivial. Hence, $U_{Y,n}$ only has trivial central extension by $\mb{G}_m$. Thus, we have a unique splitting $U_{Y,n}\to \widehat{\mc{P}_{Y,n}}$ over $U_{Y,n}$. These splittings uniquely glue to a splitting $\mc{P}_{Y,n}\to \widehat{\mc{P}_{Y,n}}$. Therefore, the line bundle $\ms{L}_i$ on $_Y\overline{\Fl}_{Y,w}$ is $\mc{P}_{Y,n}$-equivariant. Taking a limit of $_Y\overline{\Fl}_{Y,w}$ with respect to the Bruhat order on $W_{\rm aff}$, we conclude that the line bundle $\ms{L}_i$ on $\Fl_{Y}$ is $\mc{P}_{Y}$-equivariant. From our argument, the induced $M_Y$-action on  $\ms{L}_i|_{e_0}$ is trivial. Finally, the equivariant structure with the desired property is unique since the character of $M_Y$ uniquely extends to a character of $\mc{P}_Y$.
\end{proof}

Let $G^\tau{}'$ denote the simply-connected cover of $G^\tau$ and set $G(\mc{O})^\tau{}':=G(\mc{O})^\tau\times_{G^\tau} G^\tau{}'$. Let $T^\tau{}'$ be the preimage of $T^\tau$ under the map $ G^\tau{}'\to G^\tau$. Note that when $G$ is simply-connected, $G^\tau$ is connected but not necessarily simply-connected. Let $\mb{G}_m^{\mr{rot}}$ be the loop rotation torus acting on $G(\mc{K})^\tau$. 
By \cite{Faltings03,PR08},  the full affine flag variety $\Fl$ of $G(\mc{K})^\tau$
can be identified with the associated Kac-Moody flag variety.  As a consequence, the central extension of $G(\mc{K})^\tau\rtimes  \mb{G}_m^{\mr{rot}}$
$$1\to \mb{G}_m \to \tilde{L}(G,\tau)  \to G(\mc{K})^\tau \rtimes \mb{G}_m^{\mr{rot}}  \to 1 $$
associated to any dominant line bundle of central charge 1 on $\Fl$,  is isomorphic to the affine Kac-Moody group of type $X_N^{(r)}$ in the sense of Kumar and Mathieu. Then, $\tilde{L}(G,\tau)$ acts on any line bundle $\ms{L}$ on $\Fl_Y$.

Let $\tilde{T}$ be the preimage of $T^\tau \times \mb{G}_m^{\mr{rot}}$ in $\tilde{L}(G,\tau)$. This induces an exact sequence, 
\[ \textstyle 0\to X^*(T^\tau)\oplus X^*(\mb{G}_m^{\mr{rot}}) \to X^*(\tilde{T})  \xrightarrow{c} X^*(\mb{G}_m)\simeq \mb{Z}\to 0 , \]
where $c$ is the central charge map.
According to our convention in Section \ref{Notations}, the fundamental weight $\Lambda_o\in  X^*(\tilde{T})$ has central charge $\check{a}_o$, where $\check{a}_o=2$ if $X_N^{(r)}= A_{2\ell}^{(2)}$; otherwise $\check{a}_o=1$. Via the line bundle $\ms{L}(\Lambda_0)$, there is a natural embedding ${T^\tau}'\times \mb{G}^{\rm rot}\hookrightarrow \tilde{T}$, such that composing with $\tilde{T}\to T^\tau\times \mb{G}^{\rm rot}_{m}$ is the natural morphism ${T^\tau}'\times \mb{G}^{\rm rot}_{m}\to {T^\tau}\times \mb{G}^{\rm rot}_{m}$.
This induces the following embedding
$$\textstyle X^*(\tilde{T})  \hookrightarrow \frac{1}{\check{a}_0} \mb{Z}\Lambda_o \oplus X^*(T^\tau{}')\oplus \mb{Z} \delta\subseteq (\hat{\mf{h}}^\tau_\mb{Q})^*,$$ 
where we identify $\delta$ with the canonical generator of $ X^*(\mb{G}_m^{\mr{rot}})$.  Note that ${T^\tau}'=T^\tau$ when $\check{a}_o=1$, otherwise ${T^\tau}'\to T^\tau$ is a double cover. 
\begin{lem}\label{4237984}
Let $\ms{L}(\Lambda)$ be a line bundle on $\Fl_Y$ associated to $\Lambda=\sum_{i\in  Y} n_i \Lambda_i$ of central charge $c$. Let $\lambda$ be the restriction of $\Lambda$ to $T^\tau{}'$. Then, the induced $\tilde{T}$-action on $\ms{L}(\Lambda)|_{\dot{w}e_0}$ is given by the weight $-w(\Lambda)$, where $w\in W_{\mr{aff}}$ and $\dot{w}$ is a lifting of $w$ in $G(\mc{K})^\tau$. As a consequence, the induced $T^\tau{}'$-action on $\ms{L}(\Lambda)|_{\dot{w}e_0}$ is given by the weight $-w(\Lambda)|_{T^\tau{}'}$.  
\end{lem}
\begin{proof}
Recall that $\Fl_Y$ is a Kac-Moody flag variety of $\tilde{L}(G,\tau)$ and $W_{\mr{aff}}$ is the Weyl group associated to $\tilde{L}(G,\tau)$. Thus, we have that the $\tilde{T}$-action on $\ms{L}(\Lambda)|_{e_0}$ is given by the weight $-\Lambda$; generally at the point $\dot{w}e_0$ for any $w\in W_{\mr{aff}}$,  $\tilde{T}$ acts on $\ms{L}(\Lambda)|_{\dot{w}e_0}$ via the weight $-w(\Lambda)$, where the action $w$ on $\Lambda$ is given in \eqref{Weyl_action}. 
\end{proof}

Recall Lemma \ref{prop_goodprime}, when $p\nmid a_{Y}$, we can choose an interior point $\vartheta$ in the facet associated to $Y$ so that $p\nmid m$. Define a subgroup $G_\vartheta$ of $G^{\tau}$ as in \eqref{eq_Glambda}, which is a subgroup of $M_Y$. Let $G'_\vartheta$ be the pre-image of $G_\vartheta$ via the cover $G^{\tau}{}'\to G^{\tau}$. Recall Corollary \ref{cor_glambdamy}, $G_\vartheta$ is a subgroup of $M_Y$. Thus, we have a natural morphism $G'_\vartheta\to M_Y$. 

By Proposition \ref{thm_lgequivarian}, the line bundle $\ms{L}(\Lambda)$ has a unique $\mc{P}_Y$-equivariant structure with $\mc{P}_Y$ acting on $\ms{L}(\Lambda)|_{e_0}$ trivially. In the following proposition, we compare it with the unique $G(\mc{O})^\tau{}'$-equivariant structure.

\begin{prop}\label{thm_Gtau}
With the same assumptions in Lemma \ref{4237984}, the line bundle $\ms{L}(\Lambda)$ has a unique $G(\mc{O})^\tau{}'$-equivariant structure. Moreover, the $G'_\vartheta$-actions on $\ms{L}(\Lambda)$ induced from the equivariant structures of $G(\mc{O})^\tau{}'$ and $\mc{P}_Y$ differ by a weight $\lambda$.
\end{prop}
\begin{proof}

Similar to the argument in Proposition \ref{thm_lgequivarian}, we consider any $G(\mc{O})^\tau$-Schubert variety $X_w$ in $\Fl_Y$. Then, the action of $G(\mc{O})^\tau$ factors through a group $G(\mc{O}_n)^\tau$ when $n$ is large, where $\mc{O}_n=\rk[z]/(z^n)$. Let $G(\mc{O}_n)^\tau{}':=G(\mc{O}_n)^\tau\times_{G^\tau} G^\tau{}'$. Note that $G(\mc{O}_n)^\tau{}'$ is the semi-direct product of $G^\tau{}'$ with the kernel $J_n$ of the projection $G(\mc{O}_n)^\tau{}'\to G^\tau{}'$.  In view of \cite[Theorem 1]{Popov}, $\mr{Pic}(G(\mc{O}_n)^\tau{}')=\mr{Pic}(G^\tau{}')=1$. Invoking \cite[Theorem 1]{Popov} again, any $\mb{G}_{\rm m}$-central extension of $G(\mc{O}_n)^\tau{}'$ is trivial.   Hence, the restriction of the line bundle $\ms{L}(\Lambda)$ on $X_w$ has a unique $G(\mc{O}_n)^\tau{}'$-equivariant structure, thus a unique $G(\mc{O})^\tau{}'$-equivariant structure.  Taking limit of $X_w$, we get a unique $G(\mc{O})^\tau{}'$-equivariant structure on $\ms{L}(\Lambda)$ over $\Fl_Y$. This proves the first statement.

In view of the first part, we have a natural embedding $G(\mc{O})^\tau{}'\to \tilde{L}(G,\tau)$. By Lemma \ref{4237984}, the torus $T^\tau{}'$ acts on $\ms{L}(\Lambda)|_{e_0}$ via the weight $-\lambda$. On the other hand, the $\mc{P}_Y$-equivariant structure on $\ms{L}(\Lambda)$ induces a $G'_\vartheta$-action on $\ms{L}(\Lambda)$ via the natural morphism $G'_\vartheta\to M_Y$ described above. Under this second equivariant structure, $T^\tau{}'$ acts on $\ms{L}(\Lambda)|_{e_0}$ via the weight $0$. Therefore, these two $G'_\vartheta$-actions differ by a weight $\lambda$.
\end{proof}

\subsection{Equivariance of line bundles on ramified global affine Grassmannians}\label{43287437}
Now we are in the same setting as in Section \ref{Section_main}. We can fix an interior point $\vartheta$ in the facet $F_Y$ of $\textsf{A}$ associated to $Y\subseteq\hat{I}_\tau$ so that $p\nmid m$, where $m$ is the integer defined in Lemma \ref{prop_goodprime}. Define parahoric group schemes $\ms{G}_\vartheta$ over $\bar{\mc{O}}$ and $\ms{G}_{-\vartheta,\infty}$ over $\bar{\mc{O}}_\infty$ as in Section \ref{sec_parahoric}. Let $\A^1=\mr{Spec}(\rk[z])$ be the affine line with a $\tau$-action given by $\tau(z)=\epsilon^{-1}z$. Let $C:=\A^1\cup \{\infty\}$ be the projective line. Let $\bar\A^1$ and $\bar{C}$ be the quotient of $\A^1$ and $C$ by the $\tau$-action. 

We can naturally glue $\ms{G}_\vartheta$, $\mr{Res}_{\A^1-0/\bar{\A}^1-0}(G_{\A^1-0})^{\tau}$, and $\ms{G}_{-\vartheta,\infty}$, and get a parahoric Bruhat-Tits group scheme $\mc{G}$ over $\bar{C}$. By construction, it satisfies the following property:
\begin{equation}\label{6892374}
\G|_{\mb{D}_{a}}\simeq
\begin{cases}
G_{\mb{D}_{a}} & \text{if } a\neq 0,\infty\\
\ms{G}_\vartheta & \text{if } a= 0\\
\ms{G}_{-\vartheta} & \text{ if } a=\infty
\end{cases},  
\end{equation}
where $\mb{D}_{a}$ denotes the formal disk around $a\in \bar{C}$. 

Similar to the local description in Theorem \ref{paraohoric_thm}, we shall give another description of the parahoric Bruhat-Tits group scheme from the equivariant point of view. Let $\tilde{\A}^1=\mr{Spec}(\rk[t])$ be the affine line with a $\sigma$-action given by $\sigma(t):=(\epsilon_m)^{-1}t$, where $\epsilon_m$ is a fixed primitive $m$-th root of unity. Let $\tilde{C}:=\tilde{\A}^1\cup \{\infty\}$ be the projective line. Let $\tilde{C}\to C$ be the map given by $a\mapsto a^{m/r}$ for any $a\in \tilde{C}$. Then, $\bar{C}=C/\langle \tau \rangle=\tilde{C}/\langle \sigma \rangle$. 

Define the following  group scheme over $\bar{C}$,
$$\mc{G}_{\sigma}:=\mr{Res}_{\tilde{C}/\bar{C}} (G\times \tilde{C})^{\sigma,\circ},$$
where $\sigma$ acts on $\tilde{C}$ and $\sigma$ acts on $G$ defined in (\ref{eq_sigma}).
Clearly, we have $\mc{G}_{\sigma}|_{\bar{\mb{D}}_{0}}\simeq \ms{G}_\sigma$ and $\mc{G}_{\sigma}|_{\bar{\mb{D}}_\infty}\simeq\ms{G}_{\sigma,\infty}$, where  $\ms{G}_\sigma$ and $\ms{G}_{\sigma, \infty}$ are defined in Section \ref{sec_parahoric}, and $\bar{\mb{D}}_{0}=\mr{Spec}(\bar{\mc{O}})$  denotes the formal disks around $0\in \bar{C}$. 

\begin{lem}
\label{lem_glob}
The map $\mr{Ad}_{z^\vartheta}$ induces an isomorphism of group schemes over $\bar{\A}^1-0$,
\begin{equation}\label{24502384}
\phi_{\bar{\A}^1-0}: \mr{Res}_{\A^1-0/\bar{\A}^1-0}(G_{\A^1-0})^\tau\simeq \mr{Res}_{\tilde{\A}^1-0/\bar{\A}^1-0}(G_{\tilde{\A}^1-0})^\sigma.
\end{equation}
\end{lem}
\begin{proof}
The proof is identical to the proof of \cite[Proposition 4.1.1]{DH}.
\end{proof}
In view of Theorem \ref{paraohoric_thm} and Lemma \ref{lem_glob}, the isomorphisms $\phi_0$, $\phi_\infty$ (defined in Section \ref{sec_parahoric}), and $\phi_{\bar{\A}^1-0}$ can be glued to an isomorphism of group schemes over $\bar{C}$,
\begin{equation}\label{8768769}
\phi: \G\simeq \G_{\sigma}.
\end{equation} 

\begin{definition}\label{def_grass}
The global affine Grassmannian $\Gr_{\G}$, the jet group scheme $L^+\G$, and the negative loop group scheme $L^-\G$ over $\bar{C}$ are defined as follows, for any $\rk$-algebra $R$,
\[\Gr_{\G}(R):=\left\{(a,\mc{F},\beta)\,\middle|\,a\in \bar{C}(R), \mc{F} \text{ a } \G\text{-torsor over } \bar{C}_R, \beta:\mc{F}|_{\bar{C}_R\setminus\Delta_{a}}\simeq \mathring{\mc{F}}|_{\bar{C}_R\setminus\Delta_{a}}\right\}\]
and 
\begin{align}
&L^+\G(R):=\left\{ (a,\eta) \,\middle|\,a\in \bar{C}(R), \eta \in \mc{G}(\hat\Delta_{a}) \right\}, \nonumber \\
&L^-\G(R):=\left\{ (a,\eta) \,\middle|\,a\in \bar{C}(R), \eta \in \mc{G}(\bar{C}_R\setminus\Delta_{a} ),\right\}\label{eq_global_jet}
\end{align}
where $\mathring{\mc{F}}$ is the trivial $\G$-torsor on $\bar{C}_R$, and $\hat{\Delta}_{a}$ is the completion of $\bar{C}_R$ along the graph $\Delta_{a}$ of $a: {\rm Spec} (R)\to \bar{C}$.

For any $a\in \bar{C}$, let $\Gr_{\G,a}$ and $L^+\G_{a}$ denote the fiber of $\Gr_{\G}$ and $L^+\G$ at $a$ respectively. Then, $\Gr_{\mc{G}}$ is an ind-scheme of finite type over $\bar{C}$, $L^+\mc{G}$ is a group scheme over $\bar{C}$, and $L^-\mc{G}$ is an ind-group scheme of finite type over $\bar{C}$, cf.\,\cite[Proposition 2, Lemma 20]{He}. 
\end{definition}
The global Grassmannian $\Gr_{\G}$ over $\bar{C}$ has the following property, 
\[\Gr_{\G,a}\simeq 
\begin{cases}
\Fl_Y & \text{when }a=0\\
\Gr_{G} & \text{when }a\neq 0,\infty.
\end{cases}\]

Let $\mr{Bun}_{\G}$ be the moduli stack of $\G$-torsors on $\bar{C}$, i.e. for any $\rk$-algebra $R$, $\mr{Bun}_{\G}(R)$ is the groupoid of $\mc{G}$-torsors over $\bar{C}_R$. By \cite[Proposition 1]{He}, $\mr{Bun}_{\G}$ is a smooth algebraic stack locally of finite type. 
\begin{prop}\label{45393240}
There is an isomorphism of algebraic stacks $$\mr{Bun}_{\G}\simeq [\mc{P}_Y^{-}\backslash \Fl_Y].$$
When $G$ is simply-connected,  the line bundle $\ms{L}(\Lambda)$ descends to a line bundle $\mc{L}_{\mr{Bun}_{\G}}$ on $\mr{Bun}_{\G}$.
\end{prop} 
\begin{proof}
By Proposition \ref{prop_3284932}, we have $L^-\G_{0}\simeq \mc{P}_Y^{-}$. By uniformization theorem (cf.\,\cite[Theorem 5]{He}), we have $\mr{Bun}_{\G}\simeq [\mc{P}_Y^{-}\backslash \Fl_Y]$.
The second part of the statement follows from the $\mc{P}_Y^{-}$-equivariance of the line bundle $\ms{L}(\Lambda)$ on $\Fl_Y$ constructed in Proposition \ref{thm_lgequivarian}.
\end{proof}
\begin{remark}
For general parahoric Bruhat-Tits group scheme $\mc{G}$ over a smooth projective curve, line bundles on ${\rm Bun}_{\mc{G}}$ have been further investigated in \cite{DH2} by Damiolini and the first author. 
\end{remark}

We denote by $\mc{L}:=\mr{pr}^*(\mc{L}_{\mr{Bun}_{\G}})$ the pullback of $\mc{L}_{\mr{Bun}_{\G}}$ via the natural projection $\mr{pr}:\Gr_{\G}\to \mr{Bun}_{\G}$. For each $a\in \bar{C}$, let $\mc{L}_a$ denote the restriction of $\mc{L}$ to the fiber $\Gr_{\G,a}$ of $\Gr_{\G}$ at $a\in \bar{C}$. 

\begin{prop}\label{324984}
Assume that $G$ is simply-connected. We have an isomorphism $\mc{L}_0\simeq\ms{L}(\Lambda)$ of line bundles on $\Gr_{\G,0}\simeq \Fl_Y$. Moreover, for any $a\neq 0,\infty$ in $\bar{C}$, we have an isomorphism $\mc{L}_{a}\simeq L^c$ of line bundles on $\Gr_{\G,a}\simeq \Gr_G$, where $L$ denotes the level one line bundle on $\Gr_G$. 
\end{prop}
\begin{proof}
The first statement follows from the construction of $\mc{L}$. For the second statement, applying \cite[Proposition 4.1]{Zhu14}, the central charge of the restriction $\mc{L}_a$ is constant, hence is equal to $c$. Since the central charge map $\mr{Pic}(\Gr_G)\to \mb{Z}$ is an isomorphism, we conclude that $\mc{L}_{a}\simeq L^c$.
\end{proof}

\begin{thm}\label{431424}
Assume that $G$ is simply-connected. The line bundle $\mc{L}$ over $\Gr_{\G}$ has an $L^-\G$-equivariant structure such that the induced $\mc{P}^-_Y$-equivariant structure on $\mc{L}_0$ is the same as the one constructed in Proposition \ref{thm_lgequivarian}, i.e.\,the induced $M_Y$-action on $\mc{L}_0|_{e_0}$ is trivial. 
\end{thm}
\begin{proof}
Note that $[L^-\G \backslash\Gr_{\G}]\simeq \bar{C}\times 
\mr{Bun}_{\G}$ given by $(a,\mc{F}, \beta)\mapsto (a, \mc{F})$. The pullback $\mc{L}=\mr{pr}^*(\mc{L}_{\mr{Bun}_{\G}})$ is naturally equipped with an $L^-\G$-equivariant structure. Restricting to the fiber at $0\in \bar{C}$, $M_Y\subseteq \mc{P}_Y^-\simeq L^-\G_{0}$ acts on $(\mc{L}_0)_{e_0}$ trivially. Thus, the induced $\mc{P}^-_Y$-equivariant structure on $\mc{L}_0$ agrees with the one constructed in Proposition \ref{thm_lgequivarian}.
\end{proof} 

In the following paragraph, we shall define the global Schubert varieties. In order to do this, we need to consider the base change of $\Gr_{\G}$ to $C$, i.e. 
$$\Gr_{\G,C}:=\Gr_{\G}\times_{\bar{C}}C.$$
For any $a\in C$, we denote by $\Gr_{\G,a}$ the schematic fiber of $\Gr_{\G,C}$ at $a$. Similarly, we set $L\G_C:=L\G\times_{\bar{C}} C$ and $L^+\G_C:=L^+\G\times_{\bar{C}} C$.  

For any dominant coweight $\mu\in X_*(T)$, let $s_{\mu}:C\to L\G_C$  be the section constructed in \cite[Section 3.3]{Zhu14}, satisfying the following property:
\begin{equation}
\label{eq_global_sect}
s_\mu(a)=\begin{cases}
(z-a)^\mu,   \quad a\in C-0\\
n^{\mu}\in T(\mc{K})^\tau,   \quad  a=0
\end{cases},
\end{equation}
where $n^\mu$ is defined in \eqref{462918}. This induces a section $C\to \Gr_{\G,C}$, which is still denoted by $s_\mu$ if there is no confusion. For any $a\in C-0$, $s_\mu(a)=e_{\mu}\in \Gr_G$ under the identification $\Gr_{\G,a}\simeq \Gr_G$; and $s_\mu(0)=e_{\bar{\mu}}\in \Gr_{\G,0}\simeq \Fl_Y$.

\begin{definition}
\label{def_global_schubert}
For any dominant coweight $\mu\in X_*(T)$,  we define the global Schubert variety $\Grb_{\G,C}^\mu$ to be the minimal $L^+\G_C$-stable irreducible closed subvariety of $\Gr_{\G,C}$ that contains $s_\mu$. 
\end{definition}
The following theorem is due to Zhu, cf.\,\cite[Theorem 3]{Zhu14}.
\begin{thm}\label{532953}
The global Schubert variety $\Grb^\mu_{\G,C}$ is flat over $C$, and for any $a\in C$, we have
\begin{equation}\label{eq_fiber}
\Grb^\mu_{\G,a}\simeq
\begin{cases}
\mc{A}_Y(\bar{\mu}) & \text{ if } a=0\\
\Grb_{G,\mu} & \text{ if } a\neq 0, \infty
\end{cases}.
\end{equation}
\end{thm}

Note that $\Gr_{\G,\A^1-0}\simeq \Gr_G\times (\A^1-0)$, whose connected components are parametrized by $X_*(T)/\check{Q}$. Let $M$ denote the set consisting of $0$ and minuscule dominant coweights of $G$. There is a natural bijection $M\simeq X_*(T)/\check{Q}$. For any $\xi\in M$,  let $\Gr_{\G,\A^1-0}[\xi]$ denote the component containing $e_{-\xi}\times (\A^1-0)$, where $e_{-\xi}\in \Gr_G$. Let $\Gr_{\G,C}[\xi]$ be the closure of $\Gr_{\G,\A^1-0}[\xi]$  in $\Gr_{\G,C}$. When $\xi=0$, we call $\Gr_{\G,C}[\xi]$ the neutral component and denote it by $\Gr_{\G,C}[o]$. The component $\Gr_{\G,C}[\xi]$ contains the global Schubert variety $\Grb_{\G,C}^\mu$ if and only if $\xi+\mu\in \check{Q}$.

\subsection{Proof of Theorem \ref{thm_mainthm}}\label{652935}
We first assume $G$ is simply-connected. Suppose $p\nmid a_Y$. Then, we can fix an interior point $\vartheta$ in the facet $F_Y$ of $\textsf{A}$ associated to $Y\subseteq\hat{I}_\tau$ so that $p\nmid m$. As in equation \eqref{eq_Glambda}, we associate a subgroup $G_\vartheta$ of $G^\tau$.
Recall the $\mc{P}_Y$-equivariant structure of the line bundle $\ms{L}(\Lambda)$ on $\Fl_Y$ and the $G(\mc{O})$-equivariant structure of line bundle $L$ on $\Gr_{G}$, cf.\,Proposition \ref{thm_Gtau} and \cite[Theorem 7]{Faltings03}. Via the embeddings $G_\vartheta\to \mc{P}_Y$ and $G_\vartheta\to G(\mc{O})$, they induce $G_\vartheta$-equivariant structures on line bundles $\ms{L}(\Lambda)$ and $L$ with $G_\vartheta$ acting on $\ms{L}|_{e_0}$ and $L|_{e_0}$ trivially.
\begin{prop}\label{equivariant}
Assume that $G$ is simply-connected. The line bundle $\mc{L}$ over $\Gr_{\G}$ (constructed in Section \ref{43287437}) has a $G_\vartheta\times \bar{C}$-equivariant structure such that the induced $G_\vartheta$-equivariant structure on $\mc{L}_a$ is the same as the $G_\vartheta$-equivariant structure on $\ms{L}(\Lambda)$ (when $a=0,\infty$) or $L^c$ (when $a\neq 0,\infty$) defined above.
\end{prop}
\begin{proof}
For the parahoric Bruhat-Tits group scheme $\mc{G}_\sigma$, we can similarly define the ind-group scheme $L^-\G_\sigma$ over $\bar{C}$ as in \eqref{eq_global_jet}.  Note that, for any commutative ring $R$ over $\rk$, $\G_\sigma(\bar{C}_R\backslash \Delta_{\bar{a}})$ can be naturally identified with the set of  $ \sigma$-equivariant morphisms from $ \tilde{C}_R\setminus \cup_{i=0}^{m-1}\Delta_{\sigma^i\tilde{a}}$  to $G$, where $\tilde{a}$ is a preimage of $\bar{a}$ under the map $\tilde{C}(R)\to \bar{C}(R)$ and $\Delta_{\sigma^i\tilde{a}}$ is the graph of $\sigma^i\tilde{a}$ in $C_{R}$. There is a natural embedding $G^\sigma\hookrightarrow \G_\sigma(\bar{C}_R\setminus \Delta_{\bar{a}})$ given by sending $g\in G^\sigma$ to the constant morphism landing at $g$. This induces an embedding $G^\sigma \times \bar{C}\hookrightarrow L^-\G_\sigma$ of group schemes over $\bar{C}$.
Moreover, the map $\phi$ in \eqref{8768769} induces an isomorphism  $L^-\G \simeq L^-\G_\sigma$. Hence, we have the following composition of morphisms 
\begin{equation}\label{45329856}
G_\vartheta\times \bar{C} \subseteq G^\sigma \times \bar{C}\hookrightarrow L^-\G_\sigma\simeq L^-\G.
\end{equation}
By Theorem \ref{431424}, this induces a $G_\vartheta\times \bar{C}$-equivariant structure on $\mc{L}$.

Restricting to the fibers at $0\in \bar{C}$, the composition map \eqref{45329856} is given by  
\begin{equation}
\label{eq_G_34}
G_\vartheta \subseteq G^\sigma \hookrightarrow G[t^{-1}]\xrightarrow{\mr{Ad}_{t^{-h}}} \mc{P}_Y^-,\end{equation}
where the isomorphism $\mr{Ad}_{t^{-h}}$ is the restriction of $\mr{Ad}_{t^{-h}}: G((t^{-1}))^\sigma\simeq G((z^{-1}))^\tau$, see Proposition \ref{prop_3284932}.
Since $\mr{Ad}_{t^{-h}}$ restricts to the identity map on $G_\vartheta$, the composition map in \eqref{eq_G_34} is exactly the natural embedding $G_\vartheta \subseteq \mc{P}_Y^-$. 

We now consider restricting the composition map \eqref{45329856} to the fibers at $\bar{a}\in \bar{C}\setminus \{0,\infty\}$. Let $a$ be a lifting of $\bar{a}$ in $C$, and let $\tilde{a}$ be a lifting of $a$ in $\tilde{C}$. Then, the restriction of the composition map \eqref{45329856} at $\bar{a}$ given by  
$$G_\vartheta \subseteq G^\sigma \hookrightarrow G[t_{\tilde{a}}^{-1}]\xrightarrow{\simeq} G[z_a^{-1}],$$ 
where $t_{\tilde{a}}$ is a uniformizer of $\tilde{a}$ in $\tilde{C}$ and $z_a$ is the corresponding uniformizer of $a$ in $C$ via the map $\tilde{C}\to C$.  By the definition of $G_\vartheta$ in (\ref{eq_Glambda}), $\mr{Ad}_{t^{-h}}$ restricts to the identity map on $G_\vartheta$. Thus, this composition is the natural embedding $G_\vartheta \subseteq G[z_a^{-1}]$. 

In summary, the natural $L^-\G$-equivariant structure on $\mr{pr}^*(\mc{L}_{\mr{Bun}_{\G}})$ induces a $G_\vartheta\times \bar{C}$-equivariant structure on it via the embedding $G_\vartheta\times \bar{C}\hookrightarrow L^-\G$ in \eqref{45329856}. Restricting to the fiber at $0\in \bar{C}$, the induced $G_\vartheta$-equivariant structure on $\mc{L}_0$ is the same as the one induced from the $\mc{P}_Y^-$-equivariant structure constructed in Proposition \ref{thm_lgequivarian} via the embedding $G_\vartheta\subseteq \mc{P}_Y^-$. In view of Proposition \ref{thm_lgequivarian}, this is the same as the one induced from $\mc{P}_Y$-equivariant structure via the embedding $G_\vartheta\subseteq \mc{P}_Y$. Moreover, restricting to the fiber at $\bar{a}\in \bar{C}\setminus \{0\}$, the induced $G_\vartheta$-equivariant structure on $\mc{L}_{\bar{a}}$ is the same as the one induced from $G[z_a^{-1}]$-equivariant structure via the embedding $G_\vartheta\subseteq G[z_a^{-1}]$. In view of \cite[Theorem 7]{Faltings03}, this is the same as the one induced from $G[[z_a]]$-equivariant structure via the embedding $G_\vartheta\subseteq G[[z_a]]$.
\end{proof}

From now on, we are under the assumptions in Theorem \ref{thm_mainthm}, i.e.\,$G$ is a general simple algebraic group, $\mu\in X_*(T)^+$ is a dominant coweight, and  $\Fl_Y^\kappa$ is the component of $\Fl_Y$ containing $\overline{\Fl}_{Y,\bar\mu}$. We are also given an ample line bundle $\ms{L}(\Lambda)_\kappa$ on  $\Fl^\kappa_Y$. By assumption $p\nmid a_{Y}$, we can choose an interior point $\vartheta$ in the facet associated to $Y$ so that $p\nmid m$. Recall that $\mc{G}$ is the associated parahoric Bruhat-Tits group scheme over $\bar{C}$, and $\Gr_{\G,C}$ is the  global affine Grassmannian of $\mc{G}$ base changed to $C$. 


Let $G'$ be the simply-connected cover of $G$. Let $\G'$ be the corresponding Bruhat-Tits parahoric group scheme over $\bar{C}$ associated to $G'$ and $\tau$. The cover $G'\to G$ induces a natural morphism $\G'\to \G$ of group schemes over $\bar{C}$. By a similar proof as in \cite[Lemma 5.16]{HY}, there is an isomorphism $\Gr_{\G',C}\simeq \Gr_{\G,C}[o]$. Let $\ms{L}(\Lambda)$ be the line bundle on $\Fl'_Y$ associated to $\Lambda$. By Proposition \ref{45393240}, the line bundle $\ms{L}(\Lambda)$ on $\Fl'_Y$ descends to a line bundle on $\mr{Bun}_{\G'}$. Pulling back via the composition of morphisms 
$$\Gr_{\G',C}\to\Gr_{\G'} \to\mr{Bun}_{\G'},$$
we get a line bundle $\mc{L}_o$ on $\Gr_{\G',C}$, which can be regarded as a line bundle on $ \Gr_{\G,C}[o]$.

Let  $\xi\in M$ be the element such that $\xi+\mu\in \check{Q}$. Let $T_\xi: \Gr_{\G,C}[o]\simeq \Gr_{\G,C}[\xi]$ be the translation map by multiplying $s_{-\xi}$. Let $\mc{L}_\xi:=(T_\xi)_*(\mc{L}_o)$ denote the push-forward of $\mc{L}_o$. In view of \eqref{eq_global_sect}, the fiber $\Gr_{\G,a}[\xi]$ of $\Gr_{\G,C}[\xi]$ at $a\not=0$ is the component $\Gr_G^\xi$ of $\Gr_G$ containing $e_{-\xi}$. By \eqref{eq_sckappa} and \eqref{5324798} in Section \ref{Section_main}, the line bundle $\mc{L}_\xi$ restricting to $\Gr_G^\xi$ is $\ms{L}(c\Lambda_{\xi})$, where $\Lambda_\xi=\Lambda_o+ \omega_\xi$ denotes the fundamental weight of level one corresponding to $\xi$  of the untwisted affine Dynkin diagram associated to $G$. On the other hand, up to a scalar of $\delta$, $\Lambda_\xi$ is equal to the weight $$\varsigma_\xi(\Lambda_o)=\varrho^{-\xi}w_{\xi}w_0(\Lambda_o)=\Lambda_o+ \iota(\xi)-\frac{1}{2}(\iota(\xi)\st \iota(\xi))\delta,$$
where the second equality follows from \eqref{Weyl_action}, and $\varsigma_\xi$ is the length zero element associated to $\xi$ (defined in \eqref{529345}).
It follows that $\omega_\xi= \iota(\xi)$, where $\iota$ is defined in \eqref{4529634}.

In view of \eqref{eq_global_sect} again, over $0\in C$ the translation map $T_\xi:\Gr_{\G,C}\to \Gr_{\G,C}$ is  the map $\Fl_Y\to \Fl_Y$ given by $x\mapsto n^{-\xi}\cdot x$. This restricts to the isomorphism $n^{-\xi}:\Fl_Y^\circ\simeq  \Fl_Y^\kappa$, as from $\xi+ \mu \in \check{Q}$ and $\kappa+ \bar{\mu}\in \check{Q}_\tau$ we have $\kappa-\bar{\xi}\in \check{Q}_\tau$, cf.\,\eqref{eq_kappa}. Then, the restriction of $\mc{L}_\xi$ to $\Gr_{\G,0}[\xi]\simeq  \Fl_Y^\kappa$ is $\ms{L}(\Lambda)_\kappa$, since
$$(\mc{L}_\xi)|_{\Fl_Y^\kappa}= (n^{-\xi})_*((\mc{L}_o)|_{\Fl_Y^\circ})=(n^{-\xi})_*(\ms{L}(\Lambda))=\ms{L}(\Lambda)_\kappa.$$

\begin{lem}\label{4592384}
The line bundle $\mc{L}_\xi$ on $\Gr_{\G,C}[\xi]$ has a $T'^{\tau}\times C$-equivariant structure. 
\end{lem}
\begin{proof}
A $T'^\tau\times C$-equivariant structure on $\mc{L}_\xi$ over $\Gr_{\G,C}[\xi]$ is equivalent to an $\mr{Ad}_{s_{\xi}}(T'^\tau\times C)$-equivariant structure on $\mc{L}_o$ over $\Gr_{\G,C}[o]$. Applying Proposition \ref{equivariant}, the line bundle $\mc{L}_o$ on $\Gr_{\G,C}[o]$ has an $\mr{Ad}_{s_{\xi}}(T'^\tau\times C)$-equivariant structure.  Note that $\mr{Ad}_{s_{\xi}}(T'^\tau\times C)=T'^\tau\times C$. Hence, $\mc{L}_\xi$ has a $T'^\tau\times C$-equivariant structure.
\end{proof}

We are now ready to conclude the proof of Theorem \ref{thm_mainthm}.
\begin{proof}
By Theorem \ref{532953} and Lemma \ref{4592384}, the space $H^0(\Grb_{\G,C}^\mu,\mc{L}_\xi)$ is a $T'^{\tau}$-module and flat over $C$. Modifying the equivariant structure on $\mc{L}_\xi$ by a character of $T'^\tau$, we can assume that $T'^\tau$ acts on $\mc{L}_\xi|_{s_\xi}$ via the weight $-c\iota(\xi)|_{T'^\tau}$. Restricting to the fiber $\Grb_{\G,a\not=0}^\mu\simeq \Grb_{G,\mu}$, $T'^\tau$ acts on $\ms{L}(c\Lambda_{\xi})|_{e_{-\xi}}$ via the weight $-c\iota(\xi)|_{T'^\tau}$. This matches the natural $T'$-equivariant structure on the line bundle $\ms{L}(c\Lambda_{\xi})$ over $\Grb_{G,\mu}$ described in Lemma \ref{4237984}. 

Note that 
\[ H^i(\Grb_{\G,0}^\mu,\mc{L}_\xi|_0)=H^i(\Grb_{\G,a}^\mu,\mc{L}_\xi|_a)=0, \quad \text{for any }i>0.
\]
Then, by the complete reducibility of $T'^\tau$, we have an isomorphism of $T'^{\tau}$-modules over different fibers:
\begin{equation}\label{5293845}
H^0(\Grb_{\G,0}^\mu,\mc{L}_\xi|_0)\simeq H^0(\Grb_{\G,a}^\mu,\mc{L}_\xi|_a).
\end{equation}
It can be rewritten as the following isomorphism of $T'^{\tau}$-modules
\begin{equation}\label{923475}
H^0(\mc{A}_Y(\bar{\mu}),\ms{L}(\Lambda)_\kappa)\simeq H^0(\Grb_{G,\mu},\ms{L}(c\Lambda_{\xi})).
\end{equation}

Via the natural morphism $T^\tau{}' \simeq T'^\tau{}'\to T'^\tau$, we get a $T^\tau{}'$-module structure on the right side of \eqref{923475}, which matches the action stated in Theorem \ref{thm_mainthm}. Note that the line bundle $\ms{L}(\Lambda_{\xi})$ on $\Grb_{G,\mu}$ was denoted by $L$ in the statement of our theorem.

Now, we restrict to the fiber $\Grb_{\G,0}^\mu\simeq \mc{A}_Y(\bar{\mu})$. By assumption, $T'^\tau$ acts on  $\ms{L}(\Lambda)_\kappa|_{e_{-\bar{\xi}}}$ via the weight $-c\iota(\xi)|_{T'^\tau}$. Via the natural morphism $T^\tau{}' \simeq T'^\tau{}'\to T'^\tau$, we get a $T^\tau{}'$-module structure on the left side of \eqref{923475}. To finish the proof, we shall compare this $T^\tau{}'$-module structure with the one stated in Theorem \ref{thm_mainthm}.

Recall that $G(\mc{O})^{\tau}{}'=G(\mc{O})^{\tau,\circ}\times_{G^{\tau,\circ}}G^\tau{}'$. Via the isomorphism $\Fl_Y^\kappa \simeq \Fl'_{Y_\kappa}$ defined in (\ref{eq_FL_comp}),  the line bundle $\ms{L}(\Lambda)_\kappa$ on $\Fl_Y^\kappa$ corresponds to the line bundle $\ms{L}(\varsigma_\kappa(\Lambda))$ on  $\Fl'_{Y_\kappa}$.  In view of Lemma \ref{4237984} and Proposition \ref{thm_Gtau}, the line bundle $\ms{L}(\varsigma_\kappa(\Lambda))$ has a unique $G(\mc{O})^{\tau}{}'$-equivariant structure, and the induced $T^\tau{}'$-action on $\ms{L}(\varsigma_\kappa(\Lambda))|_{n^{-\xi}(\dot{\varsigma}_\kappa)^{-1}e'_0}$ is given by the weight  
\[ \varrho^{-\bar\xi}(\varsigma_\kappa)^{-1}(-\varsigma_\kappa(\Lambda))|_{T'^\tau }=\varrho^{-\bar\xi}(-\Lambda)|_{T'^\tau}=-\lambda-c\iota(\xi)|_{T'^\tau},\]
where $\bar{\xi}$ is the image of $\xi$ in $X_*(T)_\tau$, and the second equality follows from \eqref{Weyl_action}.

Recall that the isomorphism $\Fl_{Y}^\kappa\simeq \Fl'_{Y_\kappa}$ sends $e_{-\bar{\xi}}$ to $n^{-\xi}(\dot{\varsigma}_\kappa)^{-1}e'_0$. Thus, $T^\tau{}'$ acts on  $\ms{L}(\Lambda)_\kappa|_{e_{-\bar{\xi}}}$ via the weight $-\lambda-c\iota(\xi)|_{T'^\tau}$.  Hence, the $T^\tau{}'$-module structure on the left side of \eqref{923475} differs from the one induced from the $G(\mc{O})^\tau{}'$-equivariant structure by the weight $-\lambda$ as well. Therefore, there is an isomorphism of $T^\tau{}'$-modules as stated in Theorem \ref{thm_mainthm}:
\begin{equation}\label{5983479}
H^0(\mc{A}_Y(\bar{\mu}),\ms{L}(\Lambda)_\kappa)\simeq H^0(\Grb_{G,\mu},\ms{L}(c\Lambda_{\xi}))\otimes \rk_{-\lambda} .
\end{equation}
This completes the proof of Part (1) of Theorem \ref{thm_mainthm}.

We now assume $\mr{char}(\rk)=0$. 
As mentioned in Section \ref{Section_main}, both sides of \eqref{5983479} are representations of $G'_\vartheta$. Since two sides of \eqref{5983479} have the same $T^{\tau}{}'$-character, by standard representation theory of reductive group we can conclude that \eqref{5983479} are isomorphic as  $G'_\vartheta$-representations.
\end{proof}

\section{Canonical $L^+\mc{G} $-equivariance on  line bundles over  $\Gr_{\mc{G}}$ }\label{32895693}

In Section \ref{sect_unique_ext}, we prove that the splitting of a central extension of smooth affine group scheme over a curve removing finitely many points can be automatically extended, and we also show that a line bundle over a projective flat scheme over a curve with an action of a group scheme, gives rise to a central extension in fppf topology.   In Section \ref{sect_global_equiv},  we apply these results to show that any line bundle on $\Gr_{\mc{G}}$ admits a unique $L^+\mc{G}'$-equivariant structure. 
\subsection{Central extensions of group schemes and splittings}
\label{sect_unique_ext}
We first recall the definition of central extensions of group schemes. Let $T$ be a scheme over an algebraically closed field $\rk$. An exact sequence of fppf group schemes over $T$ consists of morphisms $f:H_1\to H_2$ and $g:H_2\to H_3$ of fppf group schemes over $T$,
\begin{equation}\label{238492}
1\to H_1\xrightarrow{f} H_2 \xrightarrow{g} H_3 \to 1,
\end{equation}
such that $H_1\simeq \ker g$, and $g$ is surjective in fppf topology (equivalently, $g$ is faithfully flat and of locally finite presentation, cf.\,\cite[Expos\'e IV, Corollary 6.3.3]{SGA3}). If $H_1$ is in the center of $H_2$, then we say \eqref{238492} is a central extension of $H_3$ by $H_1$. 

\begin{prop}
\label{prop_ext_split}
Let $H$ be a smooth affine group scheme over an irreducible smooth curve $\Sigma$ over $\rk$, which is fiberwise connected. Let $\hat{H}$ be a central extension of $H$ by the constant group scheme $\mathbb{G}_{m,\Sigma}$,
$$1\to \mathbb{G}_{m,\Sigma} \to \hat{H}\xrightarrow{q} H\to 1.$$
Suppose that $q:\hat{H} \to H$ admits a splitting $\varphi$ over a nonempty open subset $U\subseteq \Sigma$. Then, $\varphi$ extends uniquely to a splitting over $\Sigma$.
\end{prop}
\begin{proof}
We can assume $\Sigma \setminus U$ consists of only one point $x_0 \in \Sigma$. Let $\pi: H\to \Sigma$ be the group scheme. Set $H_U:=\pi^{-1}(U)$ and $\hat{H}_U=q^{-1}(H_U)$.
Since $\widehat{H}$ is a central extension of $H$ by $\mathbb{G}_{m,\Sigma}$ in the fppf topology, 
we may regard $\widehat{H}$ as a $\mathbb{G}_m$-torsor on $H$. It naturally corresponds to a 
line bundle $L$ over $H$. A splitting $\varphi$ of $\widehat{H}$ over $U$ gives rise to a section 
$s_\varphi \in L(H_U)$, which is nowhere-vanishing on $H_U$. By assumption, $H$ is a smooth  affine scheme over $\rk$. Let $\mathrm{div}(s_\varphi)$ be the associated divisor of $s_\varphi$ on $H$. 
Then $\mathrm{div}(s_\varphi) = n \pi^{-1}(x_0)$,
for some integer $n$.

Suppose that $n \geq 0$. Then $s_\varphi$ is defined on $\pi^{-1}(U)$ and generically on $\pi^{-1}(x_0)$, i.e. $s_\varphi$ is defined on an open subset of $H$ whose complement is of codimension 2. By Hartog's lemma, 
$s_\varphi$ is defined fully on $H$. Restricting $s_\varphi$ to $\Sigma$ via the unit $e: \Sigma \to H$, $e^\ast s_\varphi$ is the nowhere vanishing section of $e^\ast L$ corresponding to the unit 
$\hat{e}: \Sigma \to \widehat{H}$ of $\widehat{H}$, since both are regular sections over the integral curve \(\Sigma\), they agree everywhere. Thus, $n = 0$. 

Suppose that \(n\le0\). Consider the nowhere-vanishing section
\[
s_\varphi^{-1}\in\Gamma(H_U,L^{-1}).
\]Its divisor is $\operatorname{div}(s_\varphi^{-1})=-n\pi^{-1}(x_0)$. Applying the same argument to \(L^{-1}\), we obtain \(-n=0\), hence \(n=0\).

Therefore, $s_\varphi$ is a nowhere-vanishing section of $L$ over $H$. It corresponds  
to a unique extension of $\varphi$ to $\Sigma$. Moreover, this extension is automatically multiplicative over $\Sigma$. 
\end{proof}

Let $X$ be a projective and flat scheme over $\Sigma$ such that, for any $\rk$-point $ a \in \Sigma$, the fiber $X_a$ is reduced and connected. Let $\mathcal{L}$ be a line bundle on $X$.
Let $H$ be a smooth affine group scheme over $\Sigma$, acting on $X$ over $\Sigma$. We can construct a group functor $\widehat{H}$ over $\Sigma$. For any scheme $S$ over $\Sigma$, 
\begin{equation}
\label{eq_cen_ext}
\hat{H}(S)=\left\{ (h,\phi)  \,\middle|\,   h \in {\rm Map}_{\Sigma}(S,H) ,   h^*\mathcal{L}_{X_S}\simeq \mathcal{L}_{X_S}  \right\},
\end{equation}
where $X_S:=X \times_\Sigma S $ and $h$ is an automorphism on $X_S$ over $\Sigma$ arising from the action of $H$ on $X$. 
$\widehat{H}$ is clearly a fppf sheaf over $\Sigma$. Let $i: \mathbb{G}_{m, \Sigma} \to \widehat{H}$ be the natural map given by mapping $t \in \mathbb{G}_m(S) = H^0(S, \ms{O}_S)^\times$ to $(e_S \in H(S), \,\, \mathcal{L}_{X_S} \overset{t}{\simeq} \mathcal{L}_{X_S})$, where $e_S$ is the unit in $H(S)$.

\begin{prop}
\label{prop_exact_fppf}
Under the same assumption as above, suppose further that $H$ is fiberwise connected over $\Sigma$, and for any $\rk$-point $ a \in \Sigma$, $X_a$ has a discrete Picard group. Then, the sequence
\[
1 \to \mathbb{G}_{m, \Sigma} \xrightarrow{i} \widehat{H} \xrightarrow{q} H \to 1
\]
is exact in the fppf topology.
\end{prop}
\begin{proof} 
By \cite[IV3, 12.2.4(vi)]{EGA},  the morphism $\pi: X\to \Sigma$ has geometrically reduced and geometrically connected fibers.  For any $\Sigma$-scheme $ S \xrightarrow{f} \Sigma$, we consider the following pullback diagram:
$$\xymatrix{
X_S \ar[d]_{\pi_S} \ar[r]^{f_X} &X\ar[d]^\pi\\
S \ar[r]^f &\Sigma}.$$
Then
$$\ker(q)(S) = \{
\mathcal{L}_{X_S} \overset{\varphi}{\simeq} \mathcal{L}_{X_S} \}\simeq \{ \mathscr{O}_{X_S} \overset{\varphi}{\simeq} \mathscr{O}_{X_S}\},$$
where $\mc{L}_{X_S}=(f_X)^*\mc{L}$.

By base change (cf.\,\cite[Remark 9.3.1]{Hartshorne77}), there is a canonical morphism $ f^\ast (\pi_\ast \mathscr{O}_{X})\to (\pi)_\ast \mathscr{O}_{X_S}$ of coherent sheaves. Notice that $\pi_S$ is projective and flat. Moreover, $\pi_S$ has geometrically reduced and geometrically connected fibers, as for any $s\in S$ the fiber $X_s:=(X_S)_s$ of $X_S$ over $s$ is isomorphic to the fiber of $X$ over $f(s)$. In particular, $\dim_{\rk(s)} H^0(X_s, \mathscr{O}_{X_s})=1$ for any $s\in S$, where $\rk(s)$ is the residue field of $s$. By Grauert theorem (cf.\,\cite[Corollary 12.9]{Hartshorne77}), $f^\ast (\pi_\ast \mathscr{O}_{X})$ and $(\pi)_\ast \mathscr{O}_{X_S}$ are locally free, and for any $s\in S$
\[  f^\ast (\pi_\ast \mathscr{O}_{X})\otimes \rk(s)\simeq H^0(X_s, \mathscr{O}_{X_s})\simeq  ((\pi_S)_\ast \mathscr{O}_{X_S}) \otimes \rk(s) . \]
 Thus,
$f^\ast (\pi_\ast \mathscr{O}_{X})\simeq  (\pi_S)_\ast \mathscr{O}_{X_S} $. From Grauert theorem, we in particular get an isomorphism $(\pi^\ast \mathscr{O}_{X})|_a \simeq H^0(X_a,\ms{O}_{X_a})$ for any $\rk$-point $a\in \Sigma$.
Since $X_a$ is reduced, connected, and projective, $H^0(X_a, \mathscr{O}_{X_a}) = \rk.$
Thus, $\pi_\ast \mathscr{O}_X \simeq \mathscr{O}_\Sigma$, and $(\pi_S)_\ast \mathscr{O}_{X_S} \simeq f^\ast \mathscr{O}_\Sigma = \mathscr{O}_S$.
Therefore, $\varphi: \mathscr{O}_{X_S} \simeq \mathscr{O}_{X_S}$ corresponds to an element in
\[
H^0(X_S, \mathscr{O}_{X_S})^\times \simeq H^0(S, (\pi_S)_\ast \mathscr{O}_{X_S})^\times \simeq H^0(S, \mathscr{O}_S)^\times \simeq \mathbb{G}_m(S).
\]
This gives rise to an isomorphism $\mathbb{G}_{m}(S) \simeq \ker(q)(S)$, which is functorial in $\Sigma$-scheme $S$. Thus, there is an isomorphism $\mathbb{G}_{m,\Sigma}\simeq \ker(q)$. 

We now prove $q$ is faithfully flat and locally of finite presentation. Consider the action map $a: H \times_\Sigma X \to X$ and the projection $p_X:H\times_\Sigma X\to X$. We have two line bundles $a^\ast \mathcal{L}$ and $p^\ast \mathcal{L}$ on $H \times_\Sigma X$. For any $\rk$-point $g \in H$, let $a = \pi_H(g)$, where $\pi_H: H \to \Sigma$. Then,
\[
(a^\ast \mathcal{L})|_g = g^*(\mathcal{L}|_{X_a}), \quad (p_X^\ast \mathcal{L})|_g = \mathcal{L}|_{X_a}.
\]

By assumption, $\mathrm{Pic}(X_a)$ is discrete and $H_a$ is connected. Then, $H_a$ acts on $\mathrm{Pic}(X_a)$ trivially, i.e.\,$(a^\ast \mathcal{L})|_g = (p_X^\ast \mathcal{L})|_g$.
In view of Lemma \ref{lem_see_saw}, there exists an isomorphism
\[
a^\ast \mathcal{L} \simeq p^\ast \mathcal{L} \otimes p_H^\ast \mathcal{E},
\]
for some line bundle $\mathcal{E}$ on $H$, where $p_H: H \times_\Sigma X \to H$ is a projection.
Thus, for any $\rk$-point $ g \in H$, there exists an open subset $U$ of $H$ containing $g$ such that
\[
a^\ast \mathcal{L}|_{U \times_\Sigma X} \cong p^\ast \mathcal{L}|_{U \times_\Sigma X}.
\]
This amounts to a section of $q: \widehat{H} \to H$ over $U \subseteq H$.
Thus, $\widehat{H}$ is isomorphic to $H \times \mathbb{G}_m$ locally in  Zariski topology.
It follows in particular that $\widehat{H}$ is an affine group scheme over $\Sigma$, and $q: \widehat{H} \to H$ is
faithfully flat and locally of finite presentation.
\end{proof}
The following lemma is a slight generalization of \cite[Chap.III, ex.12.4]{Hartshorne77}.
\begin{lem}
\label{lem_see_saw}
Let $X \xrightarrow{\pi} S$ be a flat projective morphism of schemes, where $S$ is an integral scheme of finite type over $\rk$. Assume that for any $\rk$-point $s\in S$,  the fiber $X_s$ is reduced and connected. Let $\mathcal{L}$ and $\mathcal{L}'$ be line bundles over $X$, such that 
$\mathcal{L}|_{X_s} \simeq \mathcal{L}'|_{X_s}$ for any $\rk$-point $s\in S$.
Then $\mathcal{L} \simeq \mathcal{L}' \otimes \pi^\ast \mathcal{E}$, for some line bundle $\mathcal{E}$ over $S$.
\end{lem}
\begin{proof}
We may assume $\mathcal{L}' = \mathscr{O}_X$. By Grauert theorem (cf.\,\cite[Corollary 12.9]{Hartshorne77}), for any $\rk$-point $s\in S$ we have
\[
(\pi_\ast \mathcal{L})|_s \simeq H^0(X_s, \mathcal{L}|_{X_s}) \simeq H^0(X_s, \mathscr{O}_{X_s}) \simeq \rk,
\]
where the second isomorphism follows from the assumption that $\mc{L}|_{X_s}\simeq \ms{O}_{X_s}$, and the third one follows from the assumption that $X_s$ is projective, reduced, and connected.

Since $S$ is integral and of finite type, it follows that $\mathcal{E} := \pi_\ast \mathcal{L}$ is a line bundle on $S$. Then, $\pi^*\mathcal{E} = \pi^\ast\pi_* \mathcal{L} \simeq \mathcal{L}$.  
\end{proof}

\subsection{Picard group of  $\Gr_{\mc{G}}$}\label{4924695}
Let $G$ be a simply-connected simple algebraic group over $\rk$. As in Section \ref{3208530}, we fix a non-empty subset $Y$ of $\hat{I}_\tau$. Let $\ms{G}_Y$ be the parahoric group scheme associated to $Y$, and let $\mc{G}$ be the parahoric Bruhat-Tits group scheme over $\bar{\A}^1$ glued from $\ms{G}_Y$ and $\mr{Res}_{\A^1-0/\bar{\A}^1-0} (G_{\A^1-0})^\tau$. Let $\Gr_{\G}$ be the global affine Grassmannian of $\G$ over $\bar{\A}^1$ and let $L^+\mc{G}$ be the jet group scheme of $\G$ over $\bar{\A}^1$, see Definition \ref{def_grass}.

Note that $\Gr_{\G,0}\simeq \Fl_Y$. We have the following restriction map, 
\begin{align*}
\mf{R}:\mr{Pic}(\Gr_{\G})&\to \mr{Pic}(\Fl_Y),\\
[\mc{L}]&\mapsto [\mc{L}_0]
\end{align*}
where $\mc{L}_0$ denotes the restriction of $\mc{L}$ to $\Gr_{\G,0}$.

\begin{prop}
    \label{416947}
The restriction map $\mf{R}$ is an isomorphism.
\end{prop}
\begin{proof}
By Proposition \ref{324984}, the restriction map is surjective. Let $\mc{L}$ be a line bundle on $\Gr_{\G}$. Assume that the restriction $\mc{L}_0$ is trivial. We shall show $\mc{L}$ is also trivial.

Let $\pi:\Gr_{\G}\to \bar{\A}^1$ denote the projection map.  For any $a\in \bar{\A}^1$, by \cite[Proposition 4.1]{Zhu14}, the central charge of the restriction $\mc{L}_a$ is equal to the central charge of $\mc{L}_0$, hence is equal to $0$. For any $a\neq 0$, the central charge map $\mr{Pic}(\Gr_{\G,a})\to \mb{Z}$ is an isomorphism, cf.\,\cite[Corollary 12]{Faltings03}. Therefore, $\mc{L}_a$ is also trivial for any $a\neq 0$. Since $\Gr_{\G}$ is ind-projective over $\bar{\A}^1$, we have $\dim\big(H^0(\Gr_{\G,a},\mc{L}_a)\big)=1$ for any $a\in \bar{\A}^1$. It follows that $\pi_*(\mc{L})$ is a locally free sheaf of rank $1$ on $\bar{\A}^1$. Thus,  $\pi_*(\mc{L})\simeq \ms{O}_{\bar{\A}^1}$.  
By adjunction, we have a morphism $\pi^*\pi_*(\mc{L})\to \mc{L}$. (More precisely, all the pushforward and adjunction arguments here are understood stagewise on the global Schubert varieties, using Lemma \ref{532985} below.)

Set $X:=\Gr_{\G}$. For any $x\in X$, the map $(\pi^*\pi_*\mc{L})|_x=H^0(\Gr_{\mc{G},a}, \mc{L}_a)\to \mc{L}|_x$ is an isomorphism, since $\mc{L}_a:=\mc{L}|_{\Gr_{\mc{G},a}}$ is trivial on $\Gr_{\mc{G},a}$ and $\Gr_{\mc{G},a}$ is ind-projective, where $a=\pi(x)$.  It follows that $\mc{L}\simeq \pi^*\pi_*(\mc{L})\simeq \pi^*\ms{O}_{\bar{\A}^1}$ is trivial. Therefore, the restriction map $\mf{R}$ is also injective.
\end{proof}

For any affine weight $\Lambda=\sum_{i\in Y } n_i\Lambda_i$, by this proposition we can uniquely associate to a line bundle $\mc{L}(\Lambda)$ on $\Gr_{\mc{G}}$ such that $\mc{L}(\Lambda)_0\simeq \ms{L}(\Lambda)$ on $\Fl_Y$.
\vspace{2mm}

Let $C=\mathbb{A}^1$ and $\bar{C}=\bar{\A}^1$. We now allow $G$ to be a general simple algebraic group over $\rk$ such that $p\not \mid |\pi_1(G)|$. Recall the global affine Grassmannian $\Gr_{\G,C}:=\Gr_{\G}\times_{\bar{C}} C$ over $C$. As in Section \ref{652935}, let $M$ be the set of zero or minuscule coweights of $G$.  For any $\xi\in M$, let $\Gr_{\G,C}[\xi]$ be the $\xi$-component of $\Gr_{\G,C}$. Note that $\Gr_{\G,C}[\xi]|_0\simeq \Fl_Y^{\kappa}$, where $\kappa$ has the same image as $\xi$ in $X_*(T)_\tau/X_*(T')_\tau$. Recall the section $s_{-\xi}:C\to L\G\times_{\bar{C}}C$ which satisfies $s_{-\xi}(0)=n^{-\xi}\in T(\mc{K})^\tau$, where $n^{-\xi}$ is defined in \eqref{462918}. The translation by $s_{-\xi}$ gives rise to a morphism $T_{\xi}:\Gr_{\G,C}[o]\simeq \Gr_{\G,C}[\xi]$.

\begin{prop}
\label{prop_comp_linebun}With the assumptions above, the following restriction map is an isomorphism,
\begin{equation*}
\mr{Pic}(\Gr_{\G,C}[\xi])\to \mr{Pic}(\Fl_Y^\kappa), \quad [\mc{L}]\mapsto [\mc{L}_0].
\end{equation*}
\end{prop}
\begin{proof}
Let $\mc{L}$ be a line bundle on $\Gr_{\G,C}[\xi]$. The pull-back $T_{\xi}^*\mc{L}$ is a line bundle on the neutral component $\Gr_{\G,C}[o]\simeq \Gr_{\G',C}$, where $\G'$ is the parahoric Bruhat-Tits group scheme associated to the simply-connected cover $G'$ of $G$. Note that at the $0$-fiber, we have $\Fl_Y^\circ=\Gr_{\G,C}[o]|_0\simeq \Gr_{\G',0}= \Fl_{Y}'$. By Proposition \ref{416947}, $T_{\xi}^*\mc{L}$ is uniquely determined by its restriction on the neutral component $\Fl^\circ_{Y}$ of $\Fl_Y$. Hence, $\mc{L}$ is uniquely determined by $\mc{L}_0$, as $(T^*_\xi \mc{L})|_{\Fl^\circ_Y}\simeq( n^{-\xi} )^*(\mc{L}|_{\Fl^\kappa_Y } )=( n^{-\xi} )^*( \mc{L}_0 )$. Hence, the restriction map is injective.

For the surjectivity, let $\ms{L}$ be a line bundle on $\Fl_Y^\kappa$. Then,  $(n^{-\xi})^*\ms{L}$ is a line bundle on $\Fl_Y^\circ \simeq \Fl'_Y$. 
By Proposition \ref{416947}, there is a line bundle $\mc{L}'$ on $\Gr_{\G}[o]\simeq \Gr_{\G'}$ such that $\mc{L}'|_{\Fl_Y^\circ} \simeq (n^{-\xi})^*\ms{L}$. Let $\mc{L}:=(T_\xi)_*\mc{L}'$ be the push-forward of $\mc{L}'$ via $T_\xi$. Then, $\mc{L}$ is a line bundle on $\Gr_{\G,C}[\xi]$ whose restriction to the fiber $\Gr_{\G,C}[\xi]|_0\simeq \Fl_Y^\kappa$ is $\ms{L}$. Hence, restriction map is surjective.
\end{proof}

\subsection{$L^+\mc{G}$-equivariance on line bundles over $\Gr_{\mc{G}}$}  \label{sect_global_equiv}

Let $L^+\G$ be the global jet group scheme of $\G$ over $\bar{C}:=\bar{\A}^1$. For any $S$-point $S\xrightarrow{x} \bar{C}$, $L^+\G(S):=\G(\hat{\Gamma}_x)$, where $\hat{\Gamma}_x$ is the formal completion of $\bar{C}\times S$ along the graph $\Gamma_x$. Similarly, we define $L_k^+\G$ to be the global $k$-jet group scheme of $\G$, i.e.\,$L_k^+\G(S):=\G(\Gamma_{x,k})$, where $\Gamma_{x,k}=\mr{Spec}(\ms{O}_{\bar{C} \times S}/\ms{I}_{x}^{k+1})$ and $\ms{I}_{x}$ is the ideal sheaf of $\Gamma_x$ in $\ms{O}_{\bar{C}\times S}$.

Note that $L^+\G$ and $L^+_k\G$ are representable group schemes over $\bar{C}$, cf.\,\cite[Proposition 2]{He}. Moreover, $L^+\G=\varprojlim_{k} L^+_k\G$. For each $k$, $L^+_k\G$ is of finite type over $\bar{C}$, and it is easy to see that $L^+_k\G$ is formally smooth, as $\G$ is smooth over $\bar{C}$. Thus, $L^+_k\G$ is smooth over $\bar{C}$ as well.  Set $L^+\G_C:= L^+\G\times_{\bar{C}}C$ and $L^+_k\mc{G}_C:= L^+_k \mc{G}\times_{\bar{C}} C$, where $C=\A^1$.

Let $G'$ be the simply-connected cover of $G$, and let $\mc{G}'$ be the associated parahoric Bruhat-Tits group scheme over $\bar{C}$. We similarly have the global jet ($k$-jet) group schemes $L^+\mc{G}', L^+_k\mc{G}', L^+\mc{G}'_C$ and $L^+_k\mc{G}'_C$.

Consider the global affine Grassmannian $\Gr_{\G,C}:=\Gr_{\G}\times_{\bar{C}} C$ over $C$. There is a section $s_\mu:C\to \Gr_{\G,C}$ for any $\mu\in X_*(T)^+$.  Let $\Grb^\mu_{\mc{G},C}$ be the global Schubert variety as defined in Definition \ref{def_global_schubert}. The global Schubert variety $\overline{\Gr}_{\G, C}^\mu$ is contained in $\Gr_{\G,C}[\xi]$ if and only if $\mu+\xi\in X_*(T')$.

\begin{thm}\label{569235}
Let $\mc{L}$ be a line bundle on $\Gr_{\G,C}[\xi]$. Let $\mu\in X_*(T)^+$ such that $\mu+\xi\in X_*(T')$. There is a unique $L^+\mc{G}'_C $-equivariant structure on the restriction $\mc{L}|_{ \Grb^\mu_{\mc{G},C} }$ of $\mc{L}$ to the global Schubert variety $\Grb^\mu_{\mc{G},C}$.  
\end{thm}

\begin{proof}
We simply denote the restriction $\mc{L}|_{\Grb_{\G,C}^\mu}$ by $\mc{L}$. 
By the same argument as in \cite[Lemma 5.11]{HY}, the action of $L^+\mc{G}'_C$ on $\Grb_{\G,C}^\mu$ factors through the action of $L^+_k\mc{G}'_C$ for large enough integer $k$. In the following, we will only work with large enough $k$.

Via the line bundle $\mc{L}$, we can construct a central extension $\widehat{L^+_k\mc{G}'_C}$ of $L^+_k\G'_C$ by the constant group scheme $\mathbb{G}_{{\rm m}, C }$ as in \eqref{eq_cen_ext},
\begin{equation}\label{489434}
1\to \mathbb{G}_{{\rm m}, C }\to \widehat{L^+_k\mc{G}'_C} \xrightarrow{p} L^+_k\G'_C \to 1.
\end{equation} 
Since $\Grb_{\G,C}^\mu$ is projective and flat with reduced and connected fibers (cf.\,\cite{Zhu14}), and each fiber $\Grb_{\G,a}^\mu$ has discrete Picard group, by Proposition \ref{prop_exact_fppf}, this sequence is exact in the fppf topology.

Note that $\Grb_{\G,C\setminus\{0\}}^\mu\simeq \Grb_{G,\mu}\times (C\setminus\{0\})$ and $L^+_k\G'_{C\setminus\{0\}}\simeq L^+_kG' \times (C\setminus\{0\})$. We have a unique splitting $\beta:L^+_k\G'_{C\setminus\{0\}}\to \widehat{L^+_k\G'}_{C\setminus\{0\}}$ of the extension $\eqref{489434}$ over $C\setminus\{0\}$. By Proposition \ref{prop_ext_split}, this splitting extends to a unique splitting over $C$. 
Now, the theorem follows from the fact that an $L^+\G'_C $-equivariant structure on $\mc{L}$ is equivalent to a splitting of \eqref{489434}. 
\end{proof}

When $\mu$ is $\tau$-invariant i.e.\,$\tau(\mu)=\mu$, the section $s_\mu$ descends to a section $\bar{s}_\mu:\bar{C}\to \Gr_{\G}$, see \cite[Remark 3.2]{Zhu14}. Let $(X_*(T)^+)^\tau$ denote the set of $\tau$-invariant elements of $X_*(T)^+$.  For any  $\mu\in (X_*(T)^+)^\tau$, we define the global Schubert variety $\Grb_{\G}^\mu$ to be the minimal $L^+\G$-stable irreducible closed subvariety of $\Gr_{\G}$ containing $\bar{s}_\mu$.

\begin{lem}\label{532985}
\begin{enumerate}
\item If $\mc{G}$ is generically simply-connected, then $\Gr_{\G}=\varinjlim\limits_{\mu} \Grb_{\G}^\mu$,    a direct limit of  $\Grb_{\G}^\mu$ with respect to the natural partial order $\prec$ on $(X_*(T)^+)^\tau$.
\item In general, for any $\xi\in M$, let $P_\xi^+
:=\{\mu\in X_*(T)^+\mid \mu+\xi\in\check Q\}$.  Then, $$\Gr_{\G,C}[\xi]=\varinjlim\limits_{\mu } \Grb_{\G,C}^\mu,$$
where the direct limit is over $\mu\in P_\xi^+$ with respect to the standard partial order.
\end{enumerate}
\end{lem}
\begin{proof}
We first prove part 1). 
Consider the ind-closed embedding $j:\varinjlim\limits_\mu \Grb_{\G}^\mu\to \Gr_{\G}$ over $\bar{C}$, where $\mu\in (X_*(T)^+)^\tau$. Note that $\Gr_{\G, C\setminus\{0\}}\simeq \Gr_{G,C\setminus\{0\}}$ and it restricts to an isomorphism $\Grb_{\G,C\setminus\{0\} }^\mu\simeq \Grb_{G,C\setminus\{0\}}^\mu$ for any $\mu\in (X_*(T)^+)^\tau$. Since $\varinjlim\limits_\mu \Grb_{G,C\setminus\{0\}}^\mu \simeq \Gr_{G,C\setminus\{0\}}$,  it induces an isomorphism
$j_{C\backslash \{0\}}:\varinjlim\limits_\mu  \Grb_{\G,C\setminus\{0\}}^\mu\simeq  \Gr_{\G,C\setminus\{0\}}$. As $C\setminus\{0\}\to \bar{C}\setminus\{0\}$ is an \'etale $\mathbb{Z}/r\mathbb{Z}$-cover, the isomorphism $j_{C\backslash \{0\}}$ descends to an isomorphism $j|_{\bar{C}\setminus\{0\}}:\varinjlim\limits_\mu \Grb_{\G,\bar{C}\setminus\{0\}}^\mu\to \Gr_{\G,\bar{C}\setminus\{0\}}$. 

We now show that the restrictions $j|_0$ is also an isomorphism.  Recall that $\Gr_{\G,0}\simeq \Fl_Y$ and $\Grb_{\G,0}^\mu\simeq \mc{A}_Y(\bar\mu)$ for any $\mu\in (X_*(T)^+)^\tau$, see \eqref{eq_fiber}. Thus, the restriction $j|_0$ is the ind-closed embedding $\varinjlim\limits_\mu\mc{A}_Y(\bar\mu)\to \Fl_Y$, where $\mu\in (X_*(T)^+)^\tau$.  Given any $w\in W_{\mr{aff}}$, we shall show that $w\preccurlyeq_{Y}\varrho^{\bar\mu}$ for some $\mu\in (X_*(T)^+)^\tau$.

Let $\mu\in (X_*(T)^+)^\tau$ be an element such that $\bar\mu\neq 0$. Since the element $\varrho^{\bar\mu}$  has infinite order, any reduced word expression of $\varrho^{\bar\mu}$ in $\tilde{W}_{\rm aff}$  must contain all simple reflections $s_i$ of $W_{\rm aff}$. Therefore, $s_{i}\prec \varrho^{\bar\mu}$ for any $i\in \hat{I}_\tau$. Let $w=s_{i_1}\cdots s_{i_k}$ be a reduced word in $ W_{\mr{aff}}$. Since $\ell(\varrho^{k\bar\mu})=k\ell(\varrho^{\bar\mu})$,   the element $\varrho^{k\bar\mu}$ has a reduced word expression which contains $s_{i_1}\cdots s_{i_k}$ as a subword. Thus, $w\prec \varrho^{k\bar\mu}$.  By \cite[Lemma 1.3.18]{Ku}, we get $w\preccurlyeq_Y \varrho^{k\bar\mu}$.  Thus, $\overline{\Fl}_{Y,w}\subseteq \overline{\Fl}_{Y,k\bar\mu}\subseteq \mc{A}_Y(k\bar\mu)$. 
Since $\Fl_Y$ is the union of Schubert varieties $\bigcup_{w\in W_{\mr{aff}}}\overline{\Fl}_{Y,w}$, it follows that $j|_0:\varinjlim\limits_\mu\mc{A}_Y(\bar\mu)\simeq  \Fl_Y$. Finally, by a theorem of W.\,Reeves (cf.\,\cite[Theorem 4.23]{HY}), $j$ is an isomorphism.

We now prove part 2). Fix \(\xi\in M\), and let \(\varsigma_\kappa\in\Omega\) be the length-zero element corresponding to the component
$\Gr_{\mathcal G,C}[\xi]|_0\simeq \Fl_Y^\kappa$. Thus, \(\kappa\) and the image \(\bar\xi\) of \(\xi\) have the same class in $X_*(T)_\tau/\check Q_\tau$. We shall show that, for every Schubert variety \(\Fl_{Y,w}\subset\Fl_Y^\kappa\), there exists 
$\nu\in P_\xi^+$ such that
$
w\preccurlyeq\varrho^{\bar\nu}$. 

Write
$w=y\varsigma_\kappa$, with $ y\in W_{\mathrm{aff}}$,
where
$\varsigma_\kappa=\varrho^{-\kappa}w^\kappa $ as in \eqref{529345}. By the proof of part 1),  we may choose $
\bar\nu_1\in Q_\tau $ 
sufficiently dominant such that 
$y\preceq\varrho^{\bar\nu_1}$, and $\bar\nu_1-\kappa$ is also sufficiently dominant. Then, 
\[
w=y\varsigma_\kappa
\preceq
\varrho^{\bar\nu_1}\varsigma_\kappa
=
\varrho^{\bar\nu_1-\kappa}w^\kappa.
\]
Choose a dominant element $\bar\nu_2\in \check{Q}_\tau$ such that $w^\kappa\preceq \varrho^{\bar\nu_2}$.  Set $\bar \nu=\bar\nu_1+ \bar\nu_2-\kappa$. Then, $\ell(\varrho^{\bar\nu})=\ell(\varrho^{\bar\nu_1-\kappa})+\ell(\varrho^{\bar \nu_2})$. 
Since \(\bar\nu_1-\kappa\) is sufficiently dominant, the length formula gives
\[
\ell(\varrho^{\bar\nu_1-\kappa}w^\kappa)
=
\ell(\varrho^{\bar\nu_1-\kappa})+\ell(w^\kappa).
\]
By compatibility of the Bruhat order with length-additive  multiplication,
\[
\varrho^{\bar\nu_1-\kappa}w^\kappa
\preceq
\varrho^{\bar\nu_1-\kappa}\varrho^{\bar\nu_2}
=
\varrho^{\bar\nu}.
\]Therefore, $w\preceq  \varrho^{\bar\nu}$,
and consequently
$\Fl_{Y,w}\subseteq
\Fl_{Y,\bar\nu}
\subseteq A_Y(\bar\nu)$. Since \(\kappa\) and \(\bar\xi\) have the same class in \(X_*(T)_\tau/\check Q_\tau\), we have
\[
\bar\nu+\bar\xi
=\bar\nu_1+\bar\nu_2+(\bar\xi-\kappa)\in\check Q_\tau.
\]Taking \(\bar\nu_1\) sufficiently dominant, we may choose a dominant lift \(\nu\) of \(\bar\nu\) such that \(\nu+\xi\in\check Q\). Hence \(\nu\in P_\xi^+\). Therefore,  the special fiber of
$\varinjlim_{\nu\in P_\xi^+}
\overline{\Gr}_{\mathcal G,C}^{\,\nu}$ is \(\Fl_Y^\kappa\).  
Since \(P_\xi^+\) is a direct system with respet to the dominance order, and the corresponding Schubert varieties exhaust the component over \(C\setminus\{0\}\), the remainder of the proof is identical to part 1).
\end{proof}

\begin{cor}\label{global_thm}
Assume that $G$ is simply-connected. Let $\mc{L}$ be a line bundle on $\Gr_{\G}$. There is a unique $L^+\mc{G} $-equivariant structure on the line bundle $\mc{L}$.  
\end{cor}
\begin{proof}
Let $\mc{L}_C$ be the pull back of $\mc{L}$ to $\Gr_{\G,C}$. For any $\mu\in X_*(T)^+$, by Theorem \ref{569235}, the restriction of $\mc{L}_C$ to $\Grb_{\G,C}^\mu$ has a unique $L^+\mc{G}_C $-equivariant structure. When $\mu$ is $\tau$-invariant, this $L^+\mc{G}_C $-equivariant structure descends to an $L^+\mc{G}$-equivariant structure on $\mc{L}|_{\Grb_{\G}^\mu}$. By taking a limit and applying part 1) of Lemma \ref{532985}, we get an $L^+\mc{G}$-equivariant structure on $\mc{L}$. For the uniqueness, by Theorem \ref{569235}, the $L^+\mc{G}$-equivariant structure on $\mc{L}|_{\Grb_{\G}^\mu}$ is unique. Hence, the $L^+\mc{G}$-equivariant structure on $\mc{L}$ must be unique.

\end{proof}

\begin{prop}
Let $\mc{L}$ be a line bundle on $\Gr_{\G,C}[\xi]$ with level $c$. There is a unique $L^+\mc{G}'_C $-equivariant structure on $\mc{L}$. Moreover, the induced $T'^\tau$-action on $\mc{L}|_{s_{-\xi}}$ is given by the weight $-c\iota(\xi)$.
\end{prop}
\begin{proof}
The first statement immediately follows from Theorem \ref{569235} and part 2) of Lemma \ref{532985}. By the same argument as in Proposition \ref{equivariant}, we have an embedding $T'^\tau\times C\hookrightarrow L^+\mc{G}'_C$. Moreover, the restriction of the embedding to the fiber at $a\neq 0$ is given by the natural embedding $T'^\tau\subseteq G'\subseteq G'[[z_a]]\simeq L^+\G'_{a}$, where $z_a:=z-a$. To prove the second statement, we consider the induced $T'^\tau\times C$-equivariant structure on $\mc{L}$.

Given any point $a\not=0$ in  $C$, $\mc{L}_a$ is a level-$c$ line bundle on  $\Gr_{\G,a}[\xi]\simeq \Gr_G^\xi$, cf.\,Section \ref{652935}. It is the push-forward of the line bundle $L^c$ on $\Gr_G^o\simeq \Gr_{G'}$ via translation map $\Gr_G^o\to \Gr_G^\xi$. 
As a $G'(\mc{K})$-homogeneous space, $\Gr_G^\xi$ is isomorphic to the partial affine flag variety $ G'(\mc{K})/\mr{Ad}_{z^{-\xi}}(G'(\mc{O}))$ mapping $e_{-\xi}$ to the base point $e'_0$, and the line bundle $\mc{L}_a$ corresponds to $\ms{L}(c\Lambda_\xi)$.  
By Proposition \ref{thm_Gtau}, $\mc{L}_a$ has a unique $G'(\mc{O})$-equivariant structure. Via the natural embedding $T'^\tau\hookrightarrow  G'(\mc{O})$, this equivariant structure gives rise to an action of $T'^\tau$ on the fiber $\ms{L}(c\Lambda_\xi)|_{e'_0}$. Proposition \ref{thm_Gtau} further asserts that this action is given by the weight $-c\iota(\xi)$. Note that the isomorphism $\Gr_G^\xi \simeq G'(\mc{K})/\mr{Ad}_{z^{-\xi}}(G'(\mc{O}))$ sends $e_{-\xi}$ to $e'_0$. It follows that $T'^\tau$ acts on $\mc{L}_a|_{e_{-\xi}}$ via the weight $-c\iota(\xi)$. Thus, the induced global $T'^\tau$-action on $\mc{L}|_{s_{-\xi}}$ is also given by the weight $-c\iota(\xi)$.
\end{proof}
The fiber of $L^+\mc{G}'_C $ at $0\in C$ is isomorphic to the parahoric subgroup $\mc{P}'_Y\subset G'(\mc{K})^\tau$ associated to $Y$.  The above proposition determines the $\mc{P}'_Y$-equivariant structure on the line bundle $\mc{L}(\Lambda)_\kappa$ over $\Fl_Y^\kappa$ obtained from the unique $L^+\mc{G}'_C $-equivariant structure on $\mc{L}$, where $\kappa$ is the minimal element in $X_*(T)^+_\tau$ such that $\kappa-\bar{\xi}\in \check{Q}_\tau$, cf.\,\eqref{eq_kappa}.

\section{Application of the coherence conjecture to Affine Demazure modules}\label{3985299}

\subsection{Affine Kac-Moody algebras}\label{sec_liealgebra}

Let $\mf{g}$ be a simple Lie algebra over $\mb{C}$. We fix a Borel subalgebra $\mf{b}\subseteq \mf{g}$ and a Cartan subalgebra $\mf{h}\subseteq \mf{b}$. Let $( \cdot, \cdot)$ be the invariant (symmetric, nondegenerate) bilinear form on $\mf{g}$ normalized so that the induced form on the dual space $\mf{h}^*$ satisfies $( \theta, \theta)=2$ for the highest root $\theta$ of $\mf{g}$.

Let $\tau$ be a diagram automorphism of $\mf{g}$ of order $r$ preserving $\mf{b}$ and $\mf{h}$, which is allowed to be trivial.  We fix a primitive $r$-th  root of unity $\epsilon$.  Let $\mf{g}((t)):=\mf{g}\otimes_{\mb{C}} \mb{C}((t))$ be the loop algebra and let $\mf{g}((t))^\tau$ denote the $\tau$-fixed point subalgebra with respect to the extended $\tau$-action by $\tau(t):=\epsilon^{-1}t$. Then, the (twisted) affine Lie algebra $\tilde{L}(\mf{g},\tau)$ associated to $\mf{g}$ and $\tau$ is defined as follows, 
\begin{equation}\label{liealg}
\tilde{L}(\mf{g},\tau):=\mf{g}((t))^\tau\oplus \mb{C} K,
\end{equation}
where $K$ denotes the canonical central element and the Lie bracket on $\tilde{L}(\mf{g},\tau)$ is given by 
\[ [x(f)+zK,x'(f')+z'K]:=[x,x'](ff')+r^{-1}\mr{Res}_{t=0}((df)f')( x,x') K.\]

Let $\hat{L}(\mf{g},\tau):=\tilde{L}(\mf{g},\tau)\oplus \mb{C} d$ be the affine Kac-Moody algebra, where $d$ is the scaling element. When $\tau$ is trivial, we simply denote $\hat{L}(\mf{g},\tau)$ by $\hat{L}(\mf{g})$. Set $\hat{\mf{h}}^\tau:=\mf{h}^\tau\oplus \mb{C} K\oplus \mb{C} d$, which is the Cartan subalgebra of  $\hat{L}(\mf{g},\tau)$. Let $(\hat{\mf{h}}^\tau)^*$ denote the linear dual of $\hat{\mf{h}}^\tau$.  

Let $X_N^{(r)}$ be the affine Dynkin diagram associated to $\hat{L}(\mf{g},\tau)$, and we similarly have the notations $I_\tau, \hat{I}_\tau$, and $\{\alpha_i,\check{\alpha}_i\st i\in \hat{I}_\tau\}$, as defined in Section \ref{Notations}.  Recall the weight $\theta_0$ of $\mf{g}^\tau$ defined in \eqref{eq_theta_0}. We have another description of $\theta_0$ and its dual $\check{\theta}_0$ as follows,
\begin{gather}\label{eq_theta0}
\theta_0=
\begin{cases}
\text{highest root of }\mf{g}, &\text{if }r=1\\
\text{highest short root of }\mf{g}^\tau, &\text{if }r > 1 \text{ and } (\mf{g},r)\neq (A_{2n},2)\\
2\cdot\text{highest short root of }\mf{g}^\tau, &\text{if }(\mf{g},r)=(A_{2n},2).
\end{cases}\\
\label{eq_theta0check}
\check{\theta}_0=
\begin{cases}
\text{highest short coroot of } \mf{g} & \text{if } r=1\\
\text{highest coroot of } \mf{g}^\tau & \text{if } r>1 \text{ and } (\mf{g},r)\neq (A_{2n},2)\\
\frac{1}{2}\cdot\text{highest coroot of } \mf{g}^\tau& \text{if }  (\mf{g},r) = (A_{2n},2)
\end{cases}
\end{gather}

Recall the normalized symmetric bilinear form $(\cdot\st\cdot)$ on $\hat{\mf{h}}^\tau$ in \cite[\S 6.1]{Kac90}.
This bilinear form induces a map $\nu: \hat{\mf{h}}^\tau\to (\hat{\mf{h}}^\tau)^*$, which satisfies the property in \eqref{eq_dualmap}. By \cite[(6.2.2) and Proposition 5.1]{Kac90}, we have  
\begin{equation}\label{523969}
\textstyle (\theta_0\st\theta_0)=2a_o\check{a}_o=2\check{a}_o; \quad \nu(\check{\theta}_0)=\frac{\theta_0}{a_o\check{a}_o}=\frac{\theta_0}{\check{a}_o}.
\end{equation}

Let $\check{P}$ (resp.\,$\check{Q}$) denote the coweight (resp.\,coroot) lattice of $\mf{g}$. Let $\check{P}_\tau$ (resp.\,$\check{Q}_\tau$) denote the set of $\tau$-coinvariants of $\check{P}$ (resp.\,$\check{Q}$). Let $A:\mf{h} \to \mf{h}^\tau$ be the map defined by 
$  A(h):= \sum_{i=0}^{r-1}\tau^i(h)$, for any $ h\in \mf{h}$. This map descends to a map $\check{P}_\tau\to \mf{h}^\tau$, which is still denoted by $A$ if there is no confusion. Recall \eqref{4529634}, we have a well-defined map $\iota:\check{P}_\tau\to (\mf{h}^\tau)^*$ sending $\mu$ to $\nu(A(\mu))$.

\begin{lem}\label{numu} 
The map $\iota:\check{P}_\tau\to (\mf{h}^\tau)^*$ is induced from the normalized Killing form $(\cdot, \cdot)$ on $\mf{g}$, i.e.\,for any $\mu\in \check{P}_\tau$ and $h\in \mf{h}^\tau$, we have 
$\langle h, \iota(\mu) \rangle = (h, \tilde{\mu})$, where $\tilde\mu$ is a lifting of $\mu$ in $\check{P}$.
\end{lem}
\begin{proof}
By direct computation, we have 
\[\textstyle \langle h, \iota(\mu)\rangle=\langle h, \nu(A(\tilde\mu))\rangle=(h\st A(\tilde\mu))=\frac{1}{r}(h,A(\tilde\mu))=(h,\tilde\mu), \]
where the third equality is by \cite[Corollary 6.4]{Kac90}.
\end{proof}

\subsection{The coherence conjecture and affine Demazure modules}\label{sec_demazure}
Let $G$ be the adjoint group associated to $\mf{g}$. We fix a torus $T$ and a Borel subgroup $B$ corresponding to the Lie subalgebras $\mf{h}\subseteq \mf{b}$. 

We follow the notations in Section \ref{Section_main}. Let $W^\tau$ be the Weyl group of $\mf{g}^\tau$. We can identify $\check{P}$ (resp.\,$\check{Q}$) with the set $X_*(T)$ (resp.\,$X_*(T')$). Recall the affine Weyl group $W_{\mr{aff}}=\check{Q}_\tau\rtimes W^\tau$ and the extended affine Weyl group $\tilde{W}_{\mr{aff}}=X_*(T)_\tau\rtimes W^\tau$. The elements in $\tilde{W}_{\mr{aff}}$ are written as $\varrho^{\mu}w$ for $\mu\in X_*(T)_\tau$ and $w\in W^\tau$. We regard $W_{\mr{aff}}$ as the Weyl group of $\hat{L}(\mf{g},\tau)$, whose action on $(\hat{\mf{h}}^\tau)^*$ is given in \eqref{Weyl_action}.

For any dominant integrable weight $\Lambda$ of $\hat{L}(\mf{g},\tau)$, we denote by $V(\Lambda)$ the highest weight integrable representation of $\hat{L}(\mf{g},\tau)$. For any $\bar\mu\in X_*(T)_\tau$, there exists a unique element $\varsigma_{\kappa}\in \Omega$ such that $\varrho^{\bar{\mu}}\varsigma_{\kappa}^{-1}\in W_{\mr{aff}}$. The representation $V(\varsigma_{\kappa}(\Lambda))$ has an extremal weight vector of weight $\varrho^{\bar{\mu}}(\Lambda)$, which will be denoted by $v_{\varrho^{\bar{\mu}}(\Lambda)}$.

\begin{definition}\label{def_Demazure}
For any $\bar\mu\in X_*(T)_\tau$ and dominant weight $\Lambda$ of $\hat{L}(\mf{g},\tau)$, we define the affine Demazure module $D^\tau(\Lambda,\bar\mu)$ as follows,
\[D^\tau(\Lambda,\bar\mu):=\U(\hat{\mf{b}}^\tau)\cdot v_{\varrho^{\bar{\mu}}(\Lambda)}\subseteq V(\varsigma_{\kappa}(\Lambda)),\]
where $\hat{\mf{b}}^\tau:=\mf{b}^\tau\oplus (t\mf{g}[[t]])^\tau\oplus \mb{C} K\oplus \mb{C} d$ is a Borel subalgebra of $\hat{L}(\mf{g},\tau)$. 

When $\tau$ is trivial, for any coweight $\mu\in X_*(T)$ and dominant weight $\Lambda$ of $\hat{L}(\mf{g})$, the Demazure module will be denoted by $D(\Lambda,\mu)$. In particular, when $\mu$ is dominant, $D(c\Lambda_o,\mu)$ is $\mf{g}[t]$-stable, which will also  be denoted by $D(c,\mu)$.
\end{definition}
\begin{remark}
The $\mf{g}[t]$-stable Demazure module $D(c,\mu)$ has a lowest weight vector whose $\mf{h}$ weight is $-c\iota(\mu)$, see \cite{BH,HY}. 
\end{remark}

Let $c$ be the level of the dominant weight $\Lambda$, i.e. $ c=\langle \Lambda, K \rangle$. 
Let $\lambda$ be the restriction of $\Lambda$ to $\mf{h}^\tau$. Consider the element $\vartheta=\frac{1}{c}\nu^{-1}(\lambda)\in \mf{h}^\tau$. Define $\mf{g}_\lambda:=Z_{\mf{g}^\tau}(\vartheta)$ to be the centralizer of $\vartheta$ in $\mf{g}^\tau$. We write $\Lambda=\sum_{i\in \hat{I}_\tau}n_i \Lambda_i$. Set 
$$Y:=\{i\in\hat{I}_\tau\st n_i\neq 0\}$$ 
\begin{lem}\label{2375423}
The element $\vartheta$ is a rational interior point in the facet $F_{Y}$ associated to $Y$. Furthermore, we have $\mr{Lie}(G_{\vartheta})=\mf{g}_\lambda$.
\end{lem}
\begin{proof}
By direct computation, we have 
$$\textstyle \vartheta=\frac{1}{c}\nu^{-1}(\lambda)=\sum_{i\in I_\tau} \frac{n_i}{c}\nu^{-1}(\omega_i)=\sum_{i\in I_\tau} \frac{n_i}{c}\cdot\frac{\check{a}_i}{a_i}\check{\omega}_i=\sum_{i\in Y} \frac{n_i}{c}\cdot\frac{\check{a}_i}{a_i}\check{\omega}_i,$$
where by convention $\check{\omega}_o=0$.
Since  $\sum_{i\in Y} \frac{n_i\check{a}_i}{c}=1$, $\vartheta$ is a rational interior point in the facet $F_{Y}$ spanned by $\{\frac{\check{\omega}_i}{a_i}\st i\in Y\}$.
By \eqref{eq_Glambda}, we have 
$$ \textstyle \mr{Lie}(G_{\vartheta})=Z_{\mf{g}^\tau}(\frac{m}{r}\vartheta)=\mf{g}_\lambda,$$ 
where $m$ is the minimal positive multiple of $r$ so that $\frac{m}{r}\vartheta\in X_*(T_{\mr{ad}})^\tau$. 
\end{proof}

Now, we state the main theorem of this section. 
\begin{thm}\label{589379}
For any dominant coweight $\mu\in X_*(T)$, there is an isomorphism of $\mf{g}_\lambda$-modules
$$D(c,\mu)\otimes \mb{C}_\lambda\simeq \sum_{w\in W^\tau} D^\tau(\Lambda,w(\bar\mu)),$$
where $\mb{C}_\lambda$ is the $1$-dimensional representation of $\mf{g}_\lambda$.
\end{thm}

Consider the partial affine flag variety $\Fl_Y$ corresponding to $Y$. Recall \eqref{eq_sckappa}, we have an isomorphism $\Fl_Y^{\kappa}\simeq \Fl'_{Y_\kappa}$ of homogeneous spaces of $G'((z))^\tau$. Moreover, $\overline{\Fl}_{Y,\bar\mu}$ is contained in $\Fl_Y^{\kappa}$. Let $\ms{L}'(\varsigma_\kappa(\Lambda))$ be the ample line bundle on $\Fl'_{Y_\kappa}$ associated to $\varsigma_\kappa(\Lambda)$. Pulling back via the isomorphism \eqref{eq_sckappa}, we get a line bundle on $\Fl_Y^{\kappa}$, denoted by $\ms{L}(\Lambda)_\kappa$.

The following theorem is due to Kumar \cite{Ku87} and Mathieu \cite{Ma88}.
\begin{thm}\label{329825}
\begin{enumerate}
\item For any $\bar\mu\in X_*(T)_\tau$, there is an isomorphism of $\hat{\mf{b}}^\tau$-modules
\[ H^0(\overline{\Fl}_{Y,\bar\mu},\ms{L}(\Lambda)_\kappa)^\vee\simeq D^\tau(\Lambda,\bar\mu). \]
\item For any dominant coweight $\mu\in X_*(T)$, there is an isomorphism of $\mf{g}[t]$-modules
$$H^0(\Grb_{G,\mu},L^c)^\vee\simeq D(c,\mu),$$
where $L$ denotes the level one line bundle on $\Gr_G$.
\end{enumerate}
\end{thm}
\begin{proof}
(1) Recall that the isomorphism \eqref{eq_sckappa} restricts to an isomorphism  $\overline{\Fl}_{Y,\bar\mu}\simeq \overline{\Fl}'_{Y_\kappa, \varrho^{\bar\mu}\varsigma_\kappa^{-1}}$, see \eqref{523012}. By \cite[Theorem 8.2.2 (a)]{Ku} and \cite[\S 9f]{PR08},  the space 
$$H^0(\overline{\Fl}_{Y,\bar\mu},\ms{L}(\Lambda)_\kappa)^\vee\simeq H^0(\overline{\Fl}'_{Y_\kappa, \varrho^{\bar\mu}\varsigma_\kappa^{-1}},\ms{L}'(\varsigma_\kappa(\Lambda)))^\vee$$
is isomorphic to the affine Demazure module $D^\tau(\Lambda,\bar\mu)$.

(2)  Let $\tau$ be trivial and take $Y=\{o\}$. Then, $\Fl_{Y}$ is the affine Grassmannian $\Gr_G$. Moreover, for any dominant coweight $\mu\in X_*(T)$, we have $$\Grb_{G,\mu}={}_Y\overline{\Fl}_{Y,\varrho^\mu}=\overline{\Fl}_{Y,\mu},$$ 
where the first equality is by definition and the second equality follows from Corollary \ref{5982347}. Let $\Gr_G^\kappa:=\Fl_Y^\kappa$ be the connected component containing $\Grb_{G,\mu}$ (equivalently $\varrho^\mu\varsigma_\kappa^{-1}\in W_{\mr{aff}}$).  Note that the restriction of $L$ to $\Gr_G^\kappa$ is the line bundle $\ms{L}(\Lambda_o)_\kappa$ on $\Fl_Y^\kappa$.  By part (1), we get 
\[H^0(\Grb_{G,\mu},L^c)^\vee=H^0(\overline{\Fl}_{Y,\mu},\ms{L}(c\Lambda_o)_\kappa)^\vee\simeq D(c\Lambda_o,\mu)=D(c,\mu). \]
\end{proof}

Recall the variety $\mc{A}_Y(\bar\mu)$ defined in \eqref{adm_def},
$$\mc{A}_Y(\bar\mu):=\bigcup_{w\in W^\tau}\overline{\Fl}_{Y,w(\bar\mu)}\subseteq \Fl_Y^\kappa.$$  

\begin{thm}\label{thm_surjective}
Let $\bar\mu\in X_*(T)_\tau$.
\begin{enumerate}
\item \label{5793260} The following restriction map is surjective
\[
H^0(\Fl_Y^\kappa,\ms{L}(\Lambda)_\kappa)\twoheadrightarrow H^0(\mc{A}_Y(\bar\mu),\ms{L}(\Lambda)_\kappa).
\]
\item $H^i(\mc{A}_Y(\bar\mu),\ms{L}(\Lambda)_\kappa)=0$ for $i\geq 1$.
\end{enumerate}
\end{thm}
\begin{proof}
By \cite[Theorem 0.3]{PR08}, any Schubert variety $\overline{\Fl}_{Y,\bar{\mu}}$ is normal and can be identified with a Kac-Moody Schubert variety. Then, Part 1) follows from the same argument as in \cite[Theorem 6.1, Corollary 6.2]{LLM}.  Part 2) similarly follows from the argument in \cite[Theorem 6.4]{LLM}.

 We also sketch a different proof here.  We first consider the partial affine flag variety $\Fl_Y^\kappa$ defined over $\rk$ of characteristic $p \gg 0$.  The variety $\mc{A}_Y(\bar\mu)$ is Frobenius splitting, as Schubert varieties and their unions are Frobenius splitting. Since $\ms{L}(\Lambda)_\kappa$ is ample on $\Fl_Y^\kappa$, Part 1) and 2) are consequences of Frobenius-splitting property.  By upper semi-continuity of sheaf cohomology, we can deduce these results when ${\rm char} (\rk)=0$. 

\end{proof}

\begin{proof}[\textbf{Proof of Theorem \ref{589379}}]
We first note that there is an injective map
\[ H^0(\mc{A}_Y(\bar\mu),\ms{L}(\Lambda)_\kappa)\hookrightarrow \bigoplus_{w\in W^\tau} H^0(\overline{\Fl}_{Y,w(\bar\mu)}, \ms{L}(\Lambda)_\kappa) .\]
Composing with the map in Part (1) of Theorem \ref{thm_surjective} and taking dual, we get 
\[\bigoplus_{w\in W\tau} H^0(\overline{\Fl}_{Y,w(\bar\mu)}, \ms{L}(\Lambda)_\kappa)^\vee \twoheadrightarrow H^0(\mc{A}_Y(\bar\mu),\ms{L}(\Lambda)_\kappa)^\vee \hookrightarrow H^0(\Fl_Y^\kappa,\ms{L}(\Lambda)_\kappa)^\vee.\]
Clearly, the image of the composition map is 
$$\sum_{w\in W^\tau} \mr{Im}\Big(H^0(\overline{\Fl}_{Y,w(\bar\mu)}, \ms{L}(\Lambda)_\kappa)^\vee\hookrightarrow H^0(\Fl_Y^\kappa,\ms{L}(\Lambda)_\kappa)^\vee\Big).$$
We regard $H^0(\overline{\Fl}_{Y,w(\bar\mu)}, \ms{L}(\Lambda)_\kappa)^\vee$ as a subspace of $H^0(\Fl_Y^\kappa,\ms{L}(\Lambda)_\kappa)^\vee$, cf.\,\cite[Theorem 8.2.2 (d)]{Ku}. Then, $$H^0(\mc{A}_Y(\bar\mu),\ms{L}(\Lambda)_\kappa)^\vee\simeq \sum_{w\in W^\tau} H^0(\overline{\Fl}_{Y,w(\bar\mu)}, \ms{L}(\Lambda)_\kappa)^\vee\subseteq H^0(\Fl_Y^\kappa,\ms{L}(\Lambda)_\kappa)^\vee .$$ 

Now, the theorem follows from Theorem \ref{thm_mainthm} and Theorem \ref{329825}. 
\end{proof}

\subsection{Bruhat order on $Y$-admissible sets and maximal elements}
\label{sect_order}
Recall the element $\check{\theta}_0\in \mf{h}^\tau$ defined in \eqref{eq_theta0check} and the map $A:\check{P}_\tau \to \mf{h}^\tau$ in Section \ref{sec_liealgebra}. Let $\check{M}$ be the lattice spanned by $W^\tau\cdot \check\theta_0$. Then, $A$ restricts to an isomorphism $\check{Q}_\tau\simeq  \check{M}$, see the proof in \cite[Proposition A.1.8]{DH}.

Recall the root $\theta^0$ of $\mf{g}$ introduced in Section \ref{section_flag}  (also in
\cite[\S 8.3]{Kac90}). We have $\theta^0|_{\mf{h}^\tau }=\theta_0$, see\,\eqref{eq_theta}. 
Let $\check{\theta}^0$ denote the coroot of $\theta^0$. Define the following element in $\check{Q}$,
\[ \mu^0:=
\begin{cases}
\check{\theta}_0, & \text{ when } r=1\\
\check{\theta}^0, & \text{ when } r\neq 1 \text{ and } X_N^{(r)}\neq A_{2\ell}^{(2)}\\
\check{\alpha}_1+\cdots +\check{\alpha}_{\ell}, & \text{ when } X_N^{(r)}=A_{2\ell}^{(2)},
\end{cases}\]
where we label the vertices of $\mf{g}=A_{2\ell}$ by $1,\ldots,2\ell$.

Let $\mu_0\in \check{Q}_\tau$ be the element such that $A(\mu_0)=\check{\theta}_0\in \check{M}$, and let $\overline{\mu^0}$ denote the image of $\mu^0$ under the map $\check{P}\to \check{P}_\tau$.

\begin{lem}
$\mu_0=\overline{\mu^0}$.
\end{lem}
\begin{proof}
Since $A:\check{Q}_\tau\simeq  \check{M}$ is an isomorphism, it suffices to show that  $A(\overline{\mu^0})=\check{\theta}_0$.

When $r=1$, the map $A$ is the identity map. Hence, $A(\overline{\mu^0})=\overline{\mu^0}=\mu^0=\check{\theta}_0$.

When $r>1$ and $X_N^{(r)}\neq A_{2\ell}^{(2)}$, since $\mf{g}$ is simply-laced, we have $(\theta^0,\theta^0)=2$ under the normalized Killing form on $\mf{g}$. For any $h\in \mf{h}^\tau$, we have 
$$ \langle h, \nu(A(\overline{\mu^0}))\rangle = (h, \mu^0)=(h, \check{\theta}^0)=\langle h, \theta^0\rangle,$$ 
where the first equality is by Lemma \ref{numu}. 
Therefore, $\nu(A(\overline{\mu^0}))=\theta^0|_{\mf{h}^\tau}=\theta_0$. By equation \eqref{523969}, we get $A(\overline{\mu^0})=\nu^{-1}(\theta_0)=\check{\theta}_0$.

When $X_N^{(r)}= A_{2\ell}^{(2)}$, the simple coroots of $\mf{g}^\tau$ are $\check{\beta}_i=\check{\alpha}_{i}+\check{\alpha}_{2\ell+1-i}, (1\leq i<\ell)$, and $\check{\beta}_\ell=2(\check{\alpha}_{\ell}+\check{\alpha}_{\ell+1})$, see \cite[\S8.3 Case 5]{Kac90}. Hence, 
\[\textstyle\pushQED{\qed} A(\overline{\mu^0})=\mu^0+\tau(\mu^0)=\check{\alpha}_1+\cdots +\check{\alpha}_{2\ell}=\frac{1}{2}\cdot(2\check{\beta}_1+\cdots +2\check{\beta}_{\ell-1}+\check{\beta}_\ell)=\check\theta_0. \qedhere\]
\end{proof}

\begin{remark}
By \cite{BH,Zhu15},  we can regard $X_*(T)_\tau$ as the weight lattice of the reductive group $(G^\vee)^\tau$ where $G^\vee$ is the Langlands dual group of $G$ and $\tau$ is the diagram automorphism on $G^\vee$.  Then, $\mu_0$ is the highest short root of $(G^\vee)^\tau$.  
\end{remark}

Recall the affine Weyl group $W_{\mr{aff}}=\check{Q}_\tau\rtimes W^\tau$ and the extended affine Weyl group $\tilde{W}_{\mr{aff}}=\check{P}_\tau\rtimes W^\tau$. For any $i\in I_\tau$, let $r_i\in W^\tau$ denote the simple reflection associated to $i$. Set $$r_o:=\varrho^{-\mu_0}r_{\theta_0}\in W_{\mr{aff}}.$$
Then, $W_{\mr{aff}}$ is generated by the affine simple reflections $\{r_i\st i\in \hat{I}_\tau\}$. By \eqref{Weyl_action}, the action of $r_o$ on $(\hat{\mf{h}}^\tau)^*$ is given by 
$r_o(x)=x-\langle \check\alpha_o,x\rangle \alpha_o$.

As in Section \ref{Section_main}, we fix a non-empty subset $Y\subseteq \hat{I}_\tau$. Set 
$$\breve{Y}:=\hat{I}_\tau\setminus Y.$$
Let $W_{\breve{Y}}$ be the parabolic subgroup of $W_{\mr{aff}}$ generated by the simple reflections $r_i$, $i\in \breve{Y}$. Let $\preccurlyeq$ be the Bruhat order on $W_{\mr{aff}}$ and $\tilde{W}_{\mr{aff}}$. 
\begin{definition}
For each element $w\in \tilde{W}_{\mr{aff}}$, the coset $wW_{\breve{Y}}$ has a unique element with minimal length, denoted by $w_{\min}$. For any $w,v\in  \tilde{W}_{\mr{aff}}$, we define $w\preccurlyeq_Y v$ if $w_{\min}\preccurlyeq v_{\min}$.
\end{definition}

Let $\Lambda=\sum_{i\in Y}n_i\Lambda_i$ be a dominant integrable weight of $\hat{L}(\mf{g},\tau)$ with $n_i>0$. For any $\mu\in \check{P}_\tau$, recall the affine Schubert variety $\overline{\Fl}_{Y,\mu}$ and the affine Demazure module $D^\tau(\Lambda,\mu)$. The following result is well-known, cf.\,\cite[Proposition 3.4 and 3.7]{BH22}
\begin{prop}\label{demazure_order}
Given any $\mu_1, \mu_2\in \check{P}_\tau$ such that $\mu_1-\mu_2\in \check{Q}_\tau$, the following are equivalent: 
\begin{enumerate}
    \item $\varrho^{\mu_1}\preccurlyeq_Y \varrho^{\mu_2}$,
    \item $D^\tau(\Lambda,\mu_1)\subseteq D^\tau(\Lambda,\mu_2)$,
    \item $\overline{\Fl}_{Y,\mu_1} \subseteq \overline{\Fl}_{Y,\mu_2}$. \qed
\end{enumerate}
\end{prop}

For any $w\in W_{\mr{aff}}$, let $\bar{w}\in W^\tau$ denote its image under the composition map $W_{\mr{aff}}\to \check{Q}_\tau\backslash W_{\mr{aff}}\simeq W^\tau$. For example, $\bar{r}_o=r_{\theta_0}$ and $\bar{r_i}=r_i$ for $i\neq o$.

\begin{lem}\label{lem_r_i}
Let $\mu\in \check{P}_\tau$.
\begin{enumerate}
\item For any $w\in W_{\mr{aff}}$, we have $w\varrho^\mu w^{-1}=\varrho^{\bar{w}(\mu)}$.
\item Recall the map $\iota: \check{P}_\tau\to (\mf{h}^\tau)^*$. We have 
$$r_i(\iota(\mu))=
\begin{cases}
\iota(\bar{r_i}(\mu)) & \text{ if } i\neq o,\\
\iota(\bar{r}_o(\mu))+\langle\check\theta_0,\iota(\mu)\rangle\delta & \text{ if } i=o.
\end{cases}$$
\end{enumerate}
\end{lem}
\begin{proof}
For any $w\in W_{\mr{aff}}$, we have $w=\varrho^{\eta}\bar{w}$ for some $\eta\in \check{Q}_\tau$. Then, 
\[
    w\varrho^\mu w^{-1}=\varrho^{\eta}\bar{w}\varrho^\mu \bar{w}^{-1}\varrho^{-\eta}=\varrho^{\eta}\varrho^{\bar{w}(\mu)} \varrho^{-\eta}=\varrho^{\bar{w}(\mu)}.
\]
This proves part (1). Part (2) follows from the equation \eqref{Weyl_action} and the fact that $w(\iota(\mu))=\iota(w(\mu))$ for any $w\in W^\tau$.
\end{proof}
\begin{lem}\label{527946}
Let $\mu\in \check{P}_\tau$ and let $\check{\alpha}_i$ be the simple coroot associated to $i\in \hat{I}_\tau$. Then, $\varrho^\mu\prec\varrho^\mu r_i $ if and only if $\langle \check{\alpha}_i,\iota(\mu)\rangle \geq 0$.
\end{lem}
\begin{proof}
Let $\varsigma_\kappa\in \Omega$ be the unique length zero element such that $\varrho^\mu \varsigma_\kappa^{-1}\in W_{\mr{aff}}$. Then, 
$$\varrho^\mu=\varrho^\mu\varsigma_\kappa^{-1}\varsigma_\kappa, \quad \varrho^\mu r_i =\varrho^\mu \varsigma_\kappa^{-1} r_{j}\varsigma_\kappa \ \in W_{\mr{aff}}\cdot \varsigma_\kappa$$ 
where $j = \varsigma_\kappa(i)$. Thus, $\varrho^\mu\prec\varrho^\mu r_i $ if and only if $ \varrho^\mu \varsigma_\kappa^{-1} \prec \varrho^\mu\varsigma_\kappa^{-1}r_{j}$. By \cite[Lemma 3.11 (a)]{Kac90}, the latter one is equivalent to $\varrho^\mu \varsigma_\kappa^{-1}(\alpha_j) > 0$. Note that 
$$\varrho^\mu \varsigma_\kappa^{-1}(\alpha_j) = \varrho^\mu(\alpha_i) = \alpha_i+ (\alpha_i\st \iota(\mu))\delta.$$
Therefore, $\varrho^\mu\prec\varrho^\mu r_i $ if and only if $(\alpha_i\st \iota(\mu))\geq 0$, equivalently $\langle \check{\alpha}_i,\iota(\mu)\rangle \geq 0$.
\end{proof}
\begin{lem}\label{union_simplify}
Given any $\mu\in \check{P}_\tau$ and $i\in \breve{Y}$, we have $\varrho^{\bar{r}_i(\mu)}\preccurlyeq_Y \varrho^{\mu}$ if and only if $\langle \check\alpha_i,\iota(\mu)\rangle\geq 0$. Furthermore, $(\varrho^{\bar{r}_i(\mu)})_{\min}= (\varrho^{\mu})_{\min}$ if and only if $\langle \check\alpha_i,\iota(\mu)\rangle= 0$. 
\end{lem}
\begin{proof}
We first assume that $\langle \check\alpha_i,\iota(\mu)\rangle\geq 0$.  By Lemma \ref{527946}, we have $\varrho^\mu\prec\varrho^\mu r_i $. Since we always have $\ell(\varrho^\mu r_i)= \ell(\varrho^\mu)\pm 1$, it follows that $\ell(\varrho^\mu r_i)= \ell(\varrho^\mu)+1$. Note that $\ell(r_i\varrho^\mu r_i)=\ell(\varrho^{\bar{r}_i\mu})=\ell(\varrho^\mu)=\ell(\varrho^\mu r_i)-1$. Thus, $r_i\varrho^\mu r_i\prec \varrho^\mu r_i$. By \cite[Lemma 1.3.18]{Ku}, we get $(r_i \varrho^\mu r_i)_{\min}\preccurlyeq (\varrho^\mu r_i)_{\min}$. Since $i\in \breve{Y}$, we get 
$(\varrho^{\bar{r}_i(\mu)})_{\min} \preccurlyeq (\varrho^\mu)_{\min}$,
i.e.\,$\varrho^{\bar{r}_i(\mu)}\preccurlyeq_Y \varrho^{\mu}$.

Conversely, assume that $\varrho^{\bar{r}_i(\mu)}\preccurlyeq_Y \varrho^{\mu}$. Suppose that $\langle \check\alpha_i,\iota(\mu)\rangle< 0$. By Lemma \ref{527946}, we have $\varrho^\mu\succ\varrho^\mu r_i $.  Hence, $\ell(\varrho^\mu r_i)= \ell(\varrho^\mu)-1$. Using the same argument as above,  we have $\varrho^\mu r_i \prec r_i\varrho^\mu r_i $. Thus,  $(\varrho^\mu)_{\min}\preccurlyeq (\varrho^{\bar{r}_i(\mu)})_{\min}$. Combining with the assumption that $\varrho^{\bar{r}_i(\mu)}\preccurlyeq_Y \varrho^{\mu}$, we must have $(\varrho^\mu)_{\min}= (\varrho^{\bar{r}_i(\mu)})_{\min}$.  Thus, $\varrho^{\mu-\bar{r_i}(\mu)}\in W_{\breve{Y}}$. Note that $W_{\breve{Y}}$ is a finite group and every element has finite order. It follows that $\mu-\bar{r_i}(\mu)=0$. By part (2) of Lemma \ref{lem_r_i}, we conclude that $\big(\iota(\mu)-r_i(\iota(\mu))\st \alpha_i\big)=0$. Hence, $\langle \check\alpha_i,\iota(\mu)\rangle= 0$, which is a contradiction.  Therefore, $\langle \check\alpha_i,\iota(\mu)\rangle\geq 0$. This finishes the proof of the first part of the lemma.

From the converse part of the above proof, we can see that $(\varrho^\mu)_{\min}= (\varrho^{\bar{r}_i(\mu)})_{\min}$ implies $\langle \check\alpha_i,\iota(\mu)\rangle= 0$. Conversely, if $\langle \check\alpha_i,\iota(\mu)\rangle= 0$, then $r_i(\iota(\mu))=\iota(\mu)$. By Lemma \ref{lem_r_i},  $\iota(\bar{r_i}(\mu))=\iota(\mu)$, equivalently $\bar{r_i}(\mu)=\mu$. Thus, $\varrho^{\bar{r}_i(\mu)}=\varrho^\mu$.
\end{proof}
We remark that $\langle \check\alpha_o,\iota(\mu)\rangle=\langle -\check\theta_0,\iota(\mu)\rangle$ for any $\mu\in \check{P}_\tau$.
\begin{definition}
Set $\beta_i:=\alpha_i$ for $i\neq o$ and $\beta_o:=-\theta_0$. Then, the set $\Pi_{\breve{Y}} = \{\beta_i\}_{i\in \breve{Y}}$ forms a base of the root system of $G^{\sigma,\circ}$, see Proposition \ref{eq_levi}. We denote this root system by $\Phi_{\breve{Y}}$, where $\breve{Y}:=\hat{I}_\tau\setminus Y$. Let  $\Phi_{\breve{Y}}^+$ and $\Phi_{\breve{Y}}^-$ denote the sets of positive roots and negative roots of  $\Phi_{\breve{Y}}$ with respect to the base $\Pi_{\breve{Y}}$.

We say $\mu \in \check{P}_\tau$ is {\bf $\breve{Y}$-dominant} if $\langle \check{\beta},\iota(\mu)\rangle\geq 0$ for any $\beta\in \Phi_{\breve{Y}}^+$.
\end{definition}

\begin{lem}\label{432895346}
Assume that $\eta \in \check{P}_\tau$ is $\breve{Y}$-dominant. Then, $(\varrho^\eta)_{\min} = \varrho^\eta$.
\end{lem}
\begin{proof}
We prove it by contradiction. Suppose that $(\varrho^\eta)_{\min} \neq \varrho^\eta$. Then, there is a unique nontrivial element $w\in W_{\breve{Y}}$ such that $\varrho^\eta = (\varrho^\eta)_{\min} \cdot w$ and $\ell(\varrho^\eta)=\ell((\varrho^\eta)_{\min})+\ell(w)$. Let $w= r_{i_1}\cdots r_{i_k}$ be a reduced word expression, where $i_1,\ldots,i_k\in \breve{Y}$. Note that $k=\ell(w) \geq 1$. Moreover, $\varrho^\eta r_{i_k} = (\varrho^\eta)_{\min} \cdot r_{i_1}\cdots r_{i_{k-1}}$. Thus, $$\ell(\varrho^\eta r_{i_k})= \ell((\varrho^\eta)_{\min})+k-1< \ell(\varrho^\eta).$$
However, since $\eta$ is $\breve{Y}$-dominant, we have $\langle \check{\beta}_{i_k},\iota(\eta)\rangle\geq 0$, equivalently $\langle \check{\alpha}_{i_k},\iota(\eta)\rangle\geq 0$. By Lemma \ref{527946}, we get $\ell(\varrho^\eta r_{i_k})>\ell( \varrho^\eta)$, which is a contradiction. Therefore, $(\varrho^\eta)_{\min} = \varrho^\eta$.
\end{proof}

\begin{prop}\label{append_order}
Let $\mu\in \check{P}_\tau$ and $w\in W_{\breve{Y}}$. If $\langle \check{\beta},\iota(\mu)\rangle\geq 0$ for any $\beta\in \Phi_{\breve{Y}}^+$ with $\bar{w}(\beta)\in \Phi_{\breve{Y}}^-$, then we have
\begin{enumerate}
    \item $\varrho^{\bar{w}(\mu)}\preccurlyeq_Y \varrho^{\mu}$,
    \item $D^\tau(\Lambda,\bar{w}(\mu))\subseteq D^\tau(\Lambda,\mu)$,
    \item $\overline{\Fl}_{Y,\bar{w}(\mu)} \subseteq \overline{\Fl}_{Y,\mu}$.
\end{enumerate}
\end{prop}
\begin{proof}
By Proposition \ref{demazure_order}, it suffices to prove Part (1). Let $\bar{W}_{\breve{Y}}$ be a subgroup of $W^\tau$ defined as follows,
$$\bar{W}_{\breve{Y}}:=\{\bar{w} \st w\in W_{\breve{Y}}\}.$$ 
Then, $\bar{W}_{\breve{Y}}$ is generated by $\{\bar{r}_{i} \st i\in \breve{Y}\}$. Note that for any $i\in \breve{Y}$ and $\beta\in  \Phi_{\breve{Y}}$, we have 
$\bar{r_i}(\beta)=\beta-\langle\check\beta_i,\beta \rangle \beta_i$. Hence, $\bar{W}_{\breve{Y}}$ is the Weyl group of $\Phi_{\breve{Y}}$ with simple generators $\{\bar{r_i}\st i\in \breve{Y}\}$. Let $\bar{w}=\bar{r}_{i_n}\cdots \bar{r}_{i_1}$ be a reduced word in $\bar{W}_{\breve{Y}}$, where $i_1,\cdots, i_n\in \breve{Y}$. Set $\bar{w}_0=e$ and $\bar{w}_k:=\bar{r}_{i_{k}}\bar{w}_{k-1}$ for $1\leq k\leq n$. Then, for any $k$, we have
$$\bar{w}(\bar{w}_{k-1}^{-1}(\beta_{i_k}))=\bar{r}_{i_n}\cdots \bar{r}_{i_k}(\beta_{i_k})\in \Phi_{\breve{Y}}^-.$$
By assumption,  $\langle \bar{w}_{k-1}^{-1}(\check{\beta}_{i_k}),\iota(\mu)\rangle \geq 0$. By Lemma \ref{union_simplify}, $\varrho^{\bar{w}_{k}\mu}\preccurlyeq_Y \varrho^{\bar{w}_{k-1}\mu}$. It follows that 
\[\pushQED{\qed} \varrho^{\bar{w}(\mu)}=\varrho^{\bar{w}_n(\mu)}\preccurlyeq_Y \cdots\preccurlyeq_Y \varrho^{\bar{w}_1(\mu)}\preccurlyeq_Y \varrho^{\mu}.\qedhere \]
\end{proof}

For any $\mu\in \check{P}_\tau$, let $S_{Y,\mu}$ be the following subset of $\check{P}_\tau$,
$$ S_{Y,\mu}:=\{ \eta \in W^\tau \mu \st \eta \text{ is } \breve{Y}\text{-dominant}\}.$$
Since $\bar{W}_{\breve{Y}}$ is the Weyl group of $\Phi_{\breve{Y}}$, the $\bar{W}_{\breve{Y}}$-orbit at $\mu$ contains a unique $\breve{Y}$-dominant element. Thus, 
\begin{equation}\label{543695}
W^\tau \mu=\bigsqcup_{\eta\in S_{Y,\mu}} \bar{W}_{\breve{Y}}\eta
\end{equation}
\begin{cor}\label{5982347}
Let $\mu\in \check{P}_\tau$.
\begin{enumerate}
\item Assume that $\mu$ is $\breve{Y}$-dominant. For any $w\in W_{\breve{Y}}$, we have $\overline{\Fl}_{Y,\bar{w}(\mu)} \subseteq \overline{\Fl}_{Y,\mu}$. Moreover, ${}_Y\overline{\Fl}_{Y,\varrho^\mu}=\overline{\Fl}_{Y,\mu}.$
\item  The union in \eqref{adm_def} can be simplified as follows,
$$\mc{A}_Y(\mu)=\bigcup_{\eta\in S_{Y,\mu}} \overline{\Fl}_{Y,\eta},$$
Moreover, in this decomposition, $\{\overline{\Fl}_{Y,\eta}\st \eta\in S_{Y,\mu}\}$ are pairwise distinct irreducible components of $\mc{A}_Y(\mu)$ with dimension $\langle \mu^+, 2\rho\rangle$, where $\mu^+$ denotes the dominant element in the orbit $W^\tau\mu$. 
\item Among Demazure modules $\{D^\tau(\Lambda, w\mu )  \,|\, w\in W^\tau\}$, the collection $\{D^\tau(\Lambda,\eta)\st \eta\in S_{Y,\mu}\}$ consists of  pairwise distinct maximal Demazure modules.
\end{enumerate}
\end{cor}
\begin{proof}
Part (1) follows from Proposition \ref{append_order}, and the fact that ${}_Y\overline{\Fl}_{Y,\varrho^\mu}=\bigcup_{w\in W_{\breve{Y}}} \overline{\Fl}_{Y,\bar{w}(\mu)}$, see \eqref{eq_union} in Lemma \ref{7842349}. For part (2),  it suffices to show that the left hand side is contained in the right hand side. Using the decomposition \eqref{543695}, we have
$$\mc{A}_Y(\mu)=\bigcup_{\eta\in W^\tau\mu}\overline{\Fl}_{Y,\eta}= \bigcup_{\eta\in S_{Y,\mu}} \bigcup_{w\in \bar{W}_{\breve{Y}}} \overline{\Fl}_{Y,w(\eta)} \subseteq \bigcup_{\eta\in S_{Y,\mu}} \overline{\Fl}_{Y,\eta},$$
where the last inclusion is from part (1). Note that $\mc{A}_Y(\mu)=\mc{A}_Y(\mu^+)$ and $\dim \mc{A}_Y(\mu^+) = \langle \mu^+, 2\rho\rangle$, where $\rho$ is the half sum of positive roots of $G^\tau$. Moreover, by Lemma \ref{432895346}, we have $(\varrho^\eta)_{\min} = \varrho^\eta$. Then, 
$$\dim \overline{\Fl}_{Y,\eta} = \ell((\varrho^\eta)_{\min})= \ell(\varrho^\eta) = \langle \mu^+, 2\rho\rangle=\dim \mc{A}_Y(\mu),$$
where the first equality is by \cite[Proposition 0.1]{Ri13}. Therefore, each $\overline{\Fl}_{Y,\eta}$ is an irreducible component of $\mc{A}_Y(\mu)$. Moreover, for distinct $\eta,\eta'\in S_{Y,\mu}$, we have $(\varrho^\eta)_{\min} = \varrho^\eta \neq \varrho^{\eta'} =(\varrho^{\eta'})_{\min}$. By Proposition \ref{demazure_order}, we have $\overline{\Fl}_{Y,\eta}\neq \overline{\Fl}_{Y,\eta'}$.
This proves part (2). The third part can be proved in the same way.
\end{proof}

\appendix
\section{Examples of Theorem \ref{589379} }\label{293489}
In this appendix, we provide some examples to illustrate Theorem \ref{589379}.
\subsection{Untwisted case}
\label{ex_demazure}

Let $\mf{g}=\mr{sl}_2$ and let $\tau$ be trivial. Then, $\hat{L}(\mf{g})$ is the untwisted affine algebra of type $A_1^{(1)}$. We label the vertex of the Dynkin diagram of $\mf{g}$ by $1$. Then, $\hat{I}_\tau=\{o,1\}$. Let $\alpha:=\alpha_1$ and $\check{\alpha}:=\check{\alpha}_1$ be the simple root and coroot of $\mf{g}$. Let $\omega:=\omega_1$ and $\check{\omega}:=\check{\omega}_1$ denote the fundamental weight and coweight. There are two level one fundamental weights $\Lambda_o$ and $\Lambda_1=\Lambda_o+\omega$ of $\hat{L}(\mf{g})$. 

\begin{enumerate}
\item We take $\Lambda=\Lambda_1$ and $\mu=\check{\alpha}$ in Theorem \ref{589379}. In this case, $c=1$ and $\lambda=\omega$. We have the following isomorphism of $\mf{h}$-modules,
$$D(1,\check\alpha)\otimes \mb{C}_{\omega}\simeq D(\Lambda_1,-\check\alpha)+D(\Lambda_1,\check\alpha).$$
Note that $Y=\{1\}$ and $-\check\alpha$ is $\breve{Y}$-dominant. By Corollary \ref{5982347}, the above isomorphism can be simplified to the isomorphism $D(1,\check\alpha)\otimes \mb{C}_{\omega}\simeq D(\Lambda_1,-\check\alpha)$. In fact, by \eqref{Weyl_action}, we have 
\[\varrho^{\pm\check\alpha}(\Lambda_o)=\Lambda_o\mp\alpha-\delta;\quad \varrho^{\check\alpha}(\Lambda_1)=\Lambda_1-\alpha,\quad \varrho^{-\check\alpha}(\Lambda_1)=\Lambda_1+\alpha-2\delta.\] 
The modules $D(1,\check\alpha):=D(\Lambda_o,\check\alpha)$ and $D(\Lambda_1,-\check\alpha)$ are $4$-dimensional $\hat{\mf{b}}$-modules as shown in Figure \ref{pic_1a} and \ref{pic_1b}. 
\begin{figure}[H]
\centering
\begin{subfigure}[b]{0.5\textwidth}
\centering
\begin{tikzpicture}[scale=0.9]
\draw[thick, scale=1, domain=-1.6:1.6, smooth, variable=\x, black] plot ({\x}, {-\x*\x});
\fill (0,0) circle (2.5pt);
\draw(0,0) node[anchor=south]{$\Lambda_o$};
\foreach \p in {-1,...,1}{\fill (\p,-1) circle (2.5pt);}
\draw(-1,-1) node[anchor=east]{$\hspace{-1em}\varrho^{\check\alpha}(\Lambda_o)$};
\foreach \p in {-1,...,1}{\fill (\p,-2) circle (0.5pt);}
\end{tikzpicture}
\caption{$D(1,\check\alpha)$}\label{pic_1a}
\end{subfigure}%
\begin{subfigure}[b]{0.5\textwidth}
\centering
\begin{tikzpicture}[scale=0.9]
\draw[thick, scale=1, domain=-1.65:1.65, smooth, variable=\x, black] plot ({\x}, {0.25-\x*\x});
\foreach \p in {-0.5,0.5}{\fill (\p,0) circle (2.5pt);}
\draw(0.5,0) node[anchor=south west]{$\Lambda_1$};
\draw(-0.5,0) node[anchor=east]{$\varrho^{\check\alpha}(\Lambda_1)$};
\fill (-0.5,-1) circle (0.5pt);
\fill (0.5,-1) circle (2.5pt);
\foreach \p in {-1.5,...,0.5}{\fill (\p,-2) circle (0.5pt);}
\fill (1.5,-2) circle (2.5pt);
\draw(1.5,-2) node[anchor=west]{$\varrho^{-\check\alpha}(\Lambda_1)\hspace{-2em}$};
\end{tikzpicture}
\caption{$D(\Lambda_1,-\check\alpha)$}\label{pic_1b}
\end{subfigure}
\caption{}\label{pic_1}
\end{figure}

\item We take $\Lambda=\Lambda_o+\Lambda_1=2\Lambda_o+\omega$ and $\mu=\check{\omega}$ in Theorem \ref{589379}. In this case, $c=2,\lambda=\omega$. We get an isomorphism of $\mf{h}$-modules
$$D(2, \check{\omega})\otimes \mb{C}_\omega \simeq D(\Lambda_o+\Lambda_1,-\check\omega)+ D(\Lambda_o+\Lambda_1,\check\omega).$$
In fact, we have $\Omega=\{\mr{id}, \varsigma_{\check{\omega}}\}$, where $\varsigma_{\check{\omega}}=\varrho^{-\check{\omega}}s_\alpha$. Note that $\varsigma_{\check{\omega}}^{-1}=\varsigma_{\check{\omega}}$. Then, 
$\varrho^{\check{\omega}}\varsigma_{\check{\omega}}^{-1}=s_\alpha\in W_{\mr{aff}}$ and $\varrho^{-\check{\omega}}\varsigma_{\check{\omega}}^{-1}=\varrho^{-\check{\alpha}}s_\alpha\in W_{\mr{aff}}$. Hence, the modules $D(\Lambda_o+\Lambda_1,-\check\omega)$ and $D(\Lambda_o+\Lambda_1,\check\omega)$ are subspaces of $V(\varsigma_{\check{\omega}}(\Lambda))=V(\Lambda)$. By \eqref{Weyl_action}, we have 
$$\varrho^{\check{\omega}}(\Lambda)=\Lambda-\alpha,\quad \varrho^{-\check{\omega}}(\Lambda)=\Lambda+\alpha-\delta.$$
The modules $D(\Lambda_o+\Lambda_1,-\check\omega)$ and $D(\Lambda_o+\Lambda_1,\check\omega)$ are $2$-dimensional spaces as shown in Figure \ref{pic_2b} labeled by red and blue colors respectively. They intersect at a $1$-dimensional subspace. Hence, the sum of two spaces $D(\Lambda_o+\Lambda_1,-\check\omega)+ D(\Lambda_o+\Lambda_1,\check\omega)$ is a $3$-dimensional $\hat{\mf{b}}$-module. 

On the other hand,  $D(2,\check{\omega})$ is a subspace of $V(\varsigma_{\check{\omega}}(2\Lambda_o))=V(2\Lambda_1-\frac{1}{2}\delta)$ generated by an extremal weight vector of weight
$\varrho^{\check{\omega}}(2\Lambda_o)=2\Lambda_o-\alpha-\frac{1}{2}\delta.$ 
Thus,  $D(2,\check{\omega})$ is a $3$-dimensional $\mf{g}[t]$-module shown in Figure \ref{pic_2a}.

\begin{figure}[H]
\centering
\begin{subfigure}[b]{0.4\textwidth}
\centering
\begin{tikzpicture}[scale=0.9]
\draw[thick, scale=1, domain=-2:2, smooth, variable=\x, black] plot ({\x}, {0.5-0.5*\x*\x});
\foreach \p in {-1,...,1}{\fill (\p,0) circle (2.5pt);}
\draw(1,0) node[anchor=south west]{$2\Lambda_1-\frac{1}{2}\delta$};
\draw(-1,0) node[anchor=east]{$2\Lambda_o-\alpha-\frac{1}{2}\delta$};
\foreach \p in {-1,0,1}{\fill (\p,-1) circle (0.5pt);}
\end{tikzpicture}
\caption{$D(2,\check{\omega})\hspace{-2em}$}\label{pic_2a}
\end{subfigure}%
\begin{subfigure}[b]{0.6\textwidth}
\centering
\begin{tikzpicture}[scale=0.9]
\draw[thick, scale=1, domain=-1.9:1.9, smooth, variable=\x, black] plot ({\x}, {0.15-0.5*\x*\x});
\foreach \p in {-0.5}{\fill[blue] (\p,0) circle (2.5pt);}
\draw(0.5,0) node[anchor=south west]{$\Lambda_o+\Lambda_1$};
\filldraw[color=red, fill=blue, very thick] (0.5,0) circle (2.5pt);
\fill[red] (1.5,-1) circle (2.5pt);
\foreach \p in {-1.5,...,1.5}{\fill (\p,-1) circle (0.5pt);}
\draw(-0.5,0) node[blue,anchor=east]{$\varrho^{\check{\omega}}(\Lambda_o+\Lambda_1)$};
\draw(1.5,-1) node[red,anchor=west]{$\varrho^{-\check\omega}(\Lambda_o+\Lambda_1)\hspace{-3em}$};
\end{tikzpicture}
\caption{$D(\Lambda_o+\Lambda_1,-\check\omega)+ D(\Lambda_o+\Lambda_1,\check\omega)$}\label{pic_2b}
\end{subfigure}
\caption{}\label{pic_2}
\end{figure}

\item We take $\Lambda=\Lambda_o+\Lambda_1$ and $\mu=\check{\alpha}$ in Theorem \ref{589379}. Then, $c=2,\lambda=\omega$. We get an isomorphism of $\mf{h}$-modules
$$D(2, \check{\alpha})\otimes \mb{C}_\omega \simeq D(\Lambda_o+\Lambda_1,-\check\alpha)+ D(\Lambda_o+\Lambda_1,\check\alpha).$$
In fact, by \eqref{Weyl_action}, we have 
$$\varrho^{\check\alpha}(2\Lambda_o)=2\Lambda_o-2\alpha-2\delta;\quad \varrho^{\check\alpha}(\Lambda)=\Lambda-2\alpha-\delta,\quad \varrho^{-\check\alpha}(\Lambda)=\Lambda+2\alpha-3\delta.$$
The module $D(2,\check\alpha):=D(2\Lambda_o,\check\alpha)$ is a $9$-dimensional $\mf{g}[t]$-module as shown in Figure \ref{pic_3a}. Moreover, the  modules $D(\Lambda_o+\Lambda_1,-\check\alpha)$ and $D(\Lambda_o+\Lambda_1,\check\alpha)$ are $6$-dimensional spaces shown in Figure \ref{pic_3b} labeled by red and blue colors respectively. They intersect at a $3$-dimensional subspace. Hence, the sum of two spaces $D(\Lambda_o+\Lambda_1,-\check\alpha)+ D(\Lambda_o+\Lambda_1,\check\alpha)$ is a $9$-dimensional $\hat{\mf{b}}$-module. 
\begin{figure}[H]
\centering
\begin{subfigure}[b]{0.4\textwidth}
\centering
\begin{tikzpicture}[scale=0.8]
\draw[thick, scale=1, domain=-2.6:2.6, smooth, variable=\x, black] plot ({\x}, {-0.5*\x*\x});
\fill (0,0) circle (2.5pt);
\draw(0,0) node[anchor=south]{$2\Lambda_o$};
\foreach \p in {-1,...,1}{\fill (\p,-1) circle (2.5pt);}
\foreach \p in {-2,...,2}{\fill (\p,-2) circle (2.5pt);}
\draw(-2,-2) node[anchor=east]{$\varrho^{\check\alpha}(2\Lambda_o)$};
\foreach \p in {-2,...,2}{\fill (\p,-3) circle (0.5pt);}
\end{tikzpicture}
\caption{$D(2,\check\alpha)\hspace{-2em}$}\label{pic_3a}
\end{subfigure}%
\begin{subfigure}[b]{0.6\textwidth}
\centering
\begin{tikzpicture}[scale=0.8]
\draw[thick, scale=1, domain=-2.7:2.7, smooth, variable=\x, black] plot ({\x}, {0.15-0.5*\x*\x});
\foreach \p in {-0.5,0.5}{\filldraw[color=red, fill=blue, very thick] (\p,0) circle (2.5pt);}
\draw(0.5,0) node[anchor=south west]{$\Lambda_o+\Lambda_1$};
\foreach \p in {-1.5,...,-0.5}{\fill[blue] (\p,-1) circle (2.5pt);}
\fill[blue] (0.4,-1) circle (2.5pt);
\fill[red] (0.6,-1) circle (2.5pt);
\filldraw[color=red, fill=blue, very thick] (1.5,-1) circle (2.5pt);
\draw(-1.5,-1) node[blue, anchor=east]{$\varrho^{\check\alpha}(\Lambda_o+\Lambda_1)$};
\fill[red] (1.5,-2) circle (2.5pt);
\foreach \p in {-1.5,...,1.5}{\fill (\p,-2) circle (0.5pt);}
\fill[red] (2.5,-3) circle (2.5pt);
\foreach \p in {-2.5,...,2.5}{\fill (\p,-3) circle (0.5pt);}
\draw(2.5,-3) node[red, anchor=west]{$\varrho^{-\check\alpha}(\Lambda_o+\Lambda_1)\hspace{-2em}$};
\end{tikzpicture}
\caption{$D(\Lambda_o+\Lambda_1,-\check\alpha)+ D(\Lambda_o+\Lambda_1,\check\alpha)$}\label{pic_3b}
\end{subfigure}
\caption{}\label{pic_3}
\end{figure}
\end{enumerate}

\subsection{Twisted case}
Let $\mf{g}=\mr{sl}_3$, $G=\mr{PGL}_3$, and let $\hat{L}(\mf{g})$ be the untwisted affine algebra of type $A_2^{(1)}$. We label the vertices of  the Dynkin diagram of $\mf{g}$ by $1,2$. Then, $\hat{I}=\{o,1,2\}$. Let $\alpha_1,\alpha_2$ and $\check\alpha_1,\check\alpha_2$ denote the simple roots and coroots of $\mf{g}$ respectively. Let $\check\omega_1,\check\omega_2$ denote the fundamental coweights of $\mf{g}$.  Then, $\Omega=\{\mr{id}, \varsigma_{\check{\omega}_1}, \varsigma_{\check{\omega}_2}\}$, where $\varsigma_{\check{\omega}_1}= \varrho^{-\check{\omega}_1}s_{\alpha_1}s_{\alpha_2}$ and $\varsigma_{\check{\omega}_2}= \varrho^{-\check{\omega}_2}s_{\alpha_2}s_{\alpha_1}$. Consider the dominant weight $2\bold{\Lambda}_o$, where $\bold{\Lambda}_o$ is the fundamental weight of $\hat{L}(\mf{g})$ associated to $o\in \hat{I}$. Note that $\varsigma_{\check{\omega}_1}^{-1}=\varsigma_{\check{\omega}_2}$. We have $\varrho^{\check{\omega}_2}\varsigma_{\check{\omega}_1}^{-1}= s_{\alpha_2}s_{\alpha_1}\in W_{\mr{aff}}$. Then, $D(2,\check{\omega}_2):=D(2\bold{\Lambda}_o, \check{\omega}_2)$ is a subspace of $V(\varsigma_{\check{\omega}_1}(2\bold{\Lambda}_o))=V(2\bold{\Lambda}_1-\frac{2}{3}\delta)$ generated by an extremal weight vector of weight
$\varrho^{\check{\omega}_2}(2\bold{\Lambda}_o)=2\bold{\Lambda}_o-2\omega_2-\frac{2}{3}\delta.$ 
Thus, $D(2,\check{\omega}_2)$ has only one layer at degree $0$ as shown in Figure \ref{pic_4a}. 

Let $\tau$ be the nontrivial diagram automorphism on $\mf{g}$. Let $\hat{L}(\mf{g},\tau)$ be the twisted affine algebra of type $A_2^{(2)}$ and $\mf{g}^\tau\simeq \mr{sl}_2$. We label the vertex of the Dynkin diagram of $\mf{g}^\tau$ by $1$. Then, $\hat{I}_\tau=\{o,1\}$. Let $\beta:=\alpha_1|_{\mf{h}^\tau}$ be the restriction of the simple root $\alpha_1$ of $\mf{g}$ to $\mf{h}^\tau$. Then, $\beta$ is the simple root of $\mf{g}^\tau$. There are two level two dominant weights $\Lambda_o$ and $2\Lambda_1=\Lambda_o+\beta$ of $\hat{L}(\mf{g},\tau)$, where $\Lambda_i$ is the fundamental dominant weight of $\hat{L}(\mf{g},\tau)$ associated to $i\in \hat{I}$. Consider the fundamental coweight $\check\omega_2\in X_*(T)$. Let $\overline{\check{\omega}}_2$ denote its image in $ X_*(T)_\tau$. Then, $\iota\big(\overline{\check{\omega}}_2\big)=\beta$, see Section \ref{sec_liealgebra}. Note that $\overline{\check{\omega}}_2=\overline{\check{\alpha}}_2$ and hence $\varrho^{\pm\overline{\check{\omega}}_2}=\varrho^{\pm\overline{\check{\alpha}}_2} \in  W_{\mr{aff}}$.  By \eqref{Weyl_action}, we have 
\[\varrho^{\pm\overline{\check{\omega}}_2}(\Lambda_o)=\Lambda_o\mp2\beta-\delta;\quad \varrho^{\overline{\check{\omega}}_2}(2\Lambda_1)=2\Lambda_1-2\beta,\quad \varrho^{-\overline{\check{\omega}}_2}(2\Lambda_1)=2\Lambda_1+2\beta-2\delta.\] 
The twisted Demazure modules $D^\tau(\Lambda_o,\overline{\check{\omega}}_2)$ and $D^\tau(2\Lambda_1,-\overline{\check{\omega}}_2)$ are $6$-dimensional as shown in Figure \ref{pic_4}. \\ We now have the following examples of Theorem \ref{589379}:
\begin{enumerate}
\item We take $\Lambda=\Lambda_o$ and $\mu=\check{\omega}_2$. Then, $c=2,\lambda=0$. We get the following isomorphism of $\mf{g}^\tau$-modules,
$$D(2, \check{\omega}_2)\simeq D^\tau(\Lambda_o,\overline{\check{\omega}}_2)+D^\tau(\Lambda_o,-\overline{\check{\omega}}_2).$$
Note that $Y=\{o\}$ and $\overline{\check{\omega}}_2$ is $\breve{Y}$-dominant. By Corollary \ref{5982347}, it can be simplified as $D(2, \check{\omega}_2)\simeq D^\tau(\Lambda_o,\overline{\check{\omega}}_2).$
\item We take $\Lambda=2\Lambda_1$ and $\mu=\check{\omega}_2$. Then, $c=2,\lambda=\beta$. We get the following isomorphism of $\mf{h}^\tau$-modules,
$$D(2, \check{\omega}_2)\otimes \mb{C}_\beta\simeq D^\tau(2\Lambda_1,\overline{\check{\omega}}_2)+D^\tau(2\Lambda_1,-\overline{\check{\omega}}_2).$$
Note that $Y=\{1\}$ and $-\overline{\check{\omega}}_2$ is $\breve{Y}$-dominant. By Corollary \ref{5982347}, it can be simplified as 
$D(2, \check{\omega}_2)\otimes \C_\beta\simeq D^\tau(2\Lambda_1,-\overline{\check{\omega}}_2).$
\end{enumerate}
\begin{figure}[H]
\centering
\begin{subfigure}[b]{0.2\textwidth}
\centering
\begin{tikzpicture}[scale=0.5]
\draw[thick, -] (1.73,-1)--(-1.73,-1) -- (0,1.73)--(1.73,-1);
\foreach \p in {-1.73,0,1.73}{\fill (\p,-1) circle (3pt);}
\foreach \p in {-0.866,0.866}{\fill (\p,0.37) circle (3pt);}
\fill ((0,1.73) circle (3pt);
\draw(1.73,-1) node[anchor=north]{$\varsigma_{\check{\omega}_1}(2\bold{\Lambda}_o)$};
\draw(-1.73,-1) node[anchor=north]{$\varrho^{\check{\omega}_2}(2\bold{\Lambda}_o)$};
\end{tikzpicture}
\caption{$D(2,\check{\omega}_2)$}\label{pic_4a}
\end{subfigure}%
\begin{subfigure}[b]{0.4\textwidth}
\centering
\begin{tikzpicture}[scale=0.65]
\draw[thick, scale=1, domain=-3.2:3.2, smooth, variable=\x, black] plot ({\x}, {-0.25*\x*\x});
\fill (0,0) circle (3pt);
\draw(0,0) node[anchor=south]{$\Lambda_o$};
\foreach \p in {-2,...,2}{\fill (\p,-1) circle (3pt);}
\draw(-2,-1) node[anchor=east]{$\varrho^{\overline{\check{\omega}}_2}(\Lambda_o)$};
\foreach \p in {-2,...,2}{\fill (\p,-2) circle (0.5pt);}
\end{tikzpicture}
\caption{$D^\tau(\Lambda_o,\overline{\check{\omega}}_2)$}\label{pic_4b}
\end{subfigure}%
\begin{subfigure}[b]{0.4\textwidth}
\centering
\begin{tikzpicture}[scale=0.65]
\draw[thick, scale=1, domain=-3.3:3.3, smooth, variable=\x, black] plot ({\x}, {0.25-0.25*\x*\x});
\foreach \p in {-1,...,1}{\fill (\p,0) circle (3pt);}
\draw(1,0) node[anchor=south west]{$2\Lambda_1$};
\draw(-1,0) node[anchor=east]{$\varrho^{\overline{\check{\omega}}_2}(2\Lambda_1)$};
\foreach \p in {1,2}{\fill (\p,-1) circle (3pt);}
\foreach \p in {-2,...,0}{\fill (\p,-1) circle (0.5pt);}
\fill (3,-2) circle (3pt);
\foreach \p in {-3,...,2}{\fill (\p,-2) circle (0.5pt);}
\draw(3,-2) node[anchor=west]{$\varrho^{-\overline{\check{\omega}}_2}(2\Lambda_1)$};
\end{tikzpicture}
\caption{$D^\tau(2\Lambda_1,-\overline{\check{\omega}}_2)$}\label{pic_4c}
\end{subfigure}
\caption{}\label{pic_4}
\end{figure}

\end{document}